\documentclass[leqno]{amsart}
\usepackage{amscd,amssymb,amsmath}

\usepackage{tikz}
\usepackage{color}
\usepackage{float}
\usepackage{graphicx}
\usepackage{vntex}
\usepackage[brazil]{babel}   
\usepackage{url} 

\usepackage{amsmath,amsfonts,amsthm,amssymb}
\usepackage[all]{xy}     %para usar o XYPic
\numberwithin{equation}{section}

\theoremstyle{definition}
\newtheorem{theorem}[equation]{Teorema}

\newtheorem{definition}[equation]{Defini\c c\~ao}
\newtheorem{remark}[equation]{Observa\c c\~ao}
\newtheorem{example}[equation]{Exemplo}
\theoremstyle{definition}

\newtheorem{lem}[equation]{Lema}
\newcommand{\R}{\mathbb{R}}
\newcommand{\Sp}{\mathbb{S}}
\newcommand{\T}{\mathbb{T}}

\newcommand{\B}{\mathcal{B}}
\begin{document}

%Version (finale ?) Euler18 du 10/12/2019.

\title[Uma prova elementar da f\'ormula de Euler usando o m\'etodo de Cauchy] {Uma prova elementar da f\'ormula de Euler \\ usando o m\'etodo de Cauchy} 

\author{Jean-Paul \textsc{Brasselet} e Nguy\~\ecircumflex n Th\d{i} B\'ich Th\h{u}y}
\address{CNRS I2M e Aix-Marseille Universit\'e, Marseille, France.}
\email{jean-paul.brasselet@univ-amu.fr}

\address{UNESP, Universidade Estadual Paulista, ``J\'ulio de Mesquita Filho'', S\~ao Jos\'e do Rio Preto, Brasil}
\email{bich.thuy@unesp.br}

\begin{abstract}
O uso do m\'etodo de Cauchy para provar a conhecida f\'ormula
 de Euler \'e objeto de 
muitas controv\'ersias.
O objetivo deste artigo \'e mostrar que o m\'etodo de Cauchy se aplica 
para poliedros convexos e n\~ao somente para eles, 
mas tamb\'em para 
superf\'\i cies como o toro, o plano projetivo, a garrafa de Klein e  o toro pin\c cado.
\end{abstract}

\maketitle
 
\section*{Introdu\c c\~ao}

A f\'ormula ``de Euler'' diz que para qualquer poliedro convexo, a soma alternativa
\begin{equation}\label{carEuler}
n_0 - n_1 + n_2, 
\tag{1}
\end{equation}
vale 2, onde os n\'umeros $n_i$ est\~ao respectivamente o n\'umero de v\'ertices $n_0$, 
o n\'umero de arestas $n_1$ e o n\'umero de pol\'igonos (de dimens\~ao 2) $n_2$. 
H\'a muitas controv\'ersias sobre a paternidade da f\'ormula, 
tamb\'em sobre quem deu a primeira prova correta. 
Neste artigo, vamos investigar a hist\'oria  da f\'ormula e a sua prova pelo m\'etodo de Cauchy. 

 Na primeira se\c c\~ao do artigo, introduzimos as no\c c\~oes de poliedro,  
triangula\c c\~ao, representa\c c\~ao planar, assim que g\^enero e orienta\c c\~ao. 
Na se\c c\~ao \ref{histoire},  providenciamos precisamente 
alguns elementos sobre a hist\'oria da f\'ormula assim que 
sobre a primeira prova topol\'ogica dada por Cauchy. Alguns autores criticam a prova de Cauchy, 
dizendo que a prova necessita resultados profundos da topologia que foram provados depois 
da \'epoca de Cauchy: 
``{\it N\~ao se pode, portanto, esperar obter uma demonstra\c c\~ao elementar do Teorema 
 de Euler, com a hip\'otese de que o poliedro \'e homeomorfo a uma esfera,
  como fazem Hilbert-Cohn Vossen e Courant-Robbins}'' (Lima, \cite{Li2}). 
 Observe que a prova dada por Hilbert-Cohn Vossen \cite{HC} 
e Courant-Robbins \cite{CR}  \'e a prova de Cauchy. 

Na se\c c\~ao \ref{OTeo}, mostramos que, com uma t\'ecnica de ``elongamento'' e o uso somente 
de sub-triangula\c c\~oes,  a prova de Cauchy funciona sem usar os demais resultados.
 Mais precisamente, considere um pol\'igono triangulado no plano, com poss\'iveis identifica\c c\~oes 
 dos simplexos na sua borda,
 provamos que a diferen\c ca entre as somas (\ref{carEuler}) do pol\'igono e da sua borda 
 vale $+1$ (Teorema \ref {teo1}). A ideia de 
nossa prova \'e, a partir do buraco formado pela remo\c c\~ao de um tri\^angulo, 
estendemos o buraco por po\c cas sucessivas.
O processo \'e ilustrado pela constru\c c\~ao de uma pir\^amide adequada.
Uma conseq\"u\^encia direta do teorema \'e uma prova elementar da f\'ormula de Euler
usando apenas o m\'etodo de Cauchy. 

Na se\c c\~ao \ref{applications}, nosso prop\'osito \'e mostrar que 
a prova de Cauchy vale com ferramentas conhecidas 
do tempo de Cauchy n\~ao somente para a esfera mas tamb\'em para superf\'icies 
triangul\'aveis como o toro, o toro de genero 2 e 3, 
o plano projetivo, a garrafa de Klein e at\'e o toro pin\c cado, que \'e uma superf\'icie singular. 

Para ser completamente honesto, 
na parte de aplica\c c\~oes (se\c c\~ao \ref{applications}), exceto o caso da esfera, 
usamos tamb\'em a ideia de
``cortar'' superf\'icies que, em geral, foi introduzida por Alexander Veblen em um semin\'ario em 1915
(veja \cite{Bra}). Obviamente, pode-se perguntar por que n\~ao aplicamos o teorema \ref {teo1} e o
racioc\'inio para mostrar o resultado para todas as superf\'icies (suaves) orient\'aveis e n\~ao  
orient\'aveis.  A raz\~ao \'e
 muito simples: queremos fornecer provas, como foi poss\'ivel na \'epoca de Cauchy.
Foi somente em 1925 que T. Rad\'o \cite{Rad} provou o teorema de triangula\c c\~ao
para superf\'icies, que era mais ou menos admitido na \'epoca de Cauchy. 
O teorema da classifica\c c\~ao de
superf\'icies compactas e a representa\c c\~ao sob a forma ``normal'' foi provado 
pela primeira vez de maneira rigorosa por H.R. Brahana \cite{Bra} (1921).
\'E claro que, usando nosso teorema \ref {teo1} e a representa\c c\~ao de superf\'icies 
sob a forma normal, obtemos imediatamente o resultado para 
 a caracter\'istica de Euler-Poincar\'e de 
qualquer superf\'icie compacta.
No entanto, \'e como a serpente que morde o seu pr\'oprio rabo, ou seja ``remar contra a mar\'e''. 
Isso \'e a raz\~ao pela qual n\~ao 
apresentamos o resultado para superf\'icies em geral, mas apenas o que \'e poss\'ivel fazer 
com o m\'etodo de Cauchy em algumas superf\'icies elementares, ao fim de respeitar 
a cronologia e a hist\'oria.

Este artigo \'e uma vers\~ao em portugu\^es e ampliada do artigo dos mesmos autores, em ingl\^es: 
``An elementary proof of the Euler's formula using the Cauchy's Method''. 

O artigo \'e tamb\'em uma vers\~ao estendida da primeira parte do minicurso ministrado pelos autores 
durante a  XXVIII SEMAT Semana de Matem\'atica, 17 a 21 de Outubro de 2016 no Instituto de 
Bioci\^encias, Letras e Ci\^encias Exatas (IBILCE) - C\^ampus de S\~ao Jos\'e do Rio Preto da 
Universidade Estadual Paulista ``J\'ulio de Mesquita Filho'' (UNESP). A segunda parte, 
Teorema de Poincar\'e-Hopf, {foi publicada} na Revista Electr\^onica Paulista de Matem\'atica 
C.Q.D \cite{BT}.

O primeiro autor teve aux\'ilio financeiro da FAPESP  (processo UNESP-FAPESP no. 2015/06697-9).
\bigskip

\section{Triangula\c c\~oes - Orienta\c c\~ao} \label{rappels}

Nesta se\c c\~ao introduzimos essencialmente a no\c c\~ao de triangula\c c\~ao. Por isso, 
precisamos as defini\c c\~oes de complexo simplicial, de poliedro assim que de 
representa\c c\~ao planar. 

\subsection{Poliedros} \label{poliedros particulares}

H\'a muitas defini\c c\~oes de poliedros dependente de autores. A primeira 
discuss\~ao diz respeito \`a dimens\~ao do poliedro. Seria que um poliedro \'e um objeto s\'olido de dimens\~ao tr\^es constru\'ido  a partir de faces poligonais ou seria que \'e a parte 
superficial (de dimens\~ao 2) deste s\'olido? Neste artigo, 
chamaremos de poliedro a superf\'\i cie, isto \'e, a uni\~ao das faces poligonais. 
Por\'em, note que Euler considerou poliedros como sendo objetos s\'olidos de dimens\~ao 3
(veja \S \ref{provaEuler}). 

\smallskip

A segunda discuss\~ao diz respeito \`as condi\c c\~oes de liga\c c\~ao  dos elementos  
constituintes da parte superficial do poliedro. Neste artigo, usaremos a seguinte 
defini\c c\~ao (veja \cite{Ri}, Cap\'itulo 2).

\begin{definition}
Um {\it poliedro} $P$ \'e uma figura constru\'ida por faces poligonais tal que cada segmento \'e face comum de exatamente dois pol\'igonos de dimens\~ao 2 e cada v\'ertice \'e face comum de pelo menos tr\^es segmentos.
\end{definition}

\begin{definition} \label{convexo} 
Um poliedro em $\R^3$ \'e {\it convexo} se ele satisfaz uma das condi\c c\~oes equivalentes:
\begin{itemize}
\item Para cada face de dimens\~ao 2, o poliedro est\'a completamente localizado 
 no mesmo semi-espa\c co delimitado pelo plano que cont\'em a face. 
 \item Considerando um poliedro 
como sendo de dimens\~ao tr\^es,  para todos dois pontos dentro no poliedro, 
o segmento (reto) ligando estes dois pontos \'e contido completamente no poliedro. 
\item Considerando um poliedro 
como sendo de dimens\~ao dois, para todos dois pontos do poliedro, o segmento (reto) ligando estes dois pontos n\~ao 
encontra o poliedro em ademais pontos.
\end{itemize}
\end{definition}

Neste artigo, useramos a no\c c\~ao de triangula\c c\~ao, a qual pode ser 
definida a partir da no\c c\~ao de complexo simplicial. 

\begin{definition} \label{compsimp}
Um {\it complexo simplicial} (geom\'etrico) finito $C$ de dimens\~ao $2$ 
 \'e um conjunto finito de ``simplexos'': tri\^angulos (simplexos de dimens\~ao 2), 
arestas (simplexos de dimens\~ao 1) e v\'ertices (simplexos de dimens\~ao 0), tais que 
\begin{enumerate}
\item cada face de um simplexo \'e um elemento do complexo simplicial, 
\item  para todo par $(\sigma_i, \sigma_j)$ de simplexos, 
a interse\c c\~ao $\sigma_i \cap \sigma_j$ \'e 
seja vazia, seja uma face comum de $\sigma_i$ e $ \sigma_j$. 
\end{enumerate}
\end{definition}

Existem v\'arias maneiras para representar um complexo simplicial em um espa\c co 
euclidiano. A mais comum,  que \'e v\'alida para todos complexos simpliciais \'e a no\c c\~ao 
 de  realiza\c c\~ao geom\'etrica sob a forma de poliedro. 

\begin{example} \label{def_poliedro}
Um exemplo de poliedro \'e dado pela realiza\c c\~ao geom\'etrica de um complexo simplicial $C$, 
como sendo o subespa\c co topol\'ogico compacto 
de $\R^n$ formado pela cole\c c\~ao de todas as uni\~oes de simplexos de $C$. 
Este poliedro \'e chamado 
de {\it poliedro triangulado} ou seja de {\it realiza\c c\~ao geom\'etrica} de $C$
 e \'e denotado por $\vert C\vert$.
\end{example}

\begin{remark}
\'E sempre poss\'ivel construir o poliedro $\vert C\vert$ correspondente 
a um  complexo simplicial $C$ no espa\c co euclidiano $\R^{n_0}$ onde $n_0$  \'e o n\'umero 
de v\'ertices de $C$.
\end{remark}

\begin{definition} \label{def_triangulation}
Uma {\it triangula\c c\~ao}  $(C,\varphi)$ de uma superf\'\i cie $\mathcal{S}$ 
\'e o dado de um complexo simplicial $C$ e de um homeomorfismo $\varphi : \vert C \vert \to \mathcal{S}$. 
\end{definition}

Assim, na figura \ref{triangsphere}, temos exemplos
de triangula\c c\~oes da esfera $\Sp^2$ e do toro. 

\begin{figure}[H]
  \begin{tikzpicture}[scale=0.8]

%%%CUBE
\draw[very thick] (7,5) -- (7,6) -- (8,6) -- (8,5) -- (7,5);
\draw[very thick] (8,5) -- (8.7,5.5) -- (8.7,6.3) -- (8,6);
\draw[very thick] (8.7,6.3) -- (7.7,6.3) -- (7,6);
\draw[dotted, very thick] (7,5) -- (7.7,5.5) -- (8.7,5.5);
\draw[dotted, very thick] (7.7,5.5) -- (7.7,6.3);
\draw[very thick] (7, 5) -- (8,6);
\draw[very thick] (7, 6) -- (8.7, 6.3);
\draw[very thick] (8, 5) -- (8.7, 6.3);

\node at (8, 3) [left] {$\vert C\vert$}; 
\draw [->] (8.7, 3) --(9.7, 3);
\node at (9.5, 3.2) [left] {$\varphi$}; 
%t√©tra√®dre
\draw[very thick] (6, 0.5) -- (7,0) -- (8.7, 0.5) -- (7.5,2) -- (6, 0.5);
\draw[very thick]  (7,0) --  (7.5,2);
\draw[dotted, very thick] (6, 0.5) -- (8.7, 0.5);

\draw [->] (10, 1) cos (11.5, 2.2);
\draw [->] (10, 5.5) cos (11.5, 4.8);

\draw[very thick] (11.5, 3.5) circle (0.8cm);
\draw[dotted, thick]  (12.3cm, 3.5cm) arc (0:60:0.4) arc (60:90:1.2) arc (90:120:1.2) arc (120:180:0.4);
\draw[thick]  (12.3cm, 3.5cm) arc (0:-50:0.4) arc (-50:-90:1.02) arc (-90:-130:1.02) arc (-130:-180:0.4);
\end{tikzpicture}
\qquad\qquad\qquad
\begin{tikzpicture}[scale=0.4]

\coordinate (A) at (0,0);
\coordinate (B) at (2,0);
\coordinate (C) at (4,0);
\coordinate (D) at (6,0);
\coordinate (E) at (0,2);
\coordinate (F) at (2,2);
\coordinate (G) at (4,2);
\coordinate (H) at (6,2);
\coordinate (I) at (0,4);
\coordinate (J) at (2,4);
\coordinate (K) at (4,4);
\coordinate (L) at (6,4);
\coordinate (M) at (0,6);
\coordinate (N) at (2,6);
\coordinate (O) at (4,6);
\coordinate (P) at (6,6);

\coordinate (Q) at (7,1);
\coordinate (R) at (7,3);
\coordinate (S) at (7,5);
\coordinate (T) at (7,7);
\coordinate (U) at (1,7);
\coordinate (V) at (3,7);
\coordinate (W) at (5,7);

\coordinate (X) at (3,3);
\coordinate (Y) at (3,4);
\coordinate (Z) at (4,3);

\draw [thick](A)--(B)--(C)--(D);
\draw [thick](E)--(F)--(G)--(H);
\draw [thick](I)--(J)--(K)--(L);
\draw [thick](M)--(N)--(O)--(P);

\draw [thick](A)--(E)--(I)--(M);
\draw [thick](B)--(F)--(J)--(N);
\draw [thick](C)--(G)--(K)--(O);
\draw [thick](D)--(H)--(L)--(P);

\draw [thick](M)--(U)--(V)--(W)--(T) -- (S) -- (R)--(Q)--(D);
\draw [thick](N)--(V);
\draw [thick](O)--(W);
\draw [thick](P)--(T);
\draw [thick](L)--(S);
\draw [thick](H)--(R);

\draw [very thick](F)--(X)-- (Y);
\draw [very thick](X)-- (Z);

\draw (E)--(B);
\draw (I)--(C);
\draw (M)--(J);
\draw (N)--(H);
\draw (O)--(L);
\draw (G)--(D);
\draw (M)--(V);
\draw (N)--(W);
\draw (O)--(T);
\draw (P)--(S);
\draw (L)--(R);
\draw (H)--(Q);

\draw[very thick] (12, 3) ellipse (3cm and 1.5 cm);
\draw [very thick] (9.8, 3.1) sin (12, 2.5) cos (14.2, 3.1);
\draw [very thick] (10, 3) sin  (12, 3.5) cos (14, 3); 

\draw[very thick] (12,1.5) arc (90:0:-0.4) arc (0:-90: -0.6cm) ;
\draw[dotted, very thick] (12,1.5) arc (-90:0:+0.6) arc (0:90: 0.4cm) ;

\node at (3, -3) [left] {$\vert C\vert$}; 
\draw [->] (5, -3) --(9, 0);
\node at (8.5, -2) [left] {$\varphi$}; 

\end{tikzpicture}
\caption{Triangula\c c\~oes da esfera e do toro.}\label{triangsphere}
\end{figure}
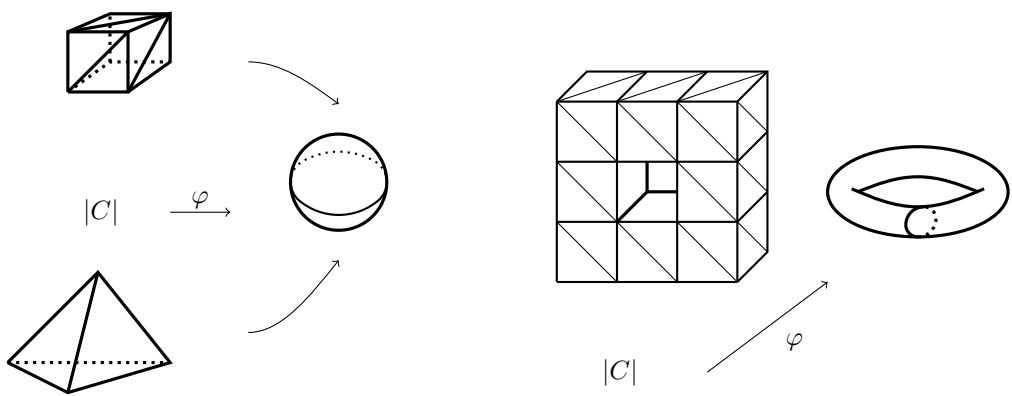
Uma maneira conveniente para representar um complexo simplicial de dimens\~ao 2 
que fornece uma triangula\c c\~ao de uma superf\'\i cie (compacta, sem borda) 
\'e a representa\c c\~ao planar,  na qual  
o complexo simplicial  \'e representado por um pol\'igono.

\begin{definition} \label{repplanar}
Uma {\it representa\c c\~ao planar} de uma superf\'\i cie (compacta, sem borda) 
$\mathcal{S}$ \'e uma tripla $(K, K_0, \varphi)$ na qual
\begin{enumerate}
\item  $K$ \'e um pol\'igono  no espa\c co euclidiano $\R^2$, que \'e triangulado pela 
uni\~ao de simplexos formando um complexo simplicial,
\item os simplexos da borda de $K$ est\~ao nomeados e orientados com identifica\c c\~oes. 
Denotamos por $K_0$ a borda de $K$ com as identifica\c c\~oes,
\item $\varphi: \vert K \vert \to \mathcal{S}$ \'e um homeomorfismo da realiza\c c\~ao geom\'etrica de $K$
(tomando em conta as identifica\c c\~oes dos simplexos na borda $K_0$) em $\mathcal{S}$. 
\end{enumerate}
\end{definition}

\begin{figure}[H]
\begin{tikzpicture}[scale=0.9]

%TRIANGLE
%grand triangle
\draw[->, very thick] (1.5,  0) -- (1.25, 0.5); 
\draw [very thick] (1.25, 0.5) -- (0.75, 1.5);
\draw [->, very thick] (0.5, 2) -- (0.75, 1.5);
\draw [->, very thick] (0.5, 2) -- (1, 2);
\draw [very thick] (1, 2) -- (2, 2);
\draw [->, very thick] (2.5, 2) -- (2, 2);
\draw [->, very thick] (2.5, 2) -- (2.25, 1.5);
\draw [very thick] (2.25, 1.5) -- (1.75, 0.5);
\draw [->, very thick] (1.5, 0) -- (1.75, 0.5);
%petit triangle
\draw[very thick] (1,1) -- (2,1) -- (1.5, 2) -- (1, 1); 
\node at (1.25, 0.5) [left] {$a$}; 
\node at (1.75, 0.5) [right] {$b$}; 
\node at (0.75, 1.5) [left] {$a$}; 
\node at (2.25, 1.5) [right] {$b$}; 
\node at (2, 2) [above] {$c$};

\node at (1, 2) [above] {$c$};

%%%%"CARRÉS"

% "ligne horizontal egf"
\draw[->, very thick] (0,  5) -- (0.5, 5); 
\draw[ ->, very thick] (0.5, 5) -- (2.5, 5); 
\draw[->, very thick] (2.5,  5) -- (3.5, 5); 
\draw[very thick] (3.5,  5) -- (4, 5); 
\node at (0.5, 5) [below] {$e$}; 
\node at (2.5, 5) [below] {$g$}; 
\node at (3.5, 5) [below] {$f$}; 

% "ligne horizontal cab"
\draw[->, very thick] (0,  6) -- (0.5, 6); 
\draw[ ->, very thick] (0.5, 6) -- (2.5, 6); 
\draw[->, very thick] (2.5,  6) -- (3.5, 6); 
\draw[very thick] (3.5,  6) -- (4, 6); 
\node at (0.5, 6) [above] {$c$}; 
\node at (2.5, 6) [above] {$a$}; 
\node at (3.5, 6) [above] {$b$}; 

% Completar "Rectangle horizontal"
\draw[->, very thick] (0, 5) -- (0, 5.5); 
\node at (0, 5.5) [left] {$d$}; 
\draw[ very thick] (0, 5.5) -- (0,6); 

\draw[->, very thick] (4, 5) -- (4, 5.5); 
\node at (4, 5.5) [right] {$d$}; 
\draw[ very thick] (4, 5.5) -- (4,6); 

\draw[very thick] (3, 5) -- (3,6); 

% "ligne vertical ec"
\draw[->, very thick] (1,  4) -- (1, 4.5); 
\draw[very thick] (1, 4.5) -- (1, 6.5); 
\draw[->, very thick] (1, 7) -- (1, 6.5); 
\draw[-<, very thick] (1, 7) -- (1.55, 7); %%
\node at (1, 4.5) [left] {$e$}; 
\node at (1, 6.5) [left] {$c$}; 
\node at (1.6, 3.8) [left] {$f$}; %%

% "ligne vertical ag"
\draw[very thick] (2,  4) -- (2, 4.5); 
\draw[->, very thick] (2, 6) -- (2, 4.5); 
\draw[->, very thick] (2, 6) -- (2, 6.5); 
\draw[very thick] (2, 7) -- (2, 6.5); 
\draw[-<, very thick] (1, 4) -- (1.55, 4); %%
\node at (2, 4.5) [left] {$g$}; 
\node at (2, 6.5) [left] {$a$}; 
\node at (1.8, 7.25) [left] {$b$}; %%

% Completar "Rectangle vertical"
\draw[very thick] (1, 4) -- (2,4) ; 
\draw[very thick] (1, 7) -- (2,7) ; 

% Diagonals-Triangulation
\draw[very thick] (0, 5) -- (2, 7);
\draw[very thick] (1, 5) -- (2, 6);
\draw[very thick] (1, 4) -- (3, 6);
\draw[very thick] (3, 5) -- (4, 6);

%%%CUBE
\draw[very thick] (7,5) -- (7,6) -- (8,6) -- (8,5) -- (7,5);
\draw[very thick] (8,5) -- (8.7,5.5) -- (8.7,6.3) -- (8,6);
\draw[very thick] (8.7,6.3) -- (7.7,6.3) -- (7,6);
\draw[dotted, very thick] (7,5) -- (7.7,5.5) -- (8.7,5.5);
\draw[dotted, very thick] (7.7,5.5) -- (7.7,6.3);
\draw[very thick] (7, 5) -- (8,6);
\draw[very thick] (7, 6) -- (8.7, 6.3);
\draw[very thick] (8, 5) -- (8.7, 6.3);

%SETAS
\draw [->] (5, 5.5) --(6, 5.5);
\draw [->] (3.5, 1) --(4.5, 1);

%tétraèdre
\draw[very thick] (6, 0.5) -- (7,0) -- (8.7, 0.5) -- (7.5,2) -- (6, 0.5);
\draw[very thick]  (7,0) --  (7.5,2);
\draw[dotted, very thick] (6, 0.5) -- (8.7, 0.5);

\draw [->] (10, 1) cos (11.5, 2.2);
\draw [->] (10, 5.5) cos (11.5, 4.8);

\draw[very thick] (11.5, 3.5) circle (0.8cm);
\draw[dotted, thick]  (12.3cm, 3.5cm) arc (0:60:0.4) arc (60:90:1.2) arc (90:120:1.2) arc (120:180:0.4);
\draw[thick]  (12.3cm, 3.5cm) arc (0:-50:0.4) arc (-50:-90:1.02) arc (-90:-130:1.02) arc (-130:-180:0.4);

\node at (2, 3) [left] {$K$}; 
\draw [->] (4.2, 3) --(5.2, 3);
\node at (8, 3) [left] {$\vert K\vert$}; 
\draw [->] (8.7, 3) --(9.7, 3);
\node at (9.5, 3.2) [left] {$\varphi$}; 

\end{tikzpicture}
\bigskip
\caption{Representa\c c\~oes planares da esfera. Os simplexos da borda de mesmo nome 
devem ser identificados, respeitando a orienta\c c\~ao.} \label{triangsphere2}
\end{figure}
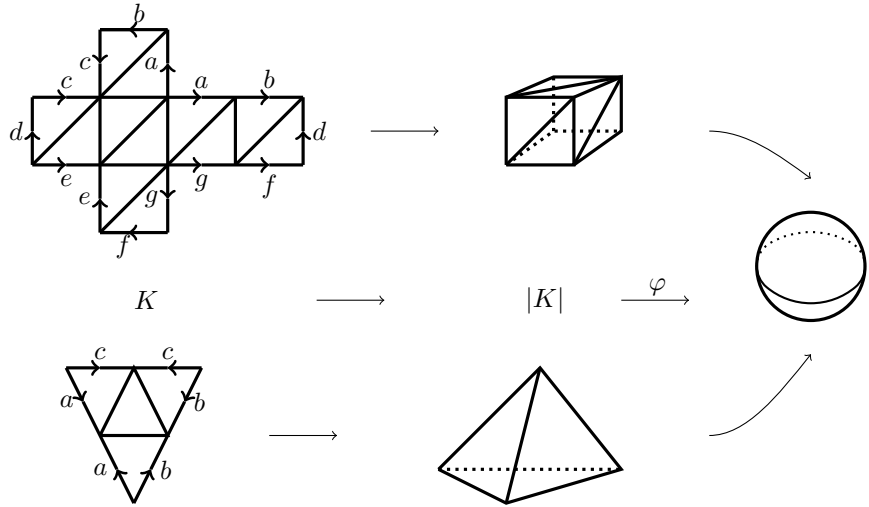

Providenciamos v\'arios exemplos de poliedros 
 e de representa\c c\~oes planares 
nas figuras \ref{triangsphere2}, \ref{o toro}, \ref{Projplanar} e \ref{Klein}.

\medskip 

A partir de agora, poderemos falar de triangula\c c\~ao de uma superf\'\i cie dada seja por 
um complexo simplicial $C$, seja por um poliedro triangulado $P$,  seja por uma representa\c c\~ao planar 
$(K, K_0, \varphi)$ tal que conforme o caso, $\vert C\vert$, $P$ ou $\vert K\vert$ seja homeomorfo 
a $\mathcal{S}$.

\begin{remark} \label{remarquelinear}
Quando realizamos a figura de uma representa\c c\~ao planar de uma superf\'\i cie, \`as vezes 
o desenho pode ter arestas curvilineares. Neste caso, podemos sempre considerar 
uma figura homeomorfa com arestas lineares.  
\end{remark} 

\begin{remark} \label{remarque}
 Em uma representa\c c\~ao planar de uma superf\'\i cie, alguns desenhos est\~ao 
proibidos. Por exemplo, se queremos representar um cone de v\'ertice $B$ (figura \ref{AA}, [1]), 
s\~ao proibidos os desenhos [2] e [3] da figura \ref{AA}: no desenho [2],  a aresta $(A,A)$ 
 representa somente um ponto $A$, ent\~ao obtemos no poliedro 
somente uma 
aresta $(A,B)$. No desenho [3], as arestas $(A,C)$ s\~ao identificados, 
assim que as arestas $(A,B)$, ent\~ao obtemos pelo poliedro 
apenas um tri\^angulo $ABC$, que n\~ao \'e  um cone como esperado. 
A representa\c c\~ao planar  [4]  da figura \ref{AA} \'e uma representa\c c\~ao correta do cone.

\begin{figure}[H] 
\begin{tikzpicture} [scale=0.9]
\draw (-3,2) ellipse (30pt and 15pt);
\draw(-3,0)node[below]{$B$} -- (-3.5,1.55) node[above]{$A$} ;
\draw (-3,0) -- (-2,1.8);
\draw(-3,0) -- (-4,1.8);
\draw (-3,-1)node[below]{$[1]$};

  \draw (1,0) node[below]{$B$}  -- (0,1) node[left]{$A$}  -- (1,2) node[above]{$A$} -- (1,0);  
  \draw (0.5,-1)node[below]{$[2]$};
  
\draw (4,0) node[below]{$B$}  -- (3,1.5) node[left]{$A$}  -- (4,2) node[above]{$C$} -- (4,0); 
\draw (4,0)  -- (4,2)  -- (5.2, 2.5) node[above]{$A$} -- (4,0); 
\draw (4.5,-1)node[below]{$[3]$};

  \draw (7, 1) node[left]{$A$}  -- (7.9,1.5) node[above]{$C$} -- 
  (8.8,2) node[above]{$D$} -- 
(9.7,2.5) node[above]{$A$} -- (9,0.5) node[below]{$B$} -- (8.8,2);  
 \draw (7,1) -- (9,0.5) -- (7.9,1.5);
 \draw (8.5,-1)node[below]{$[4]$};
 \end{tikzpicture}
\caption{Representa\c c\~ao planar de um cone. Os desenhos [2] e [3] n\~ao s\~ao 
admiss\'iveis para representar o cone.  A representa\c c\~ao [4] representa o cone [1].}\label{AA}
\end{figure}
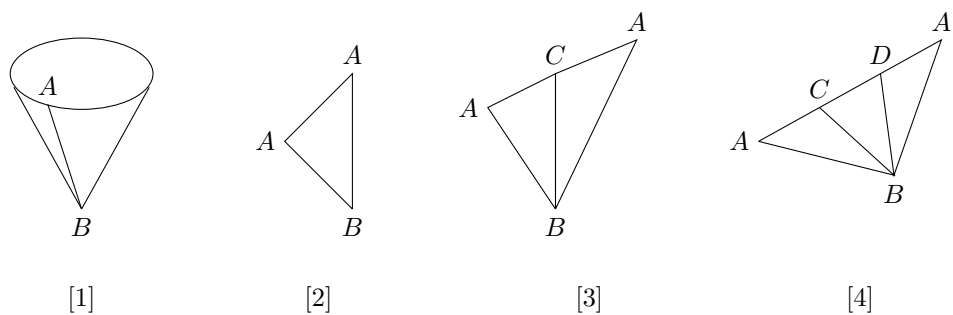
\end{remark}
\goodbreak

\subsection{Superf\'icies orient\'aveis e n\~ao orient\'aveis}

\subsubsection{Orienta{\c c}\~ao}

Consideremos um poliedro que \'e a  realiza\c c\~ao geometrica de um complexo simplicial 
de dimens\~ao dois no sentido do exemplo \ref{def_poliedro}. 
Os simplexos do poliedro considerado s\~ao v\'ertices, arestas e tri\^angulos. 
 Denotamos por $a_i$ os v\'ertices dos simplexos.
Quando definirmos uma ordem pelos v\'ertices de arestas, por exemplo a aresta $(a_0, a_1)$, 
denotaremos os respectivamente por $[a_0, a_1]$ ou $[a_1, a_0]$.  
Analogamente, $[a_0, a_1, a_2]$ ou  $[a_1, a_2, a_0]$, etc. 
 ser\~ao os tri\^angulos $(a_0, a_1, a_2)$ ordenados.

Uma ordem $[a_0, a_1, a_2]$ dos v\'ertices de um tri\^angulo define uma orienta{\c c}\~ao do 
tri\^angulo (veja figura \ref{figura12}). 
Qualquer transposi{\c c}\~ao de dois  v\'ertices define a orienta{\c c}\~ao oposta. Por exemplo, 
a transposi{\c c}\~ao de $a_1$ e $a_2$ d\'a o tri\^angulo  $[a_0, a_2, a_1]$ cuja orienta{\c c}\~ao \'e oposta da orienta{\c c}\~ao do tri\^angulo $[a_0, a_1, a_2]$. 
Denotamos $[a_0, a_2, a_1] = - [a_0, a_1, a_2]$ (veja figura \ref{figura12}). 
Uma segunda transposi{\c c}\~ao (por exemplo $[a_2, a_0, a_1]$) d\'a a volta da orienta{\c c}\~ao original.

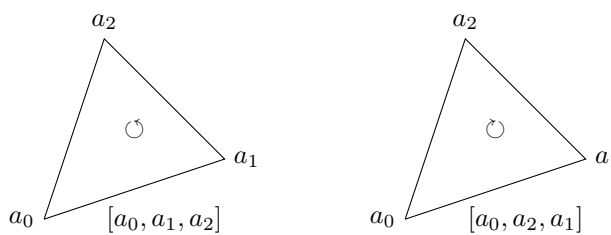
\begin{figure}[h]
\begin{tikzpicture} [scale=0.8]
  \draw (0,0) node[left]{$a_0$}   --  (1,3) node[above]{$a_2$} -- (3, 1) node[right]{$a_1$} -- (0,0);
\node at (1.5, 1.5) {$\circlearrowleft$}; 
\node at (2,0) {$[a_0, a_1, a_2 ]$};

\draw (6,0) node[left]{$a_0$}   --  (7,3) node[above]{$a_2$} -- (9, 1) node[right]{$a_1$} -- (6,0);
\node at (7.5, 1.5) {$\circlearrowright$}; 
\node at (8,0) {$[a_0, a_2, a_1 ]$};

\end{tikzpicture}
\caption{Orienta{\c c}\~oes opostas do tri\^angulo.}\label{figura12}
\end{figure}

Assim, podemos definir uma rela{\c c}\~ao de equival\^encia sobre ordens.  Duas ordens s\~ao 
equivalentes se, e somente se, podemos passar de uma \`a outra por um n\'umero par 
de transposi{\c c}\~oes. Podemos dar a defini{\c c}\~ao da orienta{\c c}\~ao.

\begin{definition}
{\rm Uma {\it orienta{\c c}\~ao} de um simplexo \'e uma classe de equival\^encia de ordens 
pela rela{\c c}\~ao acima.}
\end{definition}
 
\subsubsection{Orienta{\c c}\~ao das faces} 
Seja $\sigma = [a_0, a_1, a_2]$ um tri\^angulo orientado (pela ordem dada). 
Uma face $\tau_i$ de dimens\~ao 1 se escreve  $(a_j,a_k)$ tirando um v\'ertice $a_i$. 
A orienta{\c c}\~ao de $\sigma$ induz uma orienta{\c c}\~ao de cada tal face $(a_j,a_k)$, 
chamada de {\it orienta{\c c}\~ao induzida}, como sendo $(-1)^i [a_j,a_k]$. 
Assim a borda orientada de $\sigma = [a_0, a_1, a_2]$ pode se escrever
$$ [\widehat{a}_0, a_1, a_2] - [ a_0, \widehat{a}_1, a_2] + [a_0, a_1, \widehat{a}_2] =
[a_1, a_2] - [ a_0, a_2] + [a_0, a_1] $$
onde, $\widehat{a}_i$ significa que $a_i$ n\~ao \'e presente em $[a_0, a_1, a_2]$, para $i=0, 1, 2$.
\'E f\'acil de verificar que esta 
no\c c\~ao \'e bem definida (independente de {um n\'umero par de permuta\c c\~oes} 
dos v\'ertices de $\sigma$). 
A no\c c\~ao de orienta{\c c}\~ao induzida coincide com a no\c c\~ao intuitiva de orienta{\c c}\~ao 
da borda (veja figura \ref{figura13}). 

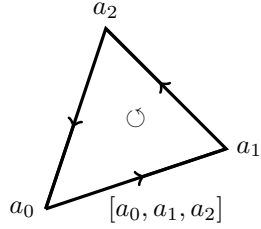
\begin{figure}[h]
\begin{tikzpicture} [scale=0.8]
  \draw [very thick] (0,0) node[left]{$a_0$}    --  (1,3) node[above]{$a_2$} -- (3, 1) node[right]{$a_1$} -- (0,0);
\node at (1.5, 1.5) {$\circlearrowleft$}; 
\node at (2,0) {$[a_0, a_1, a_2 ]$};
\draw [very thick,-<] (0,0) -- (0.5, 1.5);
\draw [very thick, -<] (1,3) -- (2,2);
\draw [very thick, -<] (3,1) -- (1.5, 0.5);
\end{tikzpicture}
\caption{Orienta{\c c}\~oes das faces do tri\^angulo.}\label{figura13}
\end{figure}

\subsubsection{Orienta{\c c}\~ao compat\'ivel}
Sejam $\sigma=(a_0, a_1, a_2)$ e $\sigma' = (b_0, a_1, a_2)$ dois tri\^angulos com 
uma face comum $(a_1, a_2)$ (figura \ref{figura14}).

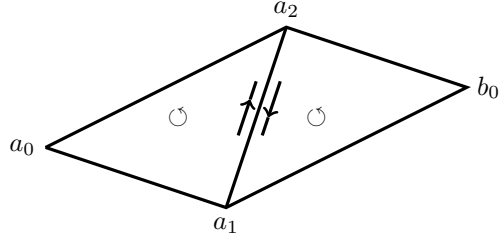
\begin{figure}[h]
\begin{tikzpicture} [scale=0.8]
  \draw[very thick] (1,1) node[left]{$a_0$}    --  (5,3) node[above]{$a_2$} -- (8, 2) node[right]{$b_0$} -- (4,0) node[below]{$a_1$}--(1,1);
\draw [very thick] (4,0) -- (5,3);
 \node at (3.2,1.5) {{\bf $\circlearrowleft$}};
 \node at (5.5,1.5) {{\bf $\circlearrowleft$}};
\draw [very thick] [->] (4.2, 1.2) -- (4.4, 1.8);
\draw  [very thick] (4.4, 1.8) --(4.5, 2.1);
\draw [very thick] [->] (4.9, 2.1) -- (4.7, 1.5);
\draw   [very thick](4.7, 1.5) -- (4.6, 1.2);

\end{tikzpicture}
\caption{Orienta{\c c}\~oes compat\'iveis.}\label{figura14}
\end{figure}

Toda orienta{\c c}\~ao de $\sigma$, por exemplo $[a_0, a_1, a_2]$, define uma orienta{\c c}\~ao 
sobre $\sigma'$, que chamaremos de {\it orienta{\c c}\~ao compat\'ivel}, tal que a face comum  
$(a_1, a_2)$ t\^em orienta{\c c}\~oes opostas como faces de $\sigma$ e $\sigma'$. 

Esta defini{\c c}\~ao corresponde \`a no{\c c}\~ao intuitiva da compatibilidade de orienta{\c c}\~oes. 

\subsubsection{Orienta{\c c}\~ao de poliedros}   
Seja $P$ um poliedro triangulado compacto, conexo de dimens\~ao 2. Sejam $\sigma$ e $\sigma'$ 
dois simplexos  de dimens\~ao 2, podemos escolher cadeias de 2-simplexos 
$$\sigma = \sigma_0, \sigma_1, \ldots, \sigma_i, \sigma_{i+1}, \ldots, \sigma_k = \sigma'$$
 tais que para todo $i = 0, \ldots, k-1$, os dois simplexos $\sigma_i$ e $\sigma_{i+1}$ t\^em uma face 
 comum de dimens\~ao 1 (veja figura \ref{figura15}).

\begin{figure}[h]
\begin{tikzpicture} 
  \draw (4,1) -- (3, 0.5) -- (2,1) -- (1,2) -- (0.5, 3.5) -- (2,4) -- (3.5, 5) -- (4.5, 4) -- (3, 3) -- 
(2, 2.5) -- (3, 1.5) -- (4,1);

\draw (3, 0.5) -- (3, 1.5);
\draw (2,1) --(3, 1.5);
\draw (2,1) --(2, 2.5);
\draw (1,2) --(2, 2.5);
\draw (0.5, 3.5) --(2, 2.5);
\draw  (2,4) --(2, 2.5);
\draw  (2,4) --((3, 3);
\draw   (3.5, 5) --((3, 3);

\node at (6.5, 0.4) {Cadeia 2};
\node at (6.5, 1) {$\cdots /  \cdots$};
\node at (6.5, 4){$\cdots /  \cdots$};
\node at (6.5, 4.6) {Cadeia 1};

 \draw (9.5, 0) -- (8.5, 1) -- (10, 1) -- (10,2) -- (11.5,2) -- (11.5, 3.5) -- (10, 4) -- (9,5) -- (8,4) -- (9,3) -- (10, 2) -- (8.5,1);

\draw (10, 2) -- (11.5, 3.5);
\draw (10,2) -- (10,4);
\draw (9,3) -- (10,4);
\draw (9,3) -- (9,5);
\draw (9.5,0) -- (10,1);
\draw (10,1) -- (11.5, 2);

\filldraw[fill=blue!20] (1,2) -- (0.5,3.5) -- (2, 2.5);

\node at (3.4,1) {$\circlearrowleft$};
\node at (2.6,1) {$\circlearrowleft$};
\node at (2.4,1.8) {$\circlearrowleft$};
\node at (1.8,1.7) {$\circlearrowleft$};
\node at (1.3,2.8) {$\circlearrowleft$};
\node at (1.35,2.5) {$\sigma = \sigma_0$};
%\node at (2.4,3.2) {$\circlearrowleft$};
\node at (2.3,3.5) {$\circlearrowleft$};
\node at (1.5,3.3) {$\circlearrowleft$};

\node at (2.7,3.7) {$\circlearrowleft$};
\node at (2.7,4.1) {$\sigma_i$};
\node at (3.5,3.8) {$\circlearrowleft$};
\node at (3.7,4.2) {$\sigma_{i+1}$};

\filldraw[fill=blue!20] (10,2) -- (11.5, 3.5) -- (10,4);

\node at (9.5,0.5) {$\circlearrowleft$};
\node at (9.5,1.2) {$\circlearrowleft$};
\node at (10.5,1.5) {$\circlearrowleft$};
\node at (11,2.5) {$\circlearrowleft$};
\node at (10.5,3) {$\circlearrowleft$};
\node at (10.5,3.5) {$\sigma'$};
\node at (9.5,3) {$\circlearrowleft$};
\node at (9.5,4) {$\circlearrowleft$};
\node at (8.5,4) {$\circlearrowleft$};
\end{tikzpicture}
\caption{Cadeias de simplexos.} \label{figura15}
\end{figure}
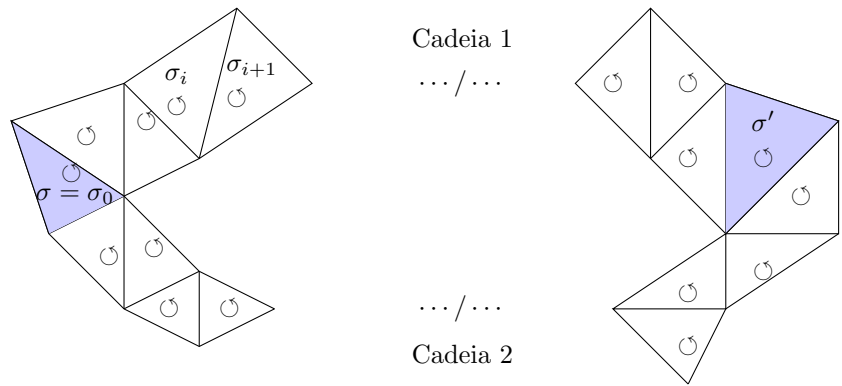

Escolhendo uma orienta{\c c}\~ao de $\sigma$ e uma cadeia de $\sigma$ a $\sigma'$, 
obtemos uma orienta{\c c}\~ao de $\sigma'$, sendo que os dois simplexos $\sigma_i$ e 
$\sigma_{i+1}$ t\^em orienta{\c c}\~oes compat\'iveis para todo $i = 0, \ldots, k-1$.

A orienta{\c c}\~ao obtida pelo tri\^angulo $\sigma'$ \'e chamada de {\it orienta{\c c}\~ao induzida 
sobre $\sigma'$} a partir da orienta{\c c}\~ao de $\sigma$ ao longo da cadeia de 2-simplexos. 

\begin{definition}
Um poliedro  compacto, conexo $K$ de dimens\~ao 2 \'e dito
{\it orient\'avel} se para todo par de tri\^angulos $(\sigma, \sigma')$, a orienta{\c c}\~ao induzida 
sobre $\sigma'$ a partir da orienta{\c c}\~ao de $\sigma$ n\~ao depende da cadeia de simplexos 
escolhida ligando $\sigma$ a $\sigma'$. 
\end{definition}

Na figura \ref{figura15}, obtemos a mesma orienta\c c\~ao induzida 
sobre $\sigma'$, seja pela cadeia 1, seja na cadeia 2. 

\subsubsection{Superf\'icies orient\'aveis (G\^enero e representa\c c\~ao planar)}

Uma superf\'icie triangulada por um poliedro $P$ \'e orient\'avel se, e somente se, o poliedro 
$P$ \'e orient\'avel.

Toda superf\'icie orient\'avel \'e homeomorfa seja \`a esfera $\Sp^2$, seja ao toro $\T$
(figura \ref{o toro}),  seja ao ``toro de g\^enero $g$'' (figura \ref{figura16}) \cite{ST}. 

O {\it g\^enero} $g$ de uma superf\'icie (orient\'avel ou n\~ao) 
\'e o n\'umero m\'aximo de curvas fechadas (homeomorfas a circunfer\^encias) disjuntas que se 
pode remover da superf\'icie sem a desconectar.  Por exemplo, o g\^enero da esfera $\Sp$ \'e 0,  
j\'a que  qualquer circunfer\^encia removida da esfera a desconecta. 
O g\^enero do toro $\T$ \'e 1, uma vez que \'e poss\'ivel remover uma circunfer\^encia 
no toro sem o desconectar. 
No entanto,  se removermos uma segunda circunfer\^encia 
qualquer, disjunta da primeira, o toro ser\'a desconectado  (veja figura \ref{figura15toro}). 

\begin{figure}[H]
\begin{tikzpicture}[scale=0.45]

%Rectangulo
 \draw [very thick] (0, 0) -- (6,0) -- (6,6) -- (0,6) -- (0,0); 

%Verticais
\draw [very thick] (2, 0) -- (2, 6); 
\draw [very thick] (4, 0) -- (4, 6); 
\draw [very thick](0, 2) -- (6, 2);
% Horizotais  
\draw [very thick] (0, 4) -- (6, 4); 
\draw [very thick] (0, 6) -- (6, 6); 

% Diagonais
\draw [very thick] (2,6) -- (0,4); 
\draw [very thick] (4,6) -- (0,2); 
\draw [very thick] (6,6) -- (0,0); 
\draw [very thick] (6,4) -- (2,0); 
\draw [very thick] (6,2) -- (4,0); 

%setas
\node at (1,0){>};
\node at (1,0)[below]{$a$};
\node at (3,0){>};
\node at (3,0)[below]{$b$};
\node at (5,0){>};
\node at (5,0)[below]{$c$};

\node at (1,6){>};
\node at (1,6)[above]{$a$};
\node at (3,6){>};
\node at (3,6)[above]{$b$};
\node at (5,6){>};
\node at (5,6)[above]{$c$};

\node at (0, 1){$\wedge$};
\node at (0,1)[left]{$f$};
\node at (0, 3){$\wedge$};
\node at (0,3)[left]{$e$};
\node at (0, 5){$\wedge$};
\node at (0,5)[left]{$d$};

\node at (6, 1){$\wedge$};
\node at (6,1)[right]{$f$};
\node at (6, 3){$\wedge$};
\node at (6,3)[right]{$e$};
\node at (6,5){$\wedge$};
\node at (6,5)[right]{$d$};

%TUBE

\draw [->] (7.5,3) -- (8.5,3); %seta

\draw[very thick] (10,2) -- (19, 2);
\draw [very thick] (10.5,4) -- (19.5, 4);
\draw [very thick] (11.1,3) -- (20.2, 3);

%CORTES

\draw [very thick]  (11.1,3) cos (12, 4);
\draw [very thick]  (12,2) cos (14.2, 3) arc (-60:0:1.2);
\draw [very thick]  (15,2) cos (17.2, 3) arc (-60:0:1.2);
\draw  [very thick] (18,2) cos (20.1, 3);

\draw[very thick] (10,2) arc (-90:0:1.2) arc (0:90: 0.8cm) ;
\draw[very thick] (10,2) arc (90:0:-0.8) arc (0:-90: -1.2cm) ;

\draw[very thick] (13,2) arc (-90:0:1.2) arc (0:90: 0.8cm) ;
\draw[dotted, very thick] (13,2) arc (90:0:-0.8) arc (0:-90: -1.2cm) ;

\draw[very thick] (16,2) arc (-90:0:1.2) arc (0:90: 0.8cm) ;
\draw[dotted, very thick] (16,2) arc (90:0:-0.8) arc (0:-90: -1.2cm) ;

\draw[very thick] (19,2) arc (-90:0:1.2) arc (0:90: 0.8cm) ;
\draw[dotted, very thick] (19,2) arc (90:0:-0.8) arc (0:-90: -1.2cm) ;

%TORO

\draw[very thick] (27, 3) ellipse (3cm and 1.5 cm);
\draw [very thick] (24.8, 3.1) sin (27, 2.5) cos (29.2, 3.1);
\draw [very thick] (25, 3) sin  (27, 3.5) cos (29, 3); 

\draw[very thick] (27,1.5) arc (90:0:-0.4) arc (0:-90: -0.6cm) ;
\draw[dotted, very thick] (27,1.5) arc (-90:0:+0.6) arc (0:90: 0.4cm) ;

\draw[->] (22, 3) -- (23,3);

\end{tikzpicture}
\caption{Uma representa\c c\~ao planar do toro }\label{o toro}
\end{figure}
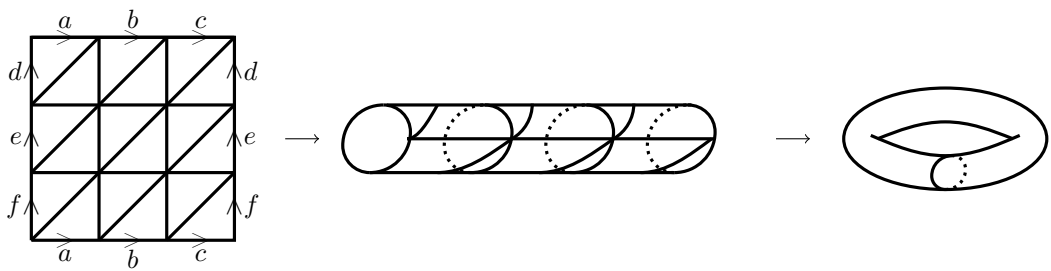

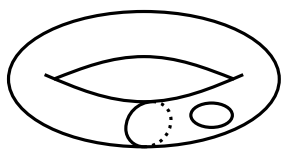
\begin{figure}[H]
\begin{tikzpicture}[scale=0.6]

\draw[very thick] (3, 1.5) ellipse (3cm and 1.5 cm);
\draw [very thick] (0.8, 1.6) sin (3, 1) cos (5.2, 1.6);
\draw [very thick] (1, 1.5) sin  (3, 2) cos (5, 1.5);); 

\draw[very thick] (4.5, 0.7) ellipse (0.46cm and 0.27 cm);

\draw[dotted, very thick] (3,0) arc (-90:0:0.6) arc (0:90: 0.4cm) ;
\draw[very thick] (3,0) arc (90:0:-0.4) arc (0:-90: -0.6cm) ;

\end{tikzpicture}
\caption{O toro \'e  de g\^enero 1 }\label{figura15toro}
\end{figure} 

\begin{figure}[H]
\begin{tikzpicture}[scale=0.9]

%le bord
\coordinate (a) at (-2.5,2);
\coordinate (a_g) at (-2.5,2.5);
\coordinate (b) at (-2.5,3);
\coordinate (b_g) at (-2.25,3.5);
\coordinate (c) at (-2,4);
\coordinate (c_g) at (-1.5,4.35);
\coordinate (d) at (-1,4.7);
\coordinate (d_g) at (-0.5,4.85);
\coordinate (e) at (0,5);
\coordinate (e_1) at (0.5,5);
\coordinate (f) at (1,5);
\coordinate (f_1) at (1.5,4.85);
\coordinate (g) at (2,4.7);
\coordinate (g_1) at (2.5,4.35);
\coordinate (h) at (3,4);
\coordinate (h_1) at (3.25,3.5);
\coordinate (i) at (3.5,3);
\coordinate (i_2) at (3.5,2.5);
\coordinate (j) at (3.5,2);
\coordinate (j_2) at (3.25,1.5);
\coordinate (k) at (3,1);
\coordinate (k_2) at (2.5,0.5);
\coordinate (l) at (2,0);
\coordinate (l_2) at (1.5,-0.15);
\coordinate (p) at (1,-0.3);

\coordinate (m) at (-2,1);
\coordinate (n) at (-1,0.2);
\coordinate (o) at (0,-0.2);

\draw (a)--(b)--(c)--(d)--(e)--(f)--(g)--(h)--(i)--(j)--(k)--(l)--(p);
\draw[dashed] (m)--(a);
\draw[dashed] (p)--(o);
\draw[thick,dotted] (m)--(n)--(o);

%les fl�ches

\draw[->>] (a)--(a_g);
\node at (a_g)[left]{$a_{g}$};
\draw[->>] (b)--(b_g);
\node at (b_g)[left]{$b_{g}$};
\draw[-<<] (c)--(c_g);
\node at (c_g)[left]{$a_{g}$};
\draw[-<<] (d)--(d_g);
\node at (d_g)[above]{$b_{g}$};
\draw[->] (e)--(e_1);
\node at (e_1)[above]{$a_1$};
\draw[->] (f)--(f_1);
\node at (f_1)[above]{$b_1$};
\draw[-<] (g)--(g_1);
\node at (g_1)[above]{$a_1$};
\draw[-<] (h)--(h_1);
\node at (h_1)[right]{$b_1$};
\draw[->>] (i)--(i_2);
\node at (i_2)[right]{$a_2$};
\draw[->>] (j)--(j_2);
\node at (j_2)[right]{$b_2$};
\draw[-<<] (k)--(k_2);
\node at (k_2)[right]{$a_2$};
\draw[-<<] (l)--(l_2);
\node at (l_2)[below]{$b_2$};

%l'int�rieur
\coordinate (A) at (-1.5,3.5);
\coordinate (B) at (-0.5,3.5);
\coordinate (C) at (0.5,4);
\coordinate (D) at (0.5,3);
\coordinate (E) at (1.5,4);
\coordinate (F) at (1.5,2.5);
\coordinate (G) at (2,3.5);
\coordinate (H) at (2.5,2.5);

\draw (c)--(A)--(d)--(B)--(e)--(C)--(f)--(E)--(g)--(G)--(h)--(H)--(i);
\draw (j)--(H)--(G)--(F)--(H);
\draw (G)--(E)--(F)--(D)--(E);
\draw (C)--(D)--(B)--(A);
\draw (B)--(C) -- (E);

\end{tikzpicture}
\caption{Uma representa{\c c}\~ao planar do toro de g\^enero $g$. De fato, a figura deve
ser subdivida {como na} observa\c c\~ao  \ref{remarque}.  Cada aresta $a_i$ e $b_i$ sendo dividida em 
{pelo} menos tr\^es arestas,
o que n\~ao fizemos  aqui para ter uma figura mais leve.}\label{figura16}
\end{figure}
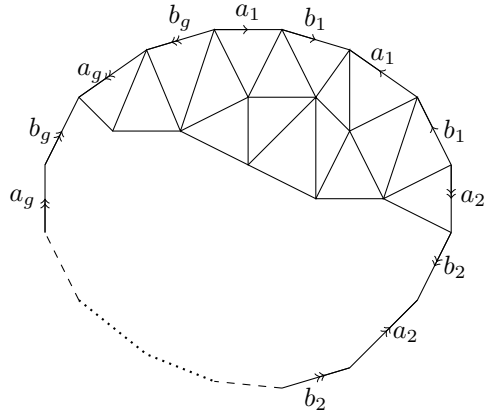

O toro de g\^enero $g$ \'e a soma conexa de um toro de g\^enero $g-1$ e do toro. 

Uma superf\'icie compacta sem borda \'e orient\'avel se, e somente 
se, ela pode ser mergulhada em $\R^3$ (veja por exemplo \cite{BW}). 

\subsubsection{Superf\'icies n\~ao orient\'aveis (G\^enero e representa\c c\~ao planar)}
Duas superf\'icies n\~ao  orient\'aveis (sem borda) 
bem conhecidas s\~ao o plano 
projetivo ${\mathbb{P}}^2$ e a garrafa de Klein. 

O plano projetivo ${\mathbb{P}}^2$ \'e o conjunto de todas as retas do espa{\c c}o euclidiano $\R^3$ 
passando pela origem. 
Uma maneira f\'acil para representar o espa{\c c}o projetivo \'e  considerar em $\R^3$ a esfera $\Sp^2$ 
de raio 1 centrada na origem. Consideramos tr\^es partes da esfera $\Sp^2$: as 
semiesferas (abertas) norte e sul e o equador (veja figura \ref{figura17}). 

\begin{figure}[H]\centering
\begin{tikzpicture} [scale=0.8]

%TRIANGULATION
\draw[thick] (2,2) circle (2cm);
\draw[dotted, thick]  (4cm, 2cm) arc (0:60:1) arc (60:90:3) arc (90:120:3) arc (120:180:1);
\draw  (4cm,2cm) arc (0:-50:1) arc (-50:-90:2.55) arc (-90:-130:2.55) arc (-130:-180:1);

\draw[thick, dashed] (2,2) --(2,4);
\draw[thick, ->] (2,4) -- (2,5);
\node at (2, 5)[above] {$z$};

\draw[thick, dashed] (2,2) -- (4,2);
\draw[thick, ->] (4,2) -- (5,2);
\node at (5,2)[right] {$y$};

\draw[very thick,red] (2,2) -- (4,4);
\node at (3.2,3.2) {$\bullet$};
\node at (3.22,3.15)[below] {$x_0$};
\node at (0.8,0.8) {$\bullet$};
%\node at (0.92,0.8)[below] {$x_S$};
\draw[very thick, red,dashed] (2,2) -- (0.6,0.6);
\draw[very thick,red] (0.6,0.6) -- (0,0);
%\node at (0.8,0.8) {$\bullet$};

\draw[thick, dashed] (2,2) -- (1.4, 0.8);
\draw[->, thick] (1.4, 0.8) -- (0.5, -1);

\node at  (0.5, -1)[below] {$x$};

\node at (2,2) {$\bullet$};
\node at (2,2)[below] {0};

\end{tikzpicture} 
\bigskip
\caption{Representa{\c c}\~ao do plano projetivo ${\mathbb{P}}^2$.}\label{figura17}
\end{figure}
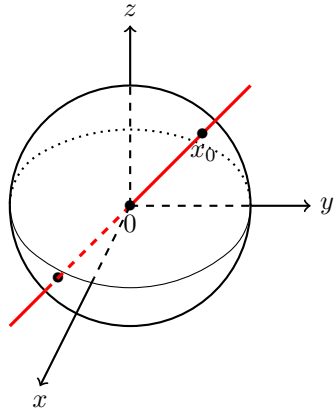

Cada reta de $\R^3$ passando pela origem que n\~ao est\'a no plano $0xy$ encontra a semiesfera norte 
 em um ponto $x_0$. Este ponto \'e um representante da reta. 

As retas contidas no plano $0xy$ encontram  o equador em dois pontos diametralmente opostos, que temos que identificar para ter somente um representante. 

Assim obtemos uma representa{\c c}\~ao do espa{\c c}o projetivo ${\mathbb{P}}^2$, 
como sendo a semiesfera norte com identifica\c c\~ao de pontos diametralmente opostos em sua borda.
Outra maneira \'e considerar a representa{\c c}\~ao planar da figura \ref{Projplanar} 
que \'e, obviamente uma representa{\c c}\~ao equivalente de ${\mathbb{P}}^2$.

\begin{figure}[H]
\begin{tikzpicture}[scale=0.7]

%Rectangulo
 \draw (0, 0) -- (6,0) -- (6,6) -- (0,6) -- (0,0); 

%Verticais
\draw (2, 0) -- (2, 6); 
\draw (4, 0) -- (4, 6); 
\draw (0, 2) -- (6, 2);
% Horizotais  
\draw (0, 4) -- (6, 4); 
\draw (0, 6) -- (6, 6); 

% Diagonais
\draw (2,6) -- (0,4); 
\draw (4,6) -- (0,2); 
\draw (6,6) -- (0,0); 
\draw (6,4) -- (2,0); 
\draw (6,2) -- (4,0); 

%setas
\node at (1,0){<};
\node at (1,0)[below]{$c$};
\node at (3,0){<};
\node at (3,0)[below]{$b$};
\node at (5,0){<};
\node at (5,0)[below]{$a$};

\node at (1,6){>};
\node at (1,6)[above]{$a$};
\node at (3,6){>};
\node at (3,6)[above]{$b$};
\node at (5,6){>};
\node at (5,6)[above]{$c$};

\node at (0, 1){$\wedge$};
\node at (0,1)[left]{$d$};
\node at (0, 3){$\wedge$};
\node at (0,3)[left]{$e$};
\node at (0, 5){$\wedge$};
\node at (0,5)[left]{$f$};

\node at (6, 1){$\vee$};
\node at (6,1)[right]{$f$};
\node at (6, 3){$\vee$};
\node at (6,3)[right]{$e$};
\node at (6,5){$\vee$};
\node at (6,5)[right]{$d$};

\node at (0, 0)[left]{$B$};
\node at (6, 0)[right]{$A$};
\node at (0, 6)[left]{$A$};
\node at (6, 6)[right]{$B$};

\end{tikzpicture}
\bigskip
\caption{Uma representa{\c c}\~ao planar do plano projetivo ${\mathbb{P}}^2$.}  \label{Projplanar} 
\end{figure}
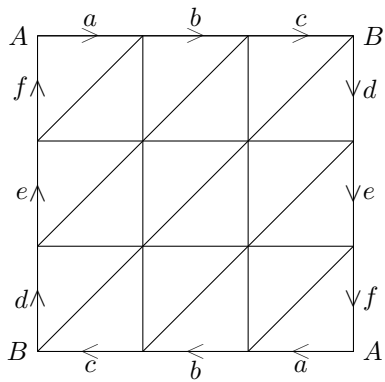

O g\^enero de ${\mathbb{P}}^2$ vale 1. 

A figura \ref{figura18} ilustra o fato de ${\mathbb{P}}^2$ n\~ao ser orient\'avel  e a figura \ref{genero_plano_projetivo} 
ilustra o g\^enero 1 do plano projetivo.
 
\vglue 1truecm 
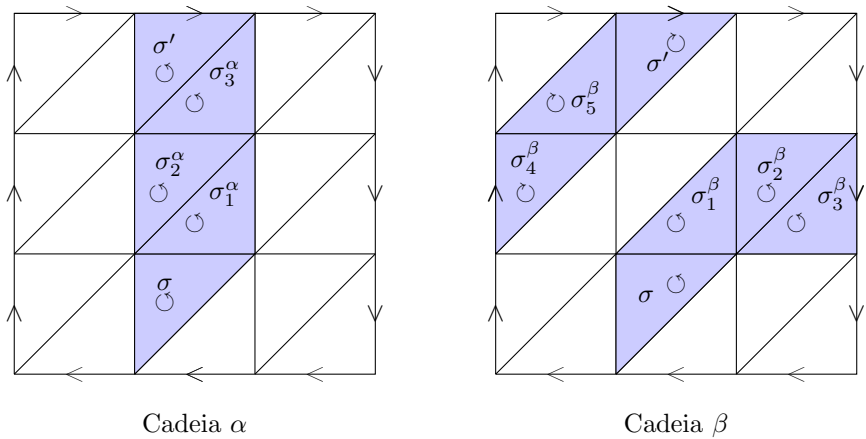
\begin{figure}[H]
\begin{tikzpicture}[scale=0.8]

%CADEIA LA BONNE !!!

%Rectangulo
 \draw (0, 0) -- (6,0) -- (6,6) -- (0,6) -- (0,0); 

%Verticais
\draw (2, 0) -- (2, 6); 
\draw (4, 0) -- (4, 6); 
\draw (0, 2) -- (6, 2);
% Horizotais  
\draw (0, 4) -- (6, 4); 
\draw (0, 6) -- (6, 6); 

% Diagonais
\draw (2,6) -- (0,4); 
\draw (4,6) -- (0,2); 
\draw (6,6) -- (0,0); 
\draw (6,4) -- (2,0); 
\draw (6,2) -- (4,0); 

%setas
\node at (1,0){<};
\node at (3,0){<};
\node at (5,0){<};

\node at (1,6){>};
\node at (3,6){>};
\node at (5,6){>};

\node at (0, 1){$\wedge$};
\node at (0, 3){$\wedge$};
\node at (0, 5){$\wedge$};

\node at (6,1){$\vee$};
\node at (6,3){$\vee$};
\node at (6,5){$\vee$};

%Orientation

\filldraw[fill=blue!20] (2,0) -- (2,2) -- (4,2) -- (2,0);
\node at (2.5,1.2) {$\circlearrowleft$};
\node at (2.5,1.5) {$\sigma$};

\filldraw[fill=blue!20] (2,2) -- (4,2) -- (4,4) -- (2,2);
\node at (3,2.5) {$\circlearrowleft$};
\node at (3.5,3) {$\sigma_1^\alpha$};

\filldraw[fill=blue!20] (2,2) -- (2,4) -- (4,4) -- (2,2);
\node at (2.4,3) {$\circlearrowleft$};
\node at (2.6,3.5) {$\sigma_2^\alpha$};

\filldraw[fill=blue!20] (2,4) -- (4,4) -- (4,6) -- (2,4);
\node at (3,4.5) {$\circlearrowleft$};
\node at (3.5,5) {$\sigma_3^\alpha$};

\filldraw[fill=blue!20] (2,4) -- (2,6) -- (4,6) -- (2,4);
\node at (2.5,5) {$\circlearrowleft$};
\node at (2.5,5.5) {$\sigma'$};

\node at (3,-0.5)[below] {Cadeia $\alpha$}; 
\node at (3,0){<};
\node at (3,6){>};

%CADEIA B

%Rectangulo
 \draw (8, 0) -- (14,0) -- (14,6) -- (8,6) -- (8,0); 

%Verticais
\draw (10, 0) -- (10, 6); 
\draw (12, 0) -- (12, 6); 
\draw (8, 2) -- (14, 2);
% Horizotais  
\draw (8, 4) -- (14, 4); 
\draw (8, 6) -- (14, 6); 

% Diagonais
\draw (10,6) -- (8,4); 
\draw (12,6) -- (8,2); 
\draw (14,6) -- (8,0); 
\draw (14,4) -- (10,0); 
\draw (14,2) -- (12,0); 

%setas
\node at (9,0){<};
\node at (11,0){<};
\node at (13,0){<};

\node at (9,6){>};
\node at (11,6){>};
\node at (13,6){>};

\node at (8, 1){$\wedge$};
\node at (8, 3){$\wedge$};
\node at (8, 5){$\wedge$};

\node at (14, 1){$\vee$};
\node at (14, 3){$\vee$};
\node at (14,5){$\vee$};

%orientation anti-horario
\filldraw[fill=blue!20] (10,0) -- (12,2) -- (10,2) -- (10,0);
\node at (11,1.5) {$\circlearrowleft$};
\node at (10.5,1.3) {$\sigma$};

\filldraw[fill=blue!20] (10,2) -- (12,2) -- (12,4) -- (10,2);
\node at (11,2.5) {$\circlearrowleft$};

\filldraw[fill=blue!20] (12,2) -- (14,2) -- (14,4) -- (12,2);
\node at (13,2.5) {$\circlearrowleft$};

\filldraw[fill=blue!20] (12,2) -- (14,4) -- (12,4) -- (12,2);
\node at (12.5,3) {$\circlearrowleft$};

%orientation horario
\filldraw[fill=blue!20] (8,2) -- (8,4) -- (10,4) -- (8,2);
\node at (8.5,3) {$\circlearrowright$};

\filldraw[fill=blue!20] (8,4) -- (10,4) -- (10,6) -- (8,4);
\node at (9,4.5) {$\circlearrowright$};

\filldraw[fill=blue!20] (10,4) -- (10,6) -- (12,6) -- (10,4);
\node at (11,5.5) {$\circlearrowright$};
\node at (10.7,5.2) {$\sigma'$};
\node at (11,-0.5)[below] {Cadeia $\beta$}; 

%Setas-escondidinhas

\node at (8, 3){$\wedge$};
\node at (14, 3){$\vee$};
\node at (11, 6){$>$};

%Cadeia $\beta$
\node at (11.5, 3){$\sigma_1^\beta$};
\node at (12.6, 3.5){$\sigma_2^\beta$};
\node at (13.6, 3){$\sigma_3^\beta$};
\node at (8.5, 3.6){$\sigma_4^\beta$};
\node at (9.5, 4.6){$\sigma_5^\beta$};
\end{tikzpicture}
\bigskip
\caption{O plano projetivo ${\mathbb{P}}^2$ n\~ao \'e orient\'avel.  
Indo de $\sigma$ at\'e $\sigma'$ 
por duas cadeias diferentes: cadeia  de 2-simplexos 
$\alpha$ : ($\sigma = \sigma^\alpha_0,  \;  \; \sigma^\alpha_1, \; \sigma^\alpha_2, \; \sigma^\alpha_3, \; \sigma^\alpha_4 = \sigma'$),  e a cadeia de 2-simplexos
$\beta$ : ($\sigma = \sigma^\beta_0,  \; \sigma^\beta_1,  \; \ldots,  \; \sigma^\beta_5,  \; \sigma^\beta_6 = \sigma'$) 
fornece  orienta{\c c}\~oes diferentes de $\sigma'$.}
\label{figura18}
\end{figure}

\vglue 1truecm

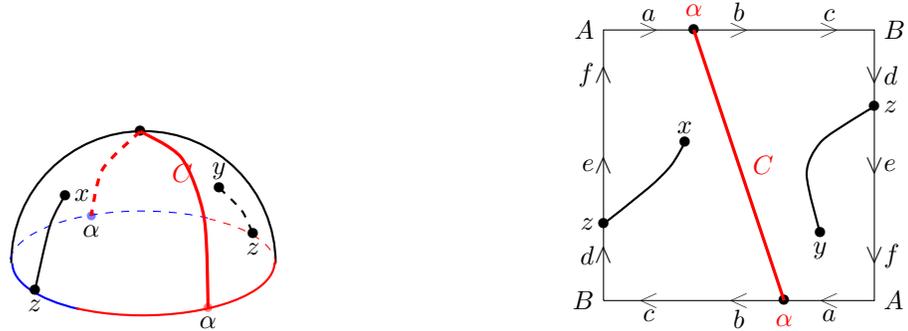
\begin{figure}[H] 
\begin{tikzpicture}  [scale=1]    
\draw [thick](0,0) arc (0:180:50pt);        
\draw[thick,blue] (-3.5,0) arc (180:240:50pt and 20pt);
\draw[thick,red] (-2.65,-0.6) arc (240:360:50pt and 20pt);

\draw [dashed,red] (0,-0.02) arc (0:60:50pt and 20pt);
\draw [dashed,blue] (-0.9,0.6) arc (60:180:50pt and 20pt);

\coordinate [label=below:$\alpha$] (A) at (-0.9,-0.59);
\coordinate [label=below:$\alpha$] (B) at (-2.45,0.63);
 \foreach \point in {A}
    \fill [red,opacity=.5] (\point) circle (1.7pt);
     \foreach \point in {B}
    \fill [blue,opacity=.5] (\point) circle (1.7pt);

%courbe C
 \node at (-1.8,1.75) {$\bullet$};   %sommet
\draw[very thick,red] plot[smooth] coordinates {(-1.8,1.75)(-1.25,1.4) (-0.95,0.6)(-0.9,-0.59)};    
\draw[very thick,dashed,red] plot[smooth] coordinates {(-1.8,1.75)(-2.3,1.25)(-2.45,0.63)};    
\node at (-1.5, 1.2)[right,red]{$C$};

%courbe de x � z
\node at (-2.8,0.9) {$\bullet$};
\node at (-2.8,0.9) [right]{$x$};
\node at (-3.2,-0.365) {$\bullet$};
\node at (-3.2,-0.365) [below]{$z$};
\draw[thick] plot[smooth] coordinates {(-2.8,0.9)(-3,0.5)(-3.2,-0.365)};     

%courbe de z � y
\node at (-0.75,1) {$\bullet$};
\node at (-0.75,1) [above]{$y$};
\node at (-0.3,0.385) {$\bullet$};
\node at (-0.3,0.385)  [below]{$z$};
\draw[thick,dashed] plot[smooth] coordinates {(-0.75,1)  (-0.45,0.7)(-0.3,0.385)};

\end{tikzpicture}
\qquad \hskip 3truecm 
\begin{tikzpicture}[scale=0.6]

%Rectangulo
 \draw (0, 0) -- (6,0) -- (6,6) -- (0,6) -- (0,0); 

%setas
\node at (1,0){<};
\node at (1,0)[below]{$c$};
\node at (3,0){<};
\node at (3,0)[below]{$b$};
\node at (5,0){<};
\node at (5,0)[below]{$a$};

\node at (1,6){>};
\node at (1,6)[above]{$a$};
\node at (3,6){>};
\node at (3,6)[above]{$b$};
\node at (5,6){>};
\node at (5,6)[above]{$c$};

\node at (0, 1){$\wedge$};
\node at (0,1)[left]{$d$};
\node at (0, 3){$\wedge$};
\node at (0,3)[left]{$e$};
\node at (0, 5){$\wedge$};
\node at (0,5)[left]{$f$};

\node at (6, 1){$\vee$};
\node at (6,1)[right]{$f$};
\node at (6, 3){$\vee$};
\node at (6,3)[right]{$e$};
\node at (6,5){$\vee$};
\node at (6,5)[right]{$d$};

\node at (0, 0)[left]{$B$};
\node at (6, 0)[right]{$A$};
\node at (0, 6)[left]{$A$};
\node at (6, 6)[right]{$B$};

\coordinate (K) at (2,6);
\coordinate (L) at (4,0);
\node at (2,6) {$\bullet$};
\node at (4,0) {$\bullet$};
\node at (2,6.1) [above]{\textcolor{red}{$\alpha$}};
\node at (4,-0.1) [below]{\textcolor{red}{$\alpha$}};

\draw [very thick, red] (K) -- (L) ;
\node at (3.1, 3)[right,red]{$C$};

\node at (0,1.7) {$\bullet$};
\node at (0,1.7) [left]{$z$};
\node at (1.8,3.5) {$\bullet$};
\node at (1.8,3.5) [above]{$x$};
\draw[thick] plot[smooth] coordinates {(0,1.7)(1,2.5) (1.5,3) (1.8,3.5)};     

\node at (6,4.3) {$\bullet$};
\node at (6,4.3) [right]{$z$};
\node at (4.8,1.5) {$\bullet$};
\node at (4.8,1.5) [below]{$y$};
\draw[thick] plot[smooth] coordinates {(4.8,1.5)(4.5,2.8) (4.8,3.5) (5.5,4)(6,4.3)};    

\end{tikzpicture}
\bigskip
\caption{0 plano projetivo ${\mathbb{P}}^2$ \'e de g\^enero $1$. A curva $C$ n\~ao desconecta o plano 
projetivo. Os pontos $x$ e $y$ est\~ao na mesma componente conexa. Qualquer 
curva suplementar fechada (homeomorfa a uma circunfer\^encia), disjunta de $C$ desconecta 
o plano projetivo.}  \label{genero_plano_projetivo}
\end{figure}

A garrafa de Klein $G$ \'e a uni\~ao de duas fitas de M\"obius ao longo de sua borda comum. 
Observamos que 
a garrafa de Klein \'e a soma conexa de dois planos projetivos, uma vez que o plano 
projetivo tirado um c\'irculo aberto \'e homeomorfo a uma fita de M\"obius.
A garrafa de Klein n\~ao \'e orient\'avel. 

No espa\c co euclidiano $\R^3$, 
n\~ao \'e poss\'ivel representar a garrafa de Klein sem ``autointerse\c c\~ao''
(figura \ref{Klein}). 

A garrafa de Klein tem g\^enero 2 e a figura \ref{genero_garrafa_Klein} ilustra este fato.

Uma superf\'icie n\~ao orient\'avel de g\^enero $g$ \'e homeomorfa \`a soma conexa de $g$ planos  
projetivos ${\mathbb{P}}^2$. 

\begin{figure}[H]
\hskip 20 pt
\scalebox{0.40}{\includegraphics{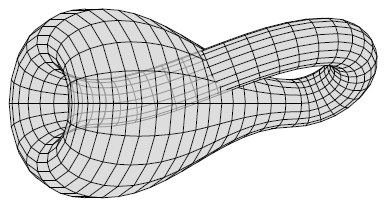}}
%\caption{A garrafa de Klein.}\label{Klein0}
%\end{figure}
\qquad\qquad 
%\begin{figure}[H]
\begin{tikzpicture}[scale=0.6]
 %Rectangulo
 \draw (0, 0) -- (6,0) -- (6,6) -- (0,6) -- (0,0); 

%Verticais
\draw (2, 0) -- (2, 6); 
\draw (4, 0) -- (4, 6); 
\draw (0, 2) -- (6, 2);
% Horizotais  
\draw (0, 4) -- (6, 4); 
\draw (0, 6) -- (6, 6); 

% Diagonais
\draw (2,6) -- (0,4); 
\draw (4,6) -- (0,2); 
\draw (6,6) -- (0,0); 
\draw (6,4) -- (2,0); 
\draw (6,2) -- (4,0); 

%setas
\node at (1,0){>};
\node at (1,0)[below]{$a$};
\node at (3,0){>};
\node at (3,0)[below]{$b$};
\node at (5,0){>};
\node at (5,0)[below]{$c$};

\node at (1,6){>};
\node at (1,6)[above]{$a$};
\node at (3,6){>};
\node at (3,6)[above]{$b$};
\node at (5,6){>};
\node at (5,6)[above]{$c$};

\node at (0, 1){$\wedge$};
\node at (0,1)[left]{$d$};
\node at (0, 3){$\wedge$};
\node at (0,3)[left]{$e$};
\node at (0, 5){$\wedge$};
\node at (0,5)[left]{$f$};

\node at (6, 1){$\vee$};
\node at (6,1)[right]{$f$};
\node at (6, 3){$\vee$};
\node at (6,3)[right]{$e$};
\node at (6,5){$\vee$};
\node at (6,5)[right]{$d$};

\node at (0, 0)[left]{$A$};
\node at (6, 0)[right]{$A$};
\node at (0, 6)[left]{$A$};
\node at (6, 6)[right]{$A$};
\end{tikzpicture}
\bigskip
\caption{A garrafa de Klein e a representa{\c c}\~ao planar da garrafa de Klein.}\label{Klein}
\end{figure}
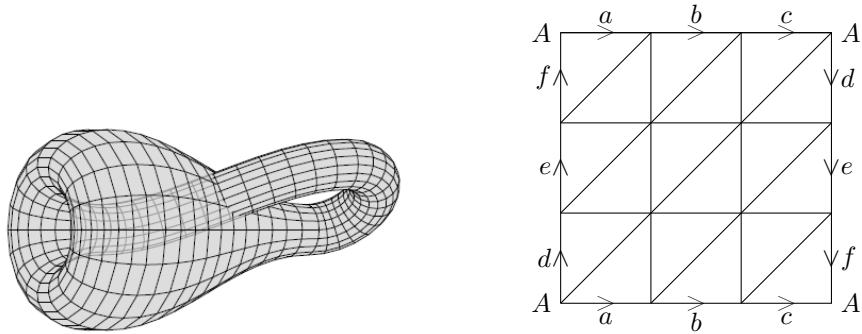

\begin{figure}[H]
\begin{tikzpicture}[scale=0.8]

%Rectangulo
 \draw (0, 0) -- (6,0) -- (6,6) -- (0,6) -- (0,0); 

%setas
\node at (1,0){>};
\node at (1,0)[below]{$a$};
\node at (3,0){>};
\node at (3,0)[below]{$b$};
\node at (5,0){>};
\node at (5,0)[below]{$c$};

\node at (1,6){>};
\node at (1,6)[above]{$a$};
\node at (3,6){>};
\node at (3,6)[above]{$b$};
\node at (5,6){>};
\node at (5,6)[above]{$c$};

\node at (0, 1){$\wedge$};
\node at (0,1)[left]{$d$};
\node at (0, 3){$\wedge$};
\node at (0,3)[left]{$e$};
\node at (0, 5){$\wedge$};
\node at (0,5)[left]{$f$};

\node at (6, 1){$\vee$};
\node at (6,1)[right]{$f$};
\node at (6, 3){$\vee$};
\node at (6,3)[right]{$e$};
\node at (6,5){$\vee$};
\node at (6,5)[right]{$d$};

\node at (0, 0)[left]{$A$};
\node at (6, 0)[right]{$A$};
\node at (0, 6)[left]{$A$};
\node at (6, 6)[right]{$A$};

%as curvas C_1 e C_2
\coordinate (K) at (0,4);
\coordinate (L) at (6,2);
\node at (0,4) {$\bullet$};
\node at (6,2) {$\bullet$};
\node at (0,4) [left]{\textcolor{red}{$\alpha$}};
\node at (6,2) [right]{\textcolor{red}{$\alpha$}};
\draw [very thick, red] (K) -- (L) ;
\node at (3.1, 3.25)[right,red]{$C_1$};

\coordinate (M) at (4,6);
\coordinate (N) at (6,3.5);
\node at (4,6) {$\bullet$};
\node at (6,3.5) {$\bullet$};
\node at (4,6) [above]{\textcolor{blue}{$\beta$}};
\node at (6,3.5) [right]{\textcolor{blue}{$\gamma$}};
\draw [very thick, blue] (M) -- (N) ;
\node at (4.85, 5)[left,blue]{$C_2$};

\coordinate (P) at (4,0);
\coordinate (Q) at (0,2.5);
\node at (4,0) {$\bullet$};
\node at (0,2.5) {$\bullet$};
\node at (4,0) [below]{\textcolor{blue}{$\beta$}};
\node at (0,2.5) [left]{\textcolor{blue}{$\gamma$}};
\draw [very thick, blue] (P) -- (Q) ;
\node at (2,1.35)[right,blue]{$C_2$};

% de $x_1$ � $x_2$
\node at (0,1.7) {$\bullet$};
\node at (0,1.7) [left]{$z_1$};
\node at (1.8,0.8) {$\bullet$};
\node at (1.8,0.8) [above]{$x_2$};
\draw[thick] plot[smooth] coordinates {(0,1.7)(0.6,1.1) (1,0.85)(1.5,0.8) (1.8,0.8)};     

\node at (6,4.3) {$\bullet$};
\node at (6,4.3) [right]{$z_1$};
\node at (5.3,5.5) {$\bullet$};
\node at (5.3,5.5) [right]{$x_1$};
\draw[thick] plot[smooth] coordinates {(5.3,5.5)(5.4,5.1) (5.6,4.7) (6,4.3)};     

% de $x_2$ � $x_3$
\node at (2.3,0) {$\bullet$};
\node at (2.3,0) [below]{$z_2$};
\draw[thick,dotted] plot[smooth] coordinates {(1.8,0.8)(2.2,0.4) (2.3,0)};    

\node at (2.3,6) {$\bullet$};
\node at (2.3,6) [above]{$z_2$};
\node at (4,4.3) {$\bullet$};
\node at (4,4.3) [right]{$x_3$};
\draw[thick,dotted] plot[smooth] coordinates {(2.3,6)(2.4,5.8) (3.1,4.5) (4,4.3)};  

% de $x_3$ � $x_4$
\node at (6,2.5) {$\bullet$};
\node at (6,2.5) [right]{$z_3$};
\draw[thick,dashed] plot[smooth] coordinates {(4,4.3)(4.6,3.15) (5.2,2.82)(6,2.5)};    

\node at (0,3.5) {$\bullet$};
\node at (0,3.5) [left]{$z_3$};
\node at (3.5,2) {$\bullet$};
\node at (3.5,2) [right]{$x_4$};
\draw[thick,dashed] plot[smooth] coordinates {(0,3.5)(1,2.7) (2,2.25)(3.5,2)};

\end{tikzpicture}
\bigskip
\caption{A garrafa de Klein  \'e de g\^enero $2$. As curvas $C_1$ e $C_2$ n\~ao desconectam a garrafa de Klein. Os pontos $x_1$, $x_2$, $x_3$ e $x_4$ est\~ao 
na mesma componente conexa. 
Qualquer curva suplementar fechada (homeomorfa a uma circunfer\^encia), disjunta de $C_1$ e $C_2$ desconecta a garrafa de Klein.}  \label{genero_garrafa_Klein}
\end{figure}
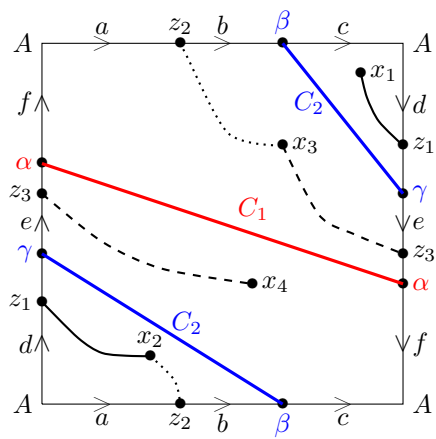

\section{Hist\'oria}\label{histoire}

\subsection{Antes de Cauchy}

O nome de ``f\'ormula de Euler'' vem do an\'uncio de Leonhard Euler em  14 de Novembro de 1750 
em uma carta a um amigo,  Goldbach do seguinte resultado.

\begin{theorem}\label{Euler's Theorem}
Seja $P$ um poliedro convexo, com  n\'umeros $n_0$ de v\'ertices, $n_1$ de arestas e $n_2$ de pol\'igonos, ent\~ao temos
\begin{equation} \label{Euler}
n_0 - n_1 + n_2 = 2.
\tag{2}
\end{equation}
\end{theorem}

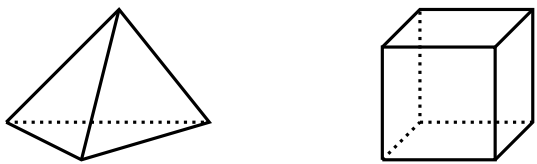
\begin{figure}[H]
\begin{tikzpicture}[scale=1]

\draw[very thick] (1, 0.5) -- (2,0) -- (3.7, 0.5) -- (2.5,2) -- (1, 0.5);
\draw[very thick]  (2,0) --  (2.5,2);
\draw[dotted, very thick] (1, 0.5) -- (3.7, 0.5);

\draw[very thick] (6,0) -- (7.5,0) -- (7.5,1.5) -- (6,1.5)-- (6,0) ;
\draw[very thick] (6,1.5) -- (6.5,2) -- (8,2) -- (8,0.5) -- (7.5,0) ;
\draw[very thick]  (7.5,1.5) -- (8,2);
\draw[dotted, very thick] (6,0) -- (6.5,0.5) -- (6.5,2);
\draw[dotted, very thick] (6.5,0.5) -- (8,0.5);
\end{tikzpicture}
\bigskip
\caption{A f\'ormula de Euler pelo tetraedro e pelo cubo.}\label{poliedros2}
\end{figure}

Na figura \ref{poliedros2}, temos dois exemplos de poliedros convexos, 
nos quais se pode verificar a rela\c c\~ao (\ref{Euler}). Pelo tetraedro, temos 
$$n_0 - n_1 + n_2 = 4 -6 + 4 =2,$$
pelo cubo temos 
$$n_0 - n_1 + n_2 = 8 - 12  + 6 =2.$$

A lenda (seria verdadeira?), providenciada em \cite{Cel}, 
 diz o seguinte:  foi um pouco por acaso que Euler descobriu a exist\^encia de uma 
rela\c c\~ao que liga os n\'umeros de faces, v\'ertices  e arestas em poliedros convexos 
(veja \cite{Cel}). 
Euler procurava classifica\c c\~oes para poliedros convexos e pensou em classific\'a-los 
de acordo com o n\'umero de 
faces. Posteriormente, para distinguir os poliedros convexos com o mesmo n\'umero de faces, ele primeiro prop\^os organiz\'a-los em poliedros do ``tipo prisma'', do ``tipo pir\^amide'' 
e associar todos os outros em uma terceira categoria. Os poliedros de quatro faces eram todos 
do ``tipo pir\^amide'' . Os poliedros com cinco faces foram divididos nos tipos ``prismas'' e ``pir\^amides''.
A partir de seis faces, as classifica\c c\~oes revelaram os tr\^es tipos de categorias: ``prismas'', 
``pir\^amides'' e outros.

Acreditando que os poliedros da terceira categoria n\~ao eram realmente todos iguais, 
ele buscou um segundo crit\'erio para classificar os poliedros com o mesmo n\'umero de faces. 
Tendo observado rapidamente que os poliedros com o mesmo n\'umero de faces nem todos tinham 
o mesmo n\'umero de v\'ertices, ele sugeriu distingui-los pelo n\'umero de v\'ertices.

Ele obteve como novas classes: Hexaedro
(poliedro de 6 faces) com 5,6,7 ou 8 v\'ertices, Heptaedro (poliedro de 7 faces) com 6 ate 10 v\'ertices, Octaedro (poliedro de 8 faces) com 6 ate 12 v\'ertices.

\begin{figure}[H]
\begin{tikzpicture}[scale=0.4]

\coordinate (A) at (0,4);
\coordinate (B) at (2,2.5);
\coordinate (C) at (2,0);
\coordinate (D) at (3,7);
\coordinate (E) at (5,4);

\draw (A)--(B)--(C)--(E)--(B)--(D)--(A)--(C);
\draw (D)--(E);
\draw[dashed] (A)--(E);
\node at (2.5,-1.5) {5 v\'ertices};

\coordinate (F) at (8,2);
\coordinate (G) at (9,6);
\coordinate (H) at (10,0);
\coordinate (J) at (13,0);
\coordinate (K) at (12,6);
\coordinate (L) at (14,2);

\draw (G)--(F)--(H)--(G)--(K)--(H)--(J)--(K)--(L)--(J);
\draw[dashed] (F)--(L);
\node at (10.5,-1.5) {6 v\'ertices};

\coordinate (M) at (17,2);
\coordinate (N) at (18,0);
\coordinate (O) at (20,0);
\coordinate (P) at (20.2,2.5);
\coordinate (Q) at (22,2);
\coordinate (R) at (19,7);

\draw (R)--(M)--(N)--(R)--(O)--(N);
\draw (O)--(Q)--(R);
\draw[dashed] (M)--(P)--(Q);
\draw[dashed] (R)--(P);
\node at (19.5,-1.5) {6 v\'ertices};

\coordinate (S) at (25,2);
\coordinate (T) at (27,0);
\coordinate (U) at (30,2);
\coordinate (V) at (28,3);
\coordinate (W) at (26,6);
\coordinate (X) at (27,5);
\coordinate (Y) at (29,7);

\draw (T)--(S)--(W)--(X)--(T)--(U)--(Y)--(X);
\draw (W)--(Y);
\draw[dashed] (S)--(V)--(U);
\draw[dashed] (V)--(Y);
\node at (27.5,-1.5) {7 v\'ertices};

\end{tikzpicture}
\end{figure}
\begin{figure}[H]
\begin{tikzpicture}[scale=0.5]

\coordinate (A) at (0,4);
\coordinate (B) at (1,2);
\coordinate (C) at (3,1);
\coordinate (D) at (2,6);
\coordinate (E) at (7,3);
\coordinate (F) at (5,7);
\coordinate (X) at (4,4);

\draw (D)--(A)--(B)--(C)--(E)--(F)--(D)--(B);
\draw (D)--(C);
\draw[dashed] (A)--(X)--(E);
\draw[dashed] (F)--(X);
\node at (2.5,-1.5) {7 v\'ertices};

\coordinate (G) at (10,3);
\coordinate (H) at (11,6);
\coordinate (K) at (12,0);
\coordinate (J) at (12.5,3);
\coordinate (L) at (15,1);
\coordinate (M) at (13,4);
\coordinate (N) at (13,7);
\coordinate (O) at (16,3);

\draw (H)--(G)--(K)--(J)--(H)--(N)--(O)--(L)--(K);
\draw(J)--(L);
\draw[dashed] (G)--(M)--(O);
\draw[dashed] (N)--(M);
\node at (10.5,-1.5) {8 v\'ertices};

\coordinate (P) at (19,5);
\coordinate (Q) at (21,6.5);
\coordinate (R) at (24,7);
\coordinate (S) at (23,5);
\coordinate (T) at (19,2);
\coordinate (U) at (21,0.5);
\coordinate (V) at (23,3);
\coordinate (W) at (25,2);

\draw (Q)--(P)--(T)--(U)--(Q)--(R)--(W)--(U);
\draw[dashed] (P)--(S)--(V)--(T);
\draw[dashed] (V)--(W);
\draw[dashed] (R)--(S);
\node at (19.5,-1.5) {8 v\'ertices};

\end{tikzpicture}
\bigskip
\caption{Hexaedros (6 faces) com 5,6,7 e 8 v\'ertices }\label{hexaedros}
\end{figure}
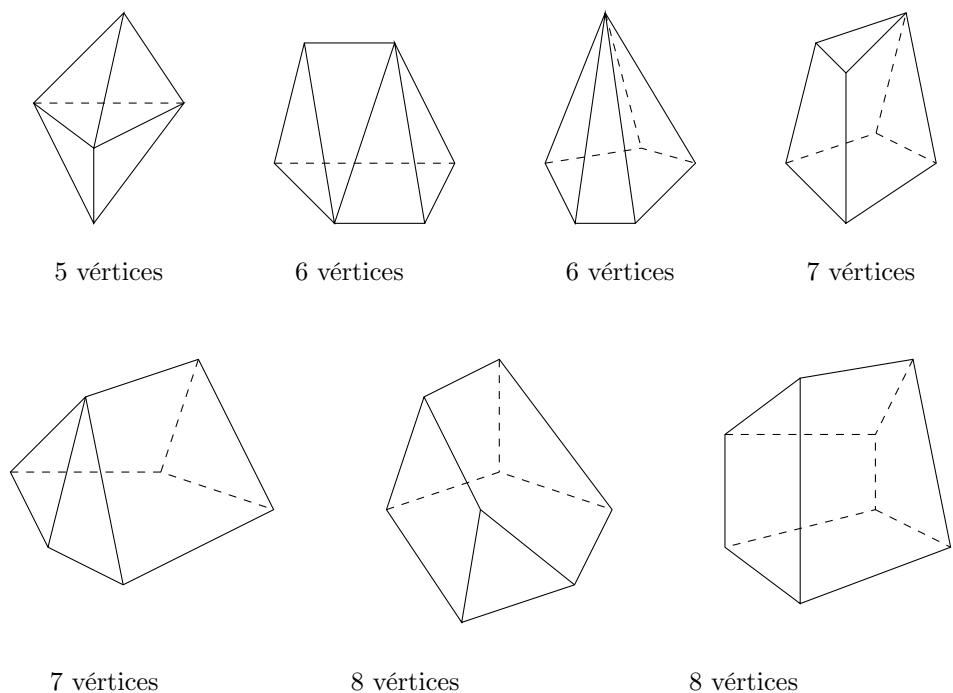

Para essas novas classifica\c c\~oes, Euler tamb\'em descobriu que alguns poliedros que tinham 
o mesmo n\'umero de faces e o mesmo n\'umero de v\'ertices n\~ao eram semelhantes. 
Por exemplo, estes s\~ao poliedros com 6, 7 ou 8 v\'ertices para hexaedros (figura 
\ref{hexaedros}). Naturalmente, Euler pensou no n\'umero 
de arestas como um terceiro crit\'erio para separar poliedros com o mesmo n\'umero de faces e 
o mesmo n\'umero de v\'ertices.

Para poliedros com seis faces e oito v\'ertices, ele ficou surpreso ao perceber que o n\'umero 
de arestas era o mesmo (seis faces, oito v\'ertices e doze arestas).

Sendo incapaz de distingui-los e de acreditar em uma anomalia das seis faces, ele imaginou usar o crit\'erio 
``arestas''
 para subclassificar os poliedros com sete faces e dez v\'ertices. Novamente, a mesma caracter\'istica 
 apareceu. Todos tinham o mesmo n\'umero de arestas (sete faces, dez v\'ertices, quinze arestas).
 Observando que era o mesmo para as outras fam\'ilias com o mesmo n\'umero de faces e v\'ertices, ele conjeturou a exist\^encia de uma rela\c c\~ao que liga o n\'umero de faces, o n\'umero de v\'ertices  e o n\'umero de arestas. -- Fim da lenda. 
\medskip

Temos duas perguntas. 
Seria que Euler era o primeiro a enunciar a f\'ormula? 
Quem providenciou a primeira prova da f\'ormula? 
De fato, h\'a v\'arias perguntas e pol\^emicas sobre a f\'ormula  de Euler.

\medskip 

Vamos \`a primeira pergunta: 
 {\it Quem foi o primeiro a expressar a f\'ormula? } 

\medskip 
Alguns autores (veja \cite{Eve}, \S 3.12; \cite{Lie}, p. 90) dizem que \'e poss\'ivel que 
Arquimedes ($\sim$ 287 AC, $\sim$ 212 AC)  j\'a conhecia a f\'ormula. O que se sabe \'e que 
Pit\'agoras ($\sim$ 580 AC, $\sim$ 495 AC) j\'a conheceu conhecia 
os tr\^es ``primeiros'' poliedros regulares e 
Plat\~ao ($\sim$ 428 AC, $\sim$ 348 AC) explicou a harmonia do mundo pela exist\^encia dos cinco 
poliedros regulares convexos: tetraedro, cubo, octaedro, dodecaedro e icosaedro. 

Alguns autores dizem que a f\'ormula 
j\'a era conhecida por Descartes (1596 -- 1650).
 De fato, Descartes, no manucrito {\it ``De solidorum elementis''} \cite{De}, 
tinha enunciado uma f\'ormula pr\'oxima que se mostra no seguinte teorema.

\begin{theorem}\label{Descartes}
Se tomarmos o \^angulo reto como a unidade,  ent\~ao a soma dos \^angulos de todas as faces 
de um poliedro convexo ser\'a igual a quatro vezes  do n\'umero de v\'ertices  ($n_0$) menos 2, {isto \'e}, 
$4(n_0 -2)$.
\end{theorem}

\begin{proof} [\bf {Justificativa do fato ``a f\'ormula (\ref{Euler}) \'e equivalente ao teorema \ref{Descartes}''}] 
Por um lado, o teorema de Descartes diz que a soma dos \^angulos de todas as faces vale $2 (n_0 -2) \pi$. 
 
Por outro lado, denotamos por $i=1, \ldots, n_2$ 
 as faces do poliedro convexo. {Para cada face $i$, o n\'umero dos v\'ertices, que tamb\'em \'e o n\'umero das arestas da face, ser\'a denotado por $k_i$}. Usamos, para cada face, a  seguinte propriedade de pol\'igonos convexos. Em um pol\'igono convexo com $k_i$ 
 lados, a soma dos \^angulos vale $(k_i -2)\pi$. Ent\~ao a soma de todos  os 
\^angulos de todas as faces do poliedro vale $\sum_{i=1}^{n_2} (k_i -2)\pi$, ou seja $(2n_1 - 2n_2) \pi$, 
observando que $\sum_{i=1}^{n_2} k_i = 2n_1$, uma vez que cada aresta do poliedro aparece em duas faces 
do poliedro. Assim, obtemos a f\'ormula (\ref{Euler}). 
\end{proof}

Descartes n\~ao publicou o manuscrito no tempo dele. A vers\~ao original do manuscrito  desapareceu, mas uma c\'opia foi redescoberta em 1860 com p\'apeis deixados 
por Leibnitz (1646-1716). Esta c\'opia sofreu complicados contratempos, em particular uma 
imers\~ao no rio Seine em Paris (veja  \cite{Fo}, \cite{dJ1}). 
Alguns autores dizem que Descartes descobriu a f\'ormula (\ref{Euler}) como uma aplica\c c\~ao 
de seu teorema \ref{Descartes}. 
Por esta raz\~ao, \`as vezes
a f\'ormula  (\ref{Euler}) \'e chamada de ``{\it f\'ormula de Descartes-Euler}''.
 Outros autores, por exemplo, Malkevich \cite{Mal1}, 
afirmou que ``{\it 
Though Descartes did discover facts about 3-dimensional polyhedra that would have enabled him to deduce Euler's f\'ormula, he did not take this extra step. With hindsight it is often difficult to see how a talented mathematician of an earlier era did not make a step forward that with today's insights seems natural, however, it often happens.}'' 
 Mas sabendo que alguns documentos de Descartes  desapareceram, 
ningu\'em pode decidir se Descartes conheceu ou n\~ao a f\'ormula. 
Ent\~ao, a resposta para a primeira pergunta tem sido desconhecida at\'e hoje em dia.

\medskip 

Discutimos agora a segunda pergunta: {\it Quem foi o primeiro 
que providenciou uma prova correta da f\'ormula   (\ref{Euler}) ?}

\medskip 

\subsection{A ``prova'' de Euler}\label{provaEuler}

Dois anos depois de ter escrito
a f\'ormula, Euler publicou uma prova \cite{Eu2,Eu3}, 
que consiste em eliminar um de cada vez cada v\'ertice do poliedro, removendo 
a pir\^amide da qual \'e o v\'ertice. A cada etapa, a soma $n_0-n_1+n_2$ fica inalterada. 

Aqui damos o exemplo do cubo, copiado do livro de Richeson \cite{Ri}. 
Na figura \ref{dEuler} (b), se elimina o v\'ertice $A$ assim como a pir\^amide 
(em branco) da qual $A$ \'e o v\'ertice. Na figura \ref{dEuler} (c), temos o mesmo processo para 
eliminar o v\'ertice $B$, e assim por diante, at\'e chegar a um tetraedro. 
Para o tetraedro temos $n_0-n_1+n_2 = 2$, 
ent\~ao  obtemos a f\'ormula. 

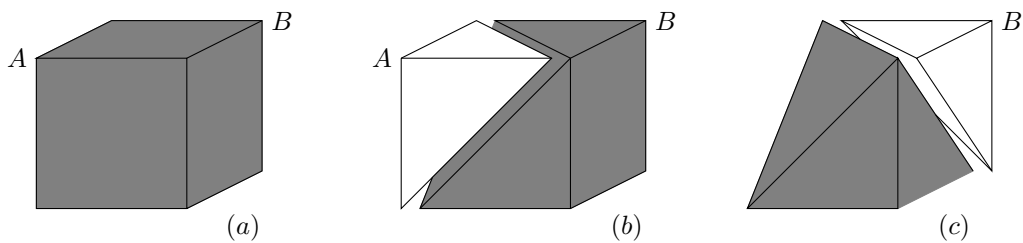
\begin{figure} [h!] % Preuve d'Euler
\begin{tikzpicture}  [scale= 0.5] % premier cube

\coordinate (A) at (0,0);
\coordinate (B) at (4,0);
\coordinate  (C) at (6,1);
\coordinate  (D) at (0,4);
\coordinate (E) at (4,4);
\coordinate (F) at (2,5);
\coordinate (G) at (6,5);

\filldraw [gray] (A) -- (D) -- (F) -- (G) -- (C) -- (B) -- (A);
\draw (D)--(A) -- (B) -- (C)--(G)--(F) -- (D) -- (E) --(B);
\draw (E) -- (G) ;
\node at (0,4)[left]{$A$};
\node at (6,5)[right]{$B$};
\node at (5.5,-0.5){$(a)$};
\end{tikzpicture}  
\qquad 
 \begin{tikzpicture}  [scale= 0.5] %deuxi�me cube

\coordinate (A) at (0,0);
\coordinate (B) at (4,0);
\coordinate  (C) at (6,1);
\coordinate  (D) at (0,4);
\coordinate (E) at (4,4);
\coordinate (F) at (2,5);
\coordinate (G) at (6,5);
\coordinate  (H) at (-0.5,0);
\coordinate (J) at (-0.5,4);
\coordinate (K) at (1.5,5);
\coordinate (L) at (3.5,4);

\coordinate (M) at (intersection of A--F and H--L);
\coordinate (N) at (intersection of K--L and A--F);

\filldraw [gray] (A) -- (M) -- (L) -- (N) -- (F) -- (G) -- (C) -- (B) -- (A);
\draw (M)--(A) -- (B) -- (C)--(G)--(F) --  (E) --(B);
\draw (G) -- (E) -- (A) ;
\draw (J) -- (H) -- (L) -- (K) -- (J) ;
\draw (J) --   (L)  ;
\node at (-0.5,4)[left]{$A$};
\node at (6,5)[right]{$B$};
\node at (5.5,-0.5){$(b)$};
\end{tikzpicture}  
\qquad
\begin{tikzpicture}  [scale= 0.5] % troisi�me cube

\coordinate (A) at (0,0);
\coordinate (B) at (4,0);
\coordinate  (C) at (6,1);
\coordinate  (D) at (0,4);
\coordinate (E) at (4,4);
\coordinate (F) at (2,5);
\coordinate (G) at (6,5);
\coordinate (P) at (2.5,5);
\coordinate (Q) at (4.5,4);
\coordinate  (R) at (6.5,5);
\coordinate  (S) at (6.5,1);
\coordinate (T) at (intersection of F--E and P--S);
\coordinate (U) at (intersection of C--E and P--S);

\filldraw [gray] (A) -- (F) -- (E) -- (C) -- (B) -- (A);

\draw (A)--(F) --  (E) --(A) -- (B)-- (E)-- (C) ;
\draw (R)--(P) -- (Q) -- (R) -- (S)--(Q);
\draw(P)--(T);
\draw (U)--(S);
\node at (6.5,5)[right]{$B$};
\node at (5.5,-0.5){$(c)$};

\end{tikzpicture}  

\caption{ A prova de Euler. Elimina\c c\~ao successiva de um v\'ertice 
assim como da pir\^amide da qual ele \'e o v\'ertice.}
\label{dEuler}
\end{figure}

A prova de Euler est\'a bem reproduzida por exemplo em \cite{Leb}  e \cite{Ri}. 
Mas esta prova n\~ao foi correta. 
No seu livro \cite{Ri}, Richeson providencia uma descri\c c\~ao clara dos  problemas na prova.
Ele escreve que estes problemas foram resolvidos 
por Samelson  e por Francese e Richardson (\cite{Sam,FR}). 

A primeira prova correta da f\'ormula (\ref{Euler}) foi 
dada por Legendre \cite{Leg} na primeira edi\c c\~ao de seu livro de 
{\it \'El\'ements de G\'eometrie } (1794) (veja 
por exemplo \cite{Leb} ou \cite{Ri} Cap\'itulo 10 para uma 
apresenta\c c\~ao da prova de Legendre).  
O argumento de Legendre era geom\'etrico, da mesma maneira da 
 prova de Descartes do teorema \ref{Descartes}. A \'unica diferen\c ca entre o 
racioc\'\i nio de Descartes e de Legendre \'e que Descartes 
raciocinou sobre a representa\c c\~ao esf\'erica 
do poliedro (poliedro polar),  enquanto Legendre raciocinou sobre o poliedro mesmo. 
A passagem do poliedro $K$ 
(Legendre) ao poliedro polar de $K$ (Descartes) faz uma permuta\c c\~ao 
de $n_0$ e $n_2$, sem modificar $n_1$. 
 Por esta raz\~ao, alguns autores dizem que a prova da f\'ormula (\ref{Euler}) 
deveria ser chamada de ``Descartes-Legendre''.

\subsection{\'Epoca de Cauchy. M\'etodo de Cauchy - A primeira prova combinat\'oria e topol\'ogica}\label{methode_Cauchy}

%Pouco depois, 

Em fevereiro de 1811, aos 22 anos de idade, Cauchy, que j\'a era   
 engenheiro {\it des Ponts et Chauss\'ees}, ministrou uma palestra intitulada {\it Recherches sur les poly\`edres}  na {\it \'Ecole Polytechnique}, em Paris. Esta palestra foi publicada em 1813 
 no {\it Journal de l'\'Ecole Polytechnique} \cite{Ca1}. 
Cauchy  publicou uma prova
combinat\'oria e  topol\'ogica da f\'ormula de Euler (\ref{Euler}). 
A bonita prova de Cauchy \'e reproduzida
em  v\'arios livros (veja, em particular, \cite{Ri}, Cap\'itulo 12 para uma apresenta\c c\~ao com coment\'arios).  

\begin{proof} \label{proof_Cauchy}
A primeira etapa na prova de Cauchy \'e da representa\c c\~ao planar. 
Parece que Cauchy foi o primeiro a usar esta ideia. 

Sendo um poliedro convexo $P$ 
homeomorfa \`a esfera, escolhemos uma face ${\mathcal F}$ de dimens\~ao 2 do poliedro. 
Removemos esta face e constru\'imos  
 a representa\c c\~ao planar $K$ associada \`a face escolhida.  
Lakatos \cite{Lak} explica a constru\c c\~ao da seguinte maneira.   
Coloque uma c\^amera acima da face removida. 
A representa\c c\~ao planar aparecer\'a na fotografia. 
 Note que esta ideia de representa\c c\~ao planar \'e parecida da proje\c c\~ao estereogr\'afica 
 (veja se\c c\~ao \ref{caso_esfera}).

\begin{figure} [H]
\begin{tikzpicture}  [scale= 0.7] % le gros cube

\coordinate (A) at (0,0);
\coordinate (T) at (4,0);
\coordinate  (C) at (2,1);
\coordinate  (D) at (6,1);
\coordinate (E) at (0,3);
\coordinate (F) at (4, 3);
\coordinate  (G) at (2,4);
\coordinate  (H) at (6,4);
\coordinate (J) at (6.3,1.5);

\node at (A)[left]{$A$};
\node at (T)[below]{$B$};
\node at (C)[below]{$C$};
\node at (D)[right]{$D$};
\node at (E)[left]{$E$};
\node at (F)[above]{$F$};
\node at (G)[above]{$G$};
\node at (H)[above]{$H$};
\node at (J)[right]{$J$};

\node at (3,3.5)[red]{${\mathcal F}$};

\draw [very thick] (E) -- (A) -- (T) -- (F) -- (E) -- (G) -- (H) ; 
\draw [very thick] (H) -- (J) --  (F)--(H) ;
\draw [very thick] (T) -- (J);
\draw [very thick] (J) -- (D) -- (T);

\draw (A) -- (C) --(G);
\draw (C) -- (D) --(H);

\end{tikzpicture}   
\caption{O poliedro $P$}\label{poliedre-Cauchy}
\end{figure}
\qquad
\begin{figure}[H]
\begin{tikzpicture}  [scale= 0.7] % le gros cube et les projections

\coordinate (A) at (0,0);
\coordinate (T) at (4,0);
\coordinate  (C) at (2,1);
\coordinate  (D) at (6,1);
\coordinate (E) at (0,3);
\coordinate (F) at (4, 3);
\coordinate  (G) at (2,4);
\coordinate  (H) at (6,4);
\coordinate (J) at (6.3,1.5);
\coordinate (J2) at (6.3,1.55);

\coordinate  (K) at (3,5.5);
\coordinate (L) at (-4.5,-0.75);
\coordinate (M) at (5.5,-0.75);
\coordinate  (N) at (0.5,1.75);
\coordinate  (O) at (10.5,1.75);
\coordinate (P) at (7.1,0.63);

\coordinate (R) at (intersection of N--L and A--E);
\coordinate (S) at (intersection of H--J and N--O);

\node at (A)[left]{$A$};
\node at (T)[below]{$B$};
\node at (C)[below]{$C$};
\node at (D)[below]{$D$};
\node at (E)[left]{$E$};
\node at (F)[above]{$F$};
\node at (G)[above]{$G$};
\node at (H)[above]{$H$};
\node at (J2)[right]{$J$};

\node at (L)[left]{$E'$};
\node at (M)[right]{$F'$};
\node at (0.4,1.75)[above]{$G'$};
\node at (O)[right]{$H'$};

\draw [very thick] (E) -- (A) -- (T) -- (F) -- (E) -- (G) -- (H) ; 
\draw [very thick] (H) -- (J) --  (F)--(H) ;
\draw [very thick] (T) -- (J);
\draw [very thick] (J) -- (D) -- (T);

\draw (A) -- (C) --(G);
\draw (C) -- (D) --(H);

\draw [blue](K)--(L) ;
\draw[blue] (K)--(G);
\draw [blue](K) --(O) ;
\draw [blue](K) --(M) ;
\draw[dashed,blue] (N) --(G);
\draw [blue](K) -- (P) ;

\draw [red] (L)-- (R);
\draw [red] (S)-- (O);
\draw [dashed,red] (R)--(N)-- (S);
\draw[dashed,red] (N) --(C);
\draw [red] (L) -- (M)-- (O);
\draw [red] (L) --(A) ;
\draw[dashed,red]  (N) --(S) ;
\draw [red] (T) -- (M);
\draw [red](O) -- (D) ;
\draw [red](D) -- (P) --(T);
\draw [red](M) -- (P) --(O);

\end{tikzpicture}   
\bigskip
\caption {Representa\c c\~ao planar seguido Cauchy}\label{rep-Cauchy}
\end{figure}
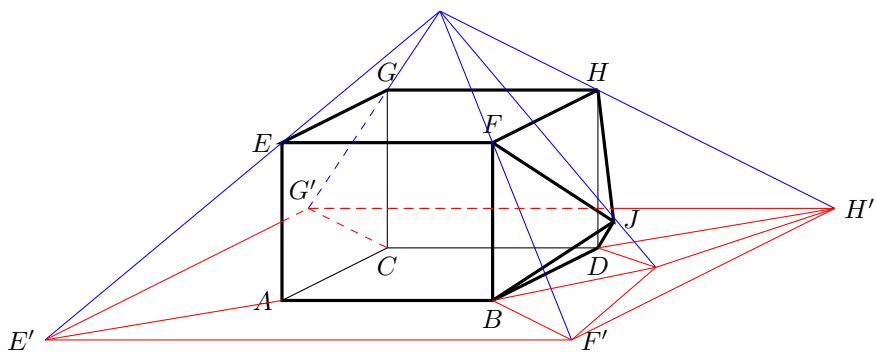

O buraco $\B$ formado pela retirada da face ${\mathcal F}$
aparece ao redor da representa\c c\~ao planar (veja figura \ref{poligone-Cauchy}, 
 onde o buraco esta em azul).

\begin{figure}[H]
\begin{tikzpicture}  [scale= 0.8] %polygone planaire
\coordinate (A) at (1.7,1.2);
\coordinate (T) at (4.2,1.2);
\coordinate  (C) at (1.7,3.3);
\coordinate  (D) at (4.2,3.3);
\coordinate (E) at (0,0);
\coordinate (F) at (6,0);
\coordinate  (G) at (0,4.5);
\coordinate  (H) at (6,4.5);
\coordinate (J) at (5,2.25);

\coordinate (L) at (0,-1);
\coordinate (M) at (-1,0);
\coordinate  (N) at (-1,4.5);
\coordinate  (O) at (0,5.5);
\coordinate (P) at (6,5.5);
\coordinate (Q) at (7,4.5);
\coordinate  (R) at (7,0);
\coordinate  (S) at (6,-1);

\fill[fill=blue!20] (L) -- (O) -- (N) --(M);
\fill[fill=blue!20] (P) -- (Q) -- (R) -- (S);
\fill[fill=blue!20] (O) -- (G) -- (H)-- (P);
\fill[fill=blue!20] (L) -- (E) -- (F)-- (S);

\node at (A)[below]{$A$};
\node at (T)[below]{$B$};
\node at (C)[above]{$C$};
\node at (D)[above]{$D$};
\node at (E)[below]{$E'$};
\node at (F)[below]{$F'$};
\node at (G)[above]{$G'$};
\node at (H)[above]{$H'$};
\node at (J)[right]{$J'$};

\node at (1.7,-0.55)[right]{$\B$};

\draw [very thick] (E) -- (A) -- (T) -- (F) -- (E) -- (G) -- (H) ; 
\draw [very thick] (H) -- (J) --  (F)--(H) ;
\draw [very thick] (T) -- (J);
\draw [very thick] (J) -- (D) -- (T);

\draw [very thick] (A) -- (C) --(G);
\draw [very thick] (C) -- (D) --(H);

\end{tikzpicture}  
\bigskip
\caption{O pol\'igono $K$, resultado do processo de representa\c c\~ao planar. }\label{poligone-Cauchy}
\end{figure}
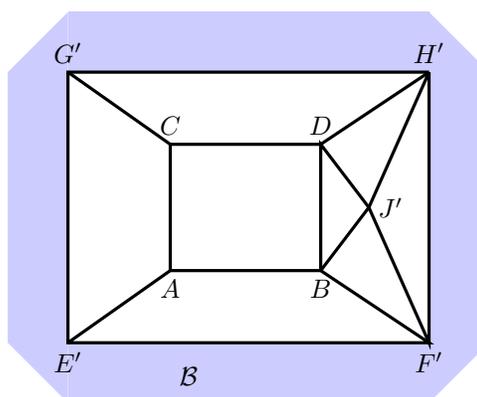

Como j\'a tiramos a face ${\mathcal F}$ que \'e 
um pol\'igono (aberto), ent\~ao j\'a temos $1$ na conta da soma de $n_0 -n_1 + n_2$. 
A etapa seguinte, na prova de Cauchy \'e definir uma triangula\c c\~ao de $K$ 
pela subdivis\~ao de todos os pol\'igonos. Note que, depois ter triangulado, 
a some alterada $n_0 -n_1 + n_2$ n\~ao muda. 

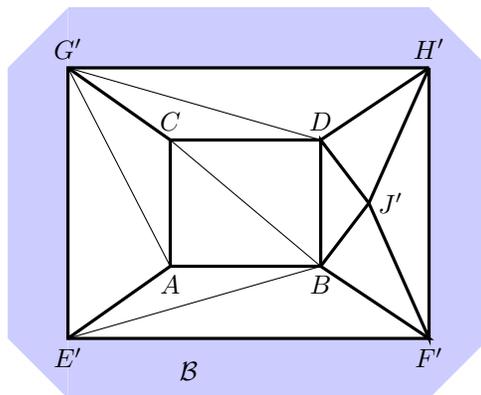
\begin{figure}[H]
\begin{tikzpicture}  [scale= 0.8] % triangulation
\coordinate (A) at (1.7,1.2);
\coordinate (T) at (4.2,1.2);
\coordinate  (C) at (1.7,3.3);
\coordinate  (D) at (4.2,3.3);
\coordinate (E) at (0,0);
\coordinate (F) at (6,0);
\coordinate  (G) at (0,4.5);
\coordinate  (H) at (6,4.5);
\coordinate (J) at (5,2.25);

\coordinate (L) at (0,-1);
\coordinate (M) at (-1,0);
\coordinate  (N) at (-1,4.5);
\coordinate  (O) at (0,5.5);
\coordinate (P) at (6,5.5);
\coordinate (Q) at (7,4.5);
\coordinate  (R) at (7,0);
\coordinate  (S) at (6,-1);

\fill[fill=blue!20] (L) -- (O) -- (N) --(M);
\fill[fill=blue!20] (P) -- (Q) -- (R) -- (S);
\fill[fill=blue!20] (O) -- (G) -- (H)-- (P);
\fill[fill=blue!20] (L) -- (E) -- (F)-- (S);

\node at (A)[below]{$A$};
\node at (T)[below]{$B$};
\node at (C)[above]{$C$};
\node at (D)[above]{$D$};
\node at (E)[below]{$E'$};
\node at (F)[below]{$F'$};
\node at (G)[above]{$G'$};
\node at (H)[above]{$H'$};
\node at (J)[right]{$J'$};

\draw [very thick] (E) -- (A) -- (T) -- (F) -- (E) -- (G) -- (H) ; 
\draw [very thick] (H) -- (J) --  (F)--(H) ;
\draw [very thick] (T) -- (J);
\draw [very thick] (J) -- (D) -- (T);

\draw [very thick] (A) -- (C) --(G);
\draw [very thick] (C) -- (D) --(H);

\draw (G)--(D);
\draw (A) -- (G);
\draw (C)--(T)--(E);

\node at (1.7,-0.55)[right]{$\B$};

\end{tikzpicture}   
\bigskip
\caption{Triangula\c c\~ao de $K$.}\label{Cauchytriangule}
\end{figure}

Posteriormente, estenderemos o buraco cuidando de manter  a borda tal que ela \'e homeomorfa a uma circunfer\^encia. Assim, o estendimento consiste em duas  
opera\c c\~oes que  descreveremos  a seguir. 

A borda de $\B$ est\'a  constitu\'ida de arestas  com a seguinte propriedade.  
Cada aresta \'e uma face de um tri\^angulo que tem somente esta aresta 
como face comum com o buraco $\B$. 
Por exemplo, na figura \ref{figura01}, 
temos dois exemplos de um  tri\^angulo $\sigma$ que 
tem a aresta $\tau$ como face comum com o buraco $\B$. 

A ``opera\c c\~ao I'' consiste  em tirar do poliedro $K$ tal tri\^angulo $\sigma$  juntamente 
com a aresta $\tau$ correspondente e assim o buraco se estende. 

Observamos que a ``opera\c c\~ao I'' n\~ao muda a soma $n_0 - n_1 + n_2$, pois
$n_0$ n\~ao muda enquanto $n_1$ e $n_2$ diminuem 
de $1$, ent\~ao $n_0 - n_1 + n_2$ mante-se igual. 

\begin{figure}[H]
\begin{tikzpicture}  [scale= 0.8] %op�ration I - premier dessin
\coordinate (A) at (1.7,1.2);
\coordinate (T) at (4.2,1.2);
\coordinate  (C) at (1.7,3.3);
\coordinate  (D) at (4.2,3.3);
\coordinate (E) at (0,0);
\coordinate (F) at (6,0);
\coordinate  (G) at (0,4.5);
\coordinate  (H) at (6,4.5);
\coordinate (J) at (5,2.25);

\coordinate (L) at (0,-1);
\coordinate (M) at (-1,0);
\coordinate  (N) at (-1,4.5);
\coordinate  (O) at (0,5.5);
\coordinate (P) at (6,5.5);
\coordinate (Q) at (7,4.5);
\coordinate  (R) at (7,0);
\coordinate  (S) at (6,-1);

\fill[fill=blue!20] (L) -- (O) -- (N) --(M);
\fill[fill=blue!20] (P) -- (Q) -- (R) -- (S);
\fill[fill=blue!20] (O) -- (G) -- (H)-- (P);
\fill[fill=blue!20] (L) -- (E) -- (F)-- (S);

\node at (A)[below]{$A$};
\node at (T)[below]{$B$};
\node at (C)[above]{$C$};
\node at (D)[above]{$D$};
\node at (E)[below]{$E'$};
\node at (F)[below]{$F'$};
\node at (G)[above]{$G'$};
\node at (H)[above]{$H'$};
\node at (J)[right]{$J'$};

\draw [very thick] (E) -- (A) -- (T) -- (F) -- (E) -- (G) -- (H) ; 
\draw [very thick] (H) -- (J) --  (F) ;
\draw [very thick] (T) -- (J);
\draw [very thick] (J) -- (D) -- (T);

\draw [very thick] (A) -- (C) --(G);
\draw [very thick] (C) -- (D) --(H);

 \filldraw[fill=pink] (J)--(H)--(F);
\draw [very thick,red] (H) -- (F) ;
\node at (6.2,2.25){$\tau$};
\node at (5.55,2.5){$\sigma$};

\draw  [very thick] (G)--(D);
\draw  [very thick] (A) -- (G);
\draw  [very thick] (C)--(T)--(E);
\node at (1.7,-0.55)[right]{$\B$};

\end{tikzpicture}   
\qquad \qquad
\begin{tikzpicture}  [scale= 0.8] %op�ration I - deuxi�me dessin
\coordinate (A) at (1.7,1.2);
\coordinate (T) at (4.2,1.2);
\coordinate  (C) at (1.7,3.3);
\coordinate  (D) at (4.2,3.3);
\coordinate (E) at (0,0);
\coordinate (F) at (6,0);
\coordinate  (G) at (0,4.5);
\coordinate  (H) at (6,4.5);
\coordinate (J) at (5,2.25);

\coordinate (L) at (0,-1);
\coordinate (M) at (-1,0);
\coordinate  (N) at (-1,4.5);
\coordinate  (O) at (0,5.5);
\coordinate (P) at (6,5.5);
\coordinate (Q) at (7,4.5);
\coordinate  (R) at (7,0);
\coordinate  (S) at (6,-1);

\fill[fill=blue!20] (L) -- (O) -- (N) --(M);
\fill[fill=blue!20] (P) -- (Q) -- (R) -- (S);
\fill[fill=blue!20] (O) -- (G) -- (H)-- (P);
\fill[fill=blue!20] (L) -- (E) -- (F)-- (S);

\draw [very thick] (E) -- (A) -- (T) -- (F) -- (E) -- (G) ; 
\draw [very thick] (H) -- (J) --  (F) ;
\draw [very thick] (T) -- (J);
\draw [very thick] (J) -- (D) -- (T);

\draw [very thick] (A) -- (C) --(G);
\draw [very thick] (C) -- (D) --(H);

\fill[fill=blue!20] (F) -- (J) -- (H);
 \filldraw[fill=pink] (G)--(H)--(D);
\draw [very thick,red] (H) -- (G) ;
\node at (3.1,4.7){$\tau$};
\node at (3.6,4){$\sigma$};

\node at (A)[below]{$A$};
\node at (T)[below]{$B$};
\node at (C)[above]{$C$};
\node at (D)[above]{$D$};
\node at (E)[below]{$E'$};
\node at (F)[below]{$F'$};
\node at (G)[above]{$G'$};
\node at (H)[above]{$H'$};
\node at (J)[right]{$J'$};

\draw  [very thick] (G)--(D);
\draw  [very thick] (A) -- (G);
\draw  [very thick] (C)--(T)--(E);
\node at (6.2,2.6)[right]{$\B$};

\end{tikzpicture}   
\bigskip
\caption{A ``opera\c c\~ao I''.}\label{figura01}
\end{figure}

Quando aparecer a situa\c c\~ao da figura \ref{figura03}, onde um tri\^angulo 
(aqui $\sigma$) tem duas arestas comuns $\tau_1$ e $\tau_2$  com o buraco, 
realizaremos a ``opera\c c\~ao II'' consistindo em 
tirar do poliedro $K$ tal tri\^angulo $\sigma$ junto com as duas arestas 
$\tau_1$ e $\tau_2$ e com o v\'ertice $H'$ que \'e o v\'ertice comum de $\tau_1$ e $\tau_2$. 
Assim, o buraco se estende. 

\begin{figure}[H]
\begin{tikzpicture}  [scale= 0.8] %op�ration II 
\coordinate (A) at (1.7,1.2);
\coordinate (T) at (4.2,1.2);
\coordinate  (C) at (1.7,3.3);
\coordinate  (D) at (4.2,3.3);
\coordinate (E) at (0,0);
\coordinate (F) at (6,0);
\coordinate  (G) at (0,4.5);
\coordinate  (H) at (6,4.5);
\coordinate (J) at (5,2.25);

\coordinate (L) at (0,-1);
\coordinate (M) at (-1,0);
\coordinate  (N) at (-1,4.5);
\coordinate  (O) at (0,5.5);
\coordinate (P) at (6,5.5);
\coordinate (Q) at (7,4.5);
\coordinate  (R) at (7,0);
\coordinate  (S) at (6,-1);

\fill[fill=blue!20] (L) -- (O) -- (N) --(M);
\fill[fill=blue!20] (P) -- (Q) -- (R) -- (S);
\fill[fill=blue!20] (O) -- (G) -- (H)-- (P);
\fill[fill=blue!20] (L) -- (E) -- (F)-- (S);

\draw [very thick] (E) -- (A) -- (T) -- (F) -- (E) -- (G) ; 
\draw [very thick] (J) --  (F) ;
\draw [very thick] (T) -- (J);
\draw [very thick] (J) -- (D) -- (T);

\draw [very thick] (A) -- (C) --(G);
\draw [very thick] (C) -- (D) ;

\fill[fill=blue!20] (F) -- (J) -- (H);
\fill[fill=blue!20] (G) -- (D) -- (H);
 \filldraw[fill=pink] (D)--(H)--(J);
\draw [very thick,red] (D) -- (H) -- (J) ;
\node at (4.9,4.1){$\tau_1$};
\node at (5.75,3.2){$\tau_2$};
\node at (5,3.3){$\sigma$};
\node at (H)[red] {$\bullet$};

\node at (A)[below]{$A$};
\node at (T)[below]{$B$};
\node at (C)[above]{$C$};
\node at (D)[above]{$D$};
\node at (E)[below]{$E'$};
\node at (F)[below]{$F'$};
\node at (G)[above]{$G'$};
\node at (H)[above]{$H'$};
\node at (J)[right]{$J'$};

\draw  [very thick] (G)--(D);
\draw  [very thick] (A) -- (G);
\draw  [very thick] (C)--(T)--(E);
\node at (6.2,2.6)[right]{$\B$};

\end{tikzpicture}   
\caption{A ``opera\c c\~ao II''.}\label{figura03}
\end{figure}

Observamos que a  ``opera\c c\~ao II''  tamb\'em n\~ao muda a soma $n_0 - n_1 + n_2$,
j\'a que ela diminui $n_0$ e $n_2$ de $1$ e diminui $n_1$ de 2, ent\~ao a soma $n_0 - n_1 + n_2$ \'e preservada.

Usando as duas opera\c c\~oes acima, o buraco se 
estende at\'e ficar com um  tri\^angulo s\'o. 
Neste tri\^angulo temos $n_0 - n_1 + n_2 = 3 - 3 + 1 = 1$. 
Relembrando que tiramos um tri\^angulo no in\'icio, que dava j\'a $+1$ na soma 
$n_0 - n_1 + n_2$, portanto, para qualquer poliedro convexo, temos 
$n_0 - n_1 + n_2 = +2$. 
\end{proof}

\subsection{Depois Cauchy} \label{depois_Cauchy}

Alguns autores, em particular Lakatos \cite{Lak},  criticam a prova de Cauchy. 
No seu livro, (veja \cite{Lak}, p\'aginas 11 e 12), 
Lakatos providencia um contra-exemplo ao m\'etodo de Cauchy.
 Aqui vamos adaptar o  contra-exemplo de Lakatos 
no contexto de  nosso 
exemplo anterior (figuras \ref{poliedre-Cauchy} e \ref{Cauchytriangule}).  

Note que, neste contexto, 
Lakatos usa uma ordem de remo\c c\~ao dos  tri\^angulos indicado na figura \ref{ordemLakatos}. 

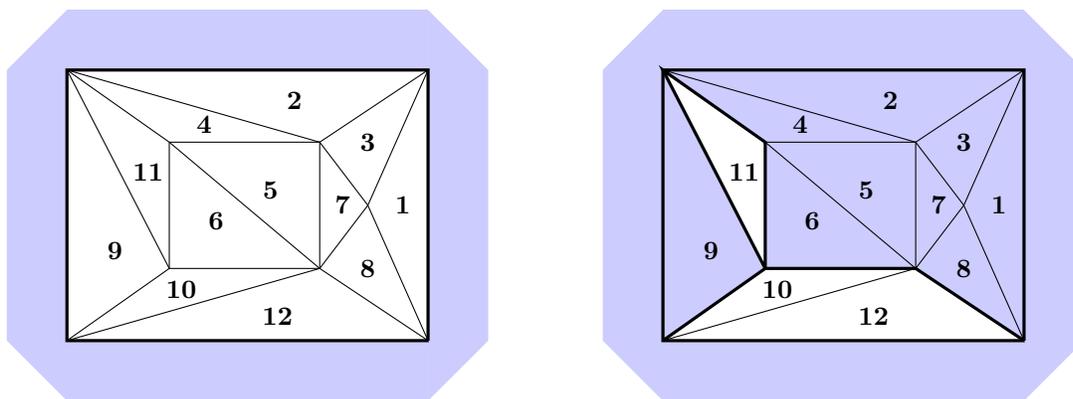
\begin{figure}[H]
\begin{tikzpicture}  [scale= 0.8] % ordre des retraits
\coordinate (A) at (1.7,1.2);
\coordinate (T) at (4.2,1.2);
\coordinate  (C) at (1.7,3.3);
\coordinate  (D) at (4.2,3.3);
\coordinate (E) at (0,0);
\coordinate (F) at (6,0);
\coordinate  (G) at (0,4.5);
\coordinate  (H) at (6,4.5);
\coordinate (J) at (5,2.25);

\coordinate (L) at (0,-1);
\coordinate (M) at (-1,0);
\coordinate  (N) at (-1,4.5);
\coordinate  (O) at (0,5.5);
\coordinate (P) at (6,5.5);
\coordinate (Q) at (7,4.5);
\coordinate  (R) at (7,0);
\coordinate  (S) at (6,-1);

\fill[fill=blue!20] (L) -- (O) -- (N) --(M);
\fill[fill=blue!20] (P) -- (Q) -- (R) -- (S);
\fill[fill=blue!20] (O) -- (G) -- (H)-- (P);
\fill[fill=blue!20] (L) -- (E) -- (F)-- (S);

\draw (E) -- (A) -- (T) -- (F) ;
\draw[very thick] (F) -- (E) -- (G) -- (H)--(F) ; 
\draw (H) -- (J) --  (F)--(H) ;
\draw  (T) -- (J);
\draw  (J) -- (D) -- (T);

\draw  (A) -- (C) --(G);
\draw   (C) -- (D) --(H);

\draw (G)--(D);
\draw (A) -- (G);
\draw (C)--(T)--(E);

\node at (5.3,2.25)[right]{$\bf 1$};
\node at (3.5,4)[right]{$\bf 2$};
\node at (5,3.3){$\bf 3$};
\node at (2,3.6)[right]{$\bf 4$};
\node at (3.1,2.5)[right]{$\bf 5$};
\node at (2.2,2)[right]{$\bf 6$};
\node at (4.3,2.25)[right]{$\bf 7$};
\node at (5,1.2){$\bf 8$};
\node at (0.8,1.5){$\bf 9$};
\node at (1.9,0.85) {$\bf 10$};
\node at (1.35,2.8) {$\bf 11$};
\node at (3.5,0.4) {$\bf 12$};

\end{tikzpicture}   
\qquad\qquad
\begin{tikzpicture}  [scale= 0.8] % dernier
\coordinate (A) at (1.7,1.2);
\coordinate (T) at (4.2,1.2);
\coordinate  (C) at (1.7,3.3);
\coordinate  (D) at (4.2,3.3);
\coordinate (E) at (0,0);
\coordinate (F) at (6,0);
\coordinate  (G) at (0,4.5);
\coordinate  (H) at (6,4.5);
\coordinate (J) at (5,2.25);

\coordinate (L) at (0,-1);
\coordinate (M) at (-1,0);
\coordinate  (N) at (-1,4.5);
\coordinate  (O) at (0,5.5);
\coordinate (P) at (6,5.5);
\coordinate (Q) at (7,4.5);
\coordinate  (R) at (7,0);
\coordinate  (S) at (6,-1);

\fill[fill=blue!20] (G) -- (C) -- (D) --(H);
\fill[fill=blue!20] (T) -- (F) -- (H) -- (D);
\fill[fill=blue!20] (A) -- (T) -- (D)-- (C);
\fill[fill=blue!20] (G) -- (E) -- (A);

\fill[fill=blue!20] (L) -- (O) -- (N) --(M);
\fill[fill=blue!20] (P) -- (Q) -- (R) -- (S);
\fill[fill=blue!20] (O) -- (G) -- (H)-- (P);
\fill[fill=blue!20] (L) -- (E) -- (F)-- (S);

%\fill [red,opacity=.5] (A) circle (7pt);  % il faudra enlever
 
 \node at (5.3,2.25)[right]{$\bf 1$};
\node at (3.5,4)[right]{$\bf 2$};
\node at (5,3.3){$\bf 3$};
\node at (2,3.6)[right]{$\bf 4$};
\node at (3.1,2.5)[right]{$\bf 5$};
\node at (2.2,2)[right]{$\bf 6$};
\node at (4.3,2.25)[right]{$\bf 7$};
\node at (5,1.2){$\bf 8$};
\node at (0.8,1.5){$\bf 9$};
\node at (1.9,0.85) {$\bf 10$};
\node at (1.35,2.8) {$\bf 11$};
\node at (3.5,0.4) {$\bf 12$};

\draw[very thick] (E)--(A) -- (T) -- (F);
\draw (E)--(T);
\draw[very thick] (F) -- (E) -- (G) -- (H)--(F) ; 
\draw (H) -- (J) --  (F)--(H) ;
\draw  (T) -- (J);
\draw  (J) -- (D) -- (T);
\draw (G)--(D);
\draw [very thick]  (A) -- (C) --(G)--(A);
\draw   (C) -- (D) --(H);
\draw (C)--(T);

\end{tikzpicture}   
\bigskip
\caption{Ordem de remo\c c\~ao seguido Lakatos}\label{ordemLakatos}
\end{figure}

\begin{figure}[H]
\begin{tikzpicture}  [scale= 0.9] % tbord de B avec point double
\coordinate (A) at (1.7,1.2);
\coordinate (T) at (4.2,1.2);
\coordinate  (C) at (1.7,3.3);
\coordinate  (D) at (4.2,3.3);
\coordinate (E) at (0,0);
\coordinate (F) at (6,0);
\coordinate  (G) at (0,4.5);
\coordinate  (H) at (6,4.5);
\coordinate (J) at (5,2.25);

\coordinate (L) at (0,-1);
\coordinate (M) at (-1,0);
\coordinate  (N) at (-1,4.5);
\coordinate  (O) at (0,5.5);
\coordinate (P) at (6,5.5);
\coordinate (Q) at (7,4.5);
\coordinate  (R) at (7,0);
\coordinate  (S) at (6,-1);

\fill[fill=blue!20] (G) -- (C) -- (D) --(H);
\fill[fill=blue!20] (T) -- (F) -- (H) -- (D);
\fill[fill=blue!20] (A) -- (T) -- (D)-- (C);
\fill[fill=blue!20] (G) -- (E) -- (A);

\fill[fill=blue!20] (L) -- (O) -- (N) --(M);
\fill[fill=blue!20] (P) -- (Q) -- (R) -- (S);
\fill[fill=blue!20] (O) -- (G) -- (H)-- (P);
\fill[fill=blue!20] (L) -- (E) -- (F)-- (S);

 \filldraw[fill=pink] (E)--(A)--(T);
 
 \node at (5.3,2.25)[right]{$\bf 1$};
\node at (3.5,4)[right]{$\bf 2$};
\node at (5,3.3){$\bf 3$};
\node at (2,3.6)[right]{$\bf 4$};
\node at (3.1,2.5)[right]{$\bf 5$};
\node at (2.2,2)[right]{$\bf 6$};
\node at (4.3,2.25)[right]{$\bf 7$};
\node at (5,1.2){$\bf 8$};
\node at (0.8,1.5){$\bf 9$};
\node at (1.9,0.85) {$\bf 10$};
\node at (1.35,2.8) {$\bf 11$};
\node at (3.5,0.4) {$\bf 12$};

\draw[very thick,red] (E)--(A) -- (T) ;
\draw[very thick] (E)--(T)--(F);
\draw[very thick] (F) -- (E) -- (G) -- (H)--(F) ; 
\draw (H) -- (J) --  (F)--(H) ;
\draw  (T) -- (J);
\draw  (J) -- (D) -- (T);
\draw (G)--(D);
\draw [very thick]  (A) -- (C) --(G)--(A);
\draw   (C) -- (D) --(H);
\draw (C)--(T);

\end{tikzpicture}   
\caption{Onde a prova de Cauchy n\~ao funciona}\label{Lakatos3}
\end{figure}
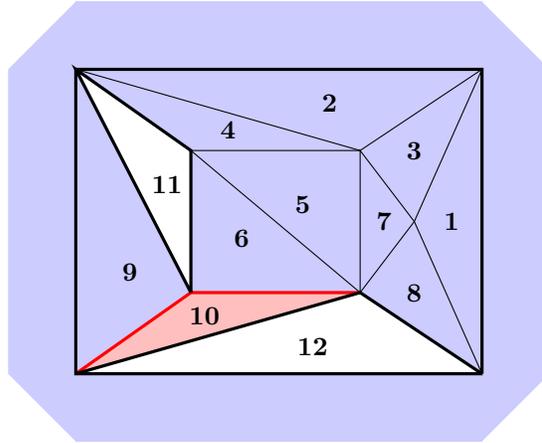

Estendemos o buraco, removendo tri\^angulos na ordem indicada nesta 
 figura, usando as opera\c c\~oes de Cauchy I e II at\'e o tri\^angulo do n\'umero 9. 
 A borda do buraco n\~ao \'e mais uma curva homeomorfa a uma circunf\^erencia.  
 Se {removemos}  o tri\^angulo do n\'umero 10, observamos que tiramos duas arestas e um tri\^angulo, 
ent\~ao a soma $n_0 - n_1 + n_2$ n\~ao \'e preservada. Al\'em disso, neste caso, 
 o buraco desconecta o resto da figura.  Os 
tri\^angulos dos n\'umeros 11 e 12  n\~ao est\~ao mais conectados. 

Assim, {\'e preciso de} 
cuidar da ordem de retirada dos tri\^angulos porque pode acontecer situa\c c\~oes 
{tais} que o buraco desconecta o poliedro $K$. Al\'em disso, 
a borda do buraco n\~ao \'e mais uma curva homeomorfa a uma {circunfer\^encia}, 
ela tem um ponto m\'ultiplo.

No artigo \cite{Li2}, Lima formaliza os argumentos de Lakatos e explicita 
situa\c c\~oes das figuras  \ref{figuraCE}. Nas figuras (a), (b) e (c),  a extens\~ao do buraco,  
retirando do poliedro $K$ o tri\^angulo $\sigma$, {por um lado} muda a soma $n_0 - n_1 + n_2$
e {por outro lado} desconecta o poliedro $K$. O exemplo de Lakatos corresponde 
a situa\c c\~ao (a) na figura de Elon Lima. 
Nos exemplos de Elon Lima, a borda do buraco $\B$ \'e uma curva com pontos m\'ultiplos. 

Observamos que na figura (d), que tamb\'em \'e  um caso tendo ponto m\'ultiplo, por\'em
 a soma $n_0 - n_1 + n_2$ est\'a preservada quando retiramos de $K$ 
o tri\^angulo $\sigma$, porque tiramos dois v\'ertices $x_2$ e $x_3$, 
tr\^es arestas $(x_1,x_2)$, $(x_2,x_3)$, $(x_1,x_3)$ e o tri\^angulo $\sigma$.
Parece que  Elon Lima n\~ao 
percebeu que este caso tamb\'em \'e admiss\'ivel. 
Ent\~ao a situa\c c\~ao da figura (d) pode ser usado no processo de Cauchy. 
Com efeito, realizaremos uma  opera\c c\~ao  adicional 
que definimos como  ``opera\c c\~ao III''  
e que pode ser usada de maneira alternativa 
na Etapa 5 da prova do teorema \ref{teo1}
(veja observa\c c\~ao \ref{rem111}).
Esta opera\c c\~ao \'e usada tamb\'em em \cite{Cel}. 

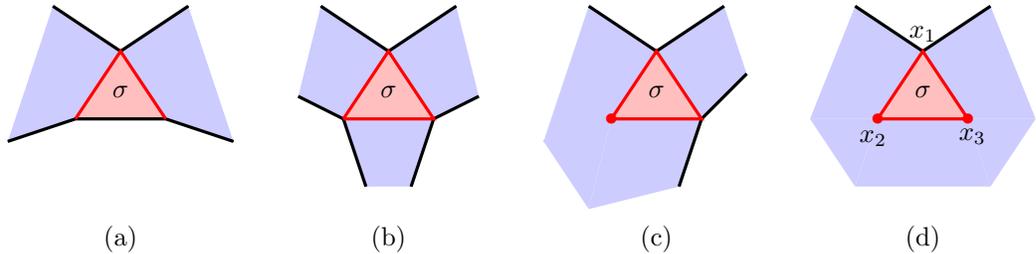
\begin{figure} [H] %figure 1
\begin{tikzpicture}  [scale= 0.30]

\coordinate (A) at (-2,0);
\coordinate (B) at (2,0);
\coordinate  (C) at (0,3);
\coordinate  (D) at (-3,5);
\coordinate (E) at (3,5);
\coordinate (F) at (-5, -1);
\coordinate  (G) at (5,-1);
\coordinate  (H) at (-4,1);
\coordinate  (I) at (4,1);
\coordinate (J) at (-1,-3);
\coordinate (K) at (1,-3);
\coordinate (M) at (-3, -4);
\coordinate  (N) at (1,-3);
\coordinate  (0) at (4,2);
\coordinate  (P) at (-5,0);
\coordinate (Q) at (-3,-3);
\coordinate (R) at (3,-3);
\coordinate (S) at (5,0);

  \draw (E)   --  (C) -- (B) -- (G) ; 

   \filldraw[fill=blue!20] (D)   --  (C) -- (A) -- (F) ; 
     \filldraw[fill=blue!20] (E)   --  (C) -- (B) -- (G) ; 
     
     \filldraw[fill=pink] (A)   --  (C) -- (B) ; 
     
      \draw [very thick](D)   --  (C) -- (E) ; 
        \draw[very thick] (F)   --  (A) -- (B) -- (G) ; 
       \draw[very thick,red]  (A)   --  (C)-- (B) ; 
         
           \node (c) at (0,1.2) {$\sigma$}; 
              \node (a) at (0,-6.3) [above] {$\rm (a)$}; 
   \end{tikzpicture}
\qquad 
    \begin{tikzpicture}  [scale= 0.30] %figure 2

\coordinate (A) at (-2,0);
\coordinate (B) at (2,0);
\coordinate  (C) at (0,3);
\coordinate  (D) at (-3,5);
\coordinate (E) at (3,5);
\coordinate (F) at (-5, -1);
\coordinate  (G) at (5,-1);
\coordinate  (H) at (-4,1);
\coordinate  (I) at (4,1);
\coordinate (J) at (-1,-3);
\coordinate (K) at (1,-3);
\coordinate (M) at (-3, -4);
\coordinate  (N) at (1,-3);
\coordinate  (0) at (4,2);
\coordinate  (P) at (-5,0);
\coordinate (Q) at (-3,-3);
\coordinate (R) at (3,-3);
\coordinate (S) at (5,0);

   \filldraw[fill=blue!20] (D)   --  (C) -- (A) -- (H) ; 
     \filldraw[fill=blue!20] (E)   --  (C) -- (B) -- (I) ; 
      \filldraw[fill=blue!20] (J) -- (A)   --  (B)-- (K) ; 

     \filldraw[fill=pink] (A)   --  (C) -- (B) ; 
      
         \draw [very thick](D)   --  (C) -- (E) ; 
        \draw[very thick] (H)   --  (A) -- (J) ; 
           \draw[very thick] (K)   --  (B) -- (I); 
       \draw[very thick,red]  (A)   --  (C)-- (B)--(A) ; 
   
     \node (c) at (0,1.2) {$\sigma$}; 
 \node (a) at (0,-6.3) [above] {$\rm (b)$}; 
   \end{tikzpicture}
\qquad 
\begin{tikzpicture}  [scale= 0.30] %figure3

\coordinate (A) at (-2,0);
\coordinate (B) at (2,0);
\coordinate  (C) at (0,3);
\coordinate  (D) at (-3,5);
\coordinate (E) at (3,5);
\coordinate (F) at (-5, -1);
\coordinate  (G) at (5,-1);
\coordinate  (H) at (-4,1);
\coordinate  (I) at (4,1);
\coordinate (J) at (-1,-3);
\coordinate (K) at (1,-3);
\coordinate (M) at (-3, -4);
\coordinate  (N) at (1,-3);
\coordinate  (O) at (4,2);
\coordinate  (P) at (-5,0);
\coordinate (Q) at (-3,-3);
\coordinate (R) at (3,-3);
\coordinate (S) at (5,0);

   \fill[fill=blue!20] (D)   --  (C) -- (A) -- (F) ; 
     \fill[fill=blue!20] (F)   --  (A) -- (M)  ; 
        \fill[fill=blue!20] (M)   -- (A) --(B) -- (N) ; 
   \filldraw[fill=blue!20] (E)   --  (C) -- (B) -- (O) ; 
   \draw[very thick,blue!20] (F)--(A)--(H);
   
     \filldraw[fill=pink] (A)   --  (C) -- (B) ; 
              
                    \draw [very thick](D)   --  (C) -- (E) ; 
           \draw[very thick] (K)   --  (B) -- (O); 
       \draw[very thick,red]  (A)   --  (C)-- (B)--(A) ; 
            
            \node (u) at (-2,0) {\textcolor{red}{$\bullet$}}; 
              \node (c) at (0,1.2) {$\sigma$}; 
             \node (a) at (0,-6.3) [above] {$\rm (c)$}; 
   \end{tikzpicture}
\qquad 
\begin{tikzpicture}  [scale= 0.30]%figure4

\coordinate (A) at (-2,0);
\coordinate (B) at (2,0);
\coordinate  (C) at (0,3);
\coordinate  (D) at (-3,5);
\coordinate (E) at (3,5);
\coordinate (F) at (-5, -1);
\coordinate  (G) at (5,-1);
\coordinate  (H) at (-4,1);
\coordinate  (I) at (4,1);
\coordinate (J) at (-1,-3);
\coordinate (K) at (1,-3);
\coordinate (M) at (-3, -4);
\coordinate  (N) at (1,-3);
\coordinate  (0) at (4,2);
\coordinate  (P) at (-5,0);
\coordinate (Q) at (-3,-3);
\coordinate (R) at (3,-3);
\coordinate (S) at (5,0);

   \fill[fill=blue!20] (D)   --  (C) -- (A) -- (P) ; 
     \fill[fill=blue!20] (P)   --  (A) -- (Q)  ; 
        \fill[fill=blue!20] (Q)   -- (A) --(B) -- (R) ; 
        \fill[fill=blue!20] (R)   --  (B) -- (S)  ;    
   \fill[fill=blue!20] (S)   --  (B) -- (C) -- (E) ; 
   
     \filldraw[fill=pink] (A)   --  (C) -- (B) ; 
        
             \draw [very thick](D)   --  (C) -- (E) ; 
       \draw[very thick,red]  (A)   --  (C)-- (B)--(A) ; 
               \node (v) at (2,0) {\textcolor{red}{$\bullet$}}; 
            \node (u) at (-2,0) {\textcolor{red}{$\bullet$}}; 
            
      \node (a) at (0,3) [above] {$x_1$}; 
      \node (b) at (-2.2,-0.1) [below]{$x_2$}; 
      \node (c) at (2.2,0)[below] {$x_3$}; 
         \node (c) at (0,1.2) {$\sigma$}; 
          \node (a) at (0,-6.3) [above] {{$\rm (d)$}}; 
   \end{tikzpicture}
\caption{Os contra-exemplos de E. Lima.  
O buraco $\B$ est\'a em azul, o  tri\^angulo  $\sigma$ 
a ser removido est\'a em cor de rosa e os demais simplexos a ser 
removidos est\~ao em vermelho.}\label{figuraCE}
\end{figure}

A conclus\~ao de Lima (\cite{Li2}, p\'agina 73) \'e que 
{\it a justifica\c c\~ao do m\'etodo de Cauchy (se\c c\~ao \ref{methode_Cauchy}) 
depende de um ``resultado 
profundo da Topologia,
equivalente \`a demonstra\c c\~ao da invari\^ancia dos grupos de homologia da esfera  $\Sp^2.$''} 
Neste artigo, na prova que providenciaremos na se\c c\~ao \ref{OTeo}, 
mostramos que o  m\'etodo de Cauchy (se\c c\~ao \ref{methode_Cauchy}) vale 
 sem usar  resultados mais profundos e somente usando sub-triangula\c c\~oes adequadas.

Note que usando a opera\c c\~ao III, o contra-exemplo de Lakatos n\~ao \'e mais 
um contra-exemplo: voltando na figura \ref{ordemLakatos}, depois de remover o tri\^angulo 
do num\'ero 9, podemos remover o tri\^angulo do num\'ero 12 pela opera\c c\~ao II e depois 
remover o tri\^angulo do num\'ero 10 pela opera\c c\~ao III. Fica somente o tri\^angulo do num\'ero 11 para qual temos $n_0 - n_1 + n_2 = +1$, o que d\'a o resultado!

Alguns autores sugerem uma estrat\'egia para definir uma ordem de remo\c c\~ao dos tri\^angulos 
que permite de usar o m\'etodo de Cauchy para chegar ao resultado. 
Por exemplo, Kirk \cite{Kirk} sugere a seguinte estrat\'egia: 
``{\it There are two important rules to follow when doing this. 
Firstly, we must always perform {\rm [Opera\c c\~ao II]}  when it is possible to do so; if there is a choice 
between {\rm [Opera\c c\~ao I]} and {\rm [Opera\c c\~ao II]} we must always choose {\rm [Opera\c c\~ao II]}. 
If we do not, the network may break up into separate pieces. 
Secondly, we must only remove faces one at a time}.'' 
Providenciamos, na figura \ref{contra-exemplo-Abigail}, um exemplo cujo processo segue  as regras da estrat\'egia definida por  Kirk mas
o processo chega a desconectar o poliedro. Ent\~ao estas regras n\~ao est\~ao 
suficientes. E f\'acil construir um exemplo mostrando tamb\'em que elas n\~ao est\~ao necess\'arias.

 \begin{figure} [H]
\begin{tikzpicture}  [scale= 0.65]

\coordinate (A) at (0,0);
\coordinate (B) at (4,0);
\coordinate  (C) at (8,0);
\coordinate  (D) at (2,2);
\coordinate (E) at (6,2);
\coordinate (F) at (0,4);
\coordinate  (G) at (2,4);
\coordinate  (H) at (4,4);
\coordinate (J) at (6,4);
\coordinate (K) at (8,4);
\coordinate  (L) at (1,5);
\coordinate (M) at (4,5);
\coordinate (N) at (0,6);
\coordinate  (O) at (7,6);
\coordinate  (P) at (2,6);
\coordinate  (Q) at (6,6);
\coordinate (R) at (8,6);

\coordinate (S) at (-1,6);
\coordinate  (T) at (0,7);
\coordinate (U) at (8,7);
\coordinate (V) at (9,6);
\coordinate  (W) at (-1,0);
\coordinate  (X) at (0,-1);
\coordinate  (Y) at (8,-1);
\coordinate (Z) at (9,0);

\fill[fill=blue!20] (S) -- (T) -- (U) --(V);
\fill[fill=blue!20] (W) -- (X) -- (Y)--(Z);
\fill[fill=blue!20] (S) -- (N) -- (A)--(W);
\fill[fill=blue!20] (C) -- (R) -- (V)--(Z);

\draw (A)--(C)--(R)--(N)--(A);
\draw   (A)--(D)--(B)--(E)--(C);
\draw   (F)--(D)--(E)--(K)--(J)--(E);
\draw (E)--(H) -- (D) -- (G) -- (F);
\draw (F) -- (L) -- (G) --(H) -- (J) -- (R);
\draw (J) -- (Q) -- (H)--(M) --(Q) ;
\draw (H) -- (P) --  (M)  ;
\draw (G) -- (P) --  (L)--(N)  ;

\node at (7.5,-0.5) {$(a)$};
\end{tikzpicture}   
\qquad
\begin{tikzpicture}  [scale= 0.65] % 10

\coordinate (A) at (0,0);
\coordinate (B) at (4,0);
\coordinate  (C) at (8,0);
\coordinate  (D) at (2,2);
\coordinate (E) at (6,2);
\coordinate (F) at (0,4);
\coordinate  (G) at (2,4);
\coordinate  (H) at (4,4);
\coordinate (J) at (6,4);
\coordinate (K) at (8,4);
\coordinate  (L) at (1,5);
\coordinate (M) at (4,5);
\coordinate (N) at (0,6);
\coordinate  (O) at (7,6);
\coordinate  (P) at (2,6);
\coordinate  (Q) at (6,6);
\coordinate (R) at (8,6);

\coordinate (S) at (-1,6);
\coordinate  (T) at (0,7);
\coordinate (U) at (8,7);
\coordinate (V) at (9,6);
\coordinate  (W) at (-1,0);
\coordinate  (X) at (0,-1);
\coordinate  (Y) at (8,-1);
\coordinate (Z) at (9,0);

\fill[fill=blue!20] (S) -- (T) -- (U) --(V);
\fill[fill=blue!20] (W) -- (X) -- (Y)--(Z);
\fill[fill=blue!20] (S) -- (N) -- (A)--(W);
\fill[fill=blue!20] (C) -- (R) -- (V)--(Z);

\fill[fill=blue!20] (K) -- (R) -- (Q) --(J);
\fill[fill=blue!20] (A) -- (D) -- (H) -- (P) --(G) -- (F) --(A);
\fill[fill=blue!20] (Q) --(J)--(H)--(M)--(Q);
\fill[fill=blue!20] (Q) --(J)--(H)--(M)--(Q);
\fill[fill=blue!20] (M)--(P) -- (Q);
\fill[fill=pink] (M)--(P) -- (H);

\draw (A)--(C)--(R)--(N)--(A);
\draw   (A)--(D)--(B)--(E)--(C);
\draw   (F)--(D)--(E)--(K)--(J)--(E);
\draw (E)--(H) -- (D) -- (G) -- (F);
\draw (F) -- (L) -- (G) --(H) -- (J) -- (R);
\draw (J) -- (Q) -- (H)--(M) --(Q) ;
\draw (H) -- (P) --  (M)  ;
\draw (G) -- (P) --  (L)--(N)  ;

\draw[very thick,red] (H)--(M) -- (P)--(H);
\node at (M) {$\bullet$};

\node at (7.5,4.5){$5_I$};
\node at (7,5.5){$6_{II}$};
\node at (5.5,4.5){$7_I$};
\node at (1,2){$1_{I}$};
\node at (1.4,3.5){$2_I$};
\node at (2.8,3.5){$3_{I}$};
\node at (2.8,4.5){$4_I$};
\node at (4.6,5){$8_I$};
\node at (4,5.6){$9_{II}$};
\node at (3.7,4.7){$10$};

\node at (P) [above]{$A$};
\node at (H) [below]{$B$};
\node at (7.5,-0.5) {$(b)$};
\end{tikzpicture}    

 \medskip 
\caption{Um contra-exemplo \`a ``regra'' dada por  Kirk.} \label{contra-exemplo-Abigail}
\end{figure}
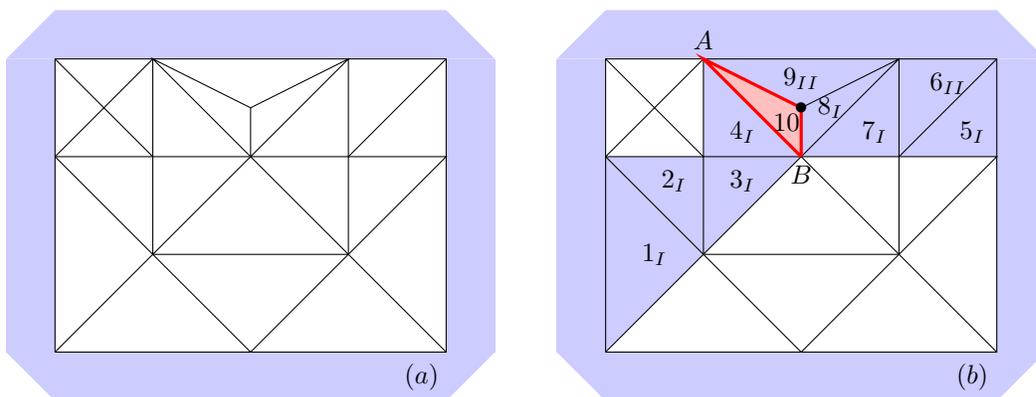

Na figura \ref{contra-exemplo-Abigail}, os tri\^angulos est\~ao removidos 
na ordem dos n\'umeros. O \'indice I ou II em baixo de cada n\'umero 
significa que este tri\^angulo est\'a removido com a opera\c c\~ao  I ou II correspondente.
Aqui, chegando ao n\'umero 8, a borda do buraco n\~ao \'e mais homeomorfa a uma circunfer\^encia. 

O exemplo da figura \ref{contra-exemplo-Abigail} 
\'e tamb\'em um contra-exemplo mostrando que, mesmo com a  opera\c c\~ao III, 
o processo de Cauchy n\~ao vale com a ordem de remo\c c\~ao escolhida.  
De fato, depois de remover o tri\^angulo do n\'umero  7, podemos continuar at\'e 
o tri\^angulo do  n\'umero  10.
Mas n\~ao podemos remover  este tri\^angulo porque, neste caso, removemos um v\'ertice, 
tr\^es arestas e um tri\^angulo. A soma alterada correspondente 
\'e $n_0 - n_1 + n_2 = 1 -3+1 = -1$. 
Note que os  v\'ertices $A$ e $B$ n\~ao pertencem ao buraco. 

A nosso conhecimento, nenhum autor deu uma estrat\'egia adaptada de {remo\c c\~ao} dos tri\^angulos 
que permite usar somente o m\'etodo de Cauchy e ferramentas conhecidas na \'epoca 
de Cauchy.

Depois de Cauchy,  v\'arios autores propuseram outras provas da f\'ormula (\ref{Euler}), 
usando argumentos originais.
Veja o site de David Eppstein \cite{Epp} que cont\^em 20 provas diferentes, usando 
ferramentas que aparecem somente depois da \'epoca de Cauchy. 
Em particular, algumas provas usam o Lema de Jordan  (Jordan \cite{Jo}, 1866).  
De fato, veremos que as curvas de Jordan aparecem  na prova de nosso Teorema \ref{teo1}
como um  artefato.

 A conclus\~ao \'e que a  resposta para a segunda pergunta tamb\'em n\~ao \'e clara. 
O que \'e claro  \'e que todos os quatro matem\'aticos Descartes, Euler, Legendre e Cauchy 
 jogaram um papel importante na descoberta e/ou  na prova da f\'ormula. 
 Embora disso, nenhum deles percebeu a import\^ancia da f\'ormula. 
 Depois da contribu\c c\~ao que vimos,  todos os quatro deixaram
este assunto e n\~ao voltaram a trabalhar nele.

\smallskip

\subsection{Generaliza\c c\~oes da f\'ormula (\ref{Euler}).}\label{genera}

A f\'ormula (\ref{Euler}) foi estendida primeiramente por Lhuilier, para superf\'\i cies orient\'aveis 
de g\^enero $g$, como sendo:
$$ n_0 -n_1 + n_2 = 2-2g.$$

No caso de superf\'\i cies n\~ao orient\'aveis,  a f\'ormula \'e dada (veja \cite{Mas}) por 
$$ n_0 -n_1 + n_2 = 2-g.$$

O resultado geral foi providenciado por Poincar\'e \cite{Po3,Po1} que mostra que, para qualquer 
triangula\c c\~ao de um poliedro $X$ de dimens\~ao $k$, chamando 
de $n_i$ o n\'umero de simplexos de dimens\~ao $i$, a soma 
\begin{equation}\label{Euler-Poincare}
\chi (X) = \sum_{i=0}^k (-1)^i n_i
\tag{3}
\end{equation}
n\~ao depende da triangula\c c\~ao de $X$. Esta soma \'e chamada de 
caracter\'istica de Euler-Poincar\'e de $X$.

Assim, a caracter\'istica de Euler-Poincar\'e do toro $\T$ vale $0$ (o toro \'e de 
g\^enero $2$). A caracter\'istica de Euler-Poincar\'e do plano projetivo ${\mathbb{P}}^2$ vale $+1$, 
da garrafa de Klein vale $0$. Observe-se que a caracter\'istica de Euler-Poincar\'e do toro pin\c cado 
(veja \S \ref{pinc\'e}), que \'e uma superf\'\i cie singular, vale $+1$.

\section{O m\'etodo de Cauchy na demonstra\c c\~ao da f\'ormula de Euler} \label{OTeo} 

As representa\c c\~oes planares das superf\'icies que vimos: a esfera (figura \ref{triangsphere2}), 
o toro (figura \ref{o toro}), o plano projetivo (figura \ref{Projplanar}) e 
a garrafa de Klein (figura \ref{Klein}) s\~ao exemplos da seguinte situa\c c\~ao: Um pol\'igono 
triangulado  $K$ homeomorfo a um disco $D$, tal que temos, eventualmente, identifica\c c\~oes de simplexos 
 na borda de $K$ (veja Defini\c c\~ao \ref{repplanar}). 
Nesta se\c c\~ao, mostramos o seguinte teorema utilizando s\'o  o m\'etodo  de Cauchy 
(se\c c\~ao \ref{methode_Cauchy}) e sub-triangula\c c\~oes:
\begin{theorem} \label{teo1}
Seja $K$  um pol\'igono  triangulado em $\R^2$, homeomorfo 
a um disco $D$, com poss\'iveis identifica\c c\~oes dos simplexos na borda de $K$, 
o que denotamos por $K_0$.  Temos
$$\chi(K) = \chi (K_0) +1.$$
\end{theorem}

Ressaltamos que o ponto importante de nossa demonstra\c c\~ao \'e, a partir  da triangula\c c\~ao 
dada, providenciamos uma sub-triangula\c c\~ao tal que possamos 
 provar o teorema usando apenas o m\'etodo 
de Cauchy (se\c c\~ao \ref{methode_Cauchy}),  sem outras ferramentas. 
Uma tal sub-triangula\c c\~ao pode ser realizada usando o seguinte lema.

\begin{lem}[Sub-triangula\c c\~ao]\label{subtriang}
Seja $\sigma$ um simplexo de dimens\~ao 2 (tri\^angulo) em uma triangula\c c\~ao $K$ 
de um pol\'igono 
e $\tau = (x,y)$ um segmento que divide $\sigma$, onde $x$ e $y$ pertencem  a 
dois lados diferentes do tri\^angulo $\sigma$.    Ent\~ao, existe uma 
sub-triangula\c c\~ao $K'$ de $K$ tal que $\tau$ \'e um simplexo de dimens\~ao $1$ de $K'$ 
e que o n\'umero $n_0 -n_1 + n_2$ \'e o mesmo para $K$ e $K'$. 
\end{lem}

\begin{proof}
Temos duas situa\c c\~oes poss\'iveis, descritas nas figuras 
\ref{subdividi} (i) e  (ii), seguido que um v\'ertice do segmento 
$\tau$ \'e um v\'ertice do tri\^angulo $\sigma$ ou n\~ao. 
Descrevemos a primeira  situa\c c\~ao (figura \ref{subdividi} (i)). 
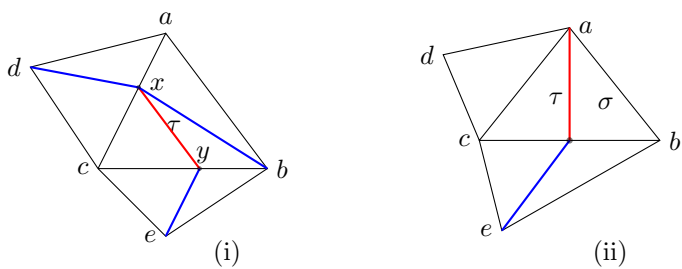
\begin{figure}[H] 
\begin{tikzpicture} [scale=0.45]

% triangle $abc$
\draw (0,0) -- (5,0) -- (2,4) -- (0,0);
\node at (0,0)[left] {$c$};
\node at (5,0)[right] {$b$};
\node at (2,4)[above] {$a$};

% triangle $acd$
\draw (0,0) -- (-2, 3) -- (2,4);
\node at (-2,3)[left] {$d$};

% triangle $cbe$
\draw (0,0) -- (2, -2) -- (5,0);
\node at (2,-2)[left] {$e$};

%\draw[thick, red] (5,-2) -- (-1,4); 
\node at (1.25, 2.5)[right] {$x$};
\node at (3.1,-0.1)[above] {$y$};
%\node at (0.5, 3) {\textcolor{red}{$\ell$}};
\node at (2.25, 1.25) {$\tau$};
\draw[thick, red] (1.2, 2.4) -- (3,0); 

%sub-triangulation: lignes $xd$
%ligne xd
\draw [thick, blue] (1.2, 2.4) -- (-2,3);
%ligne $xb$
\draw [thick, blue] (1.2, 2.4) -- (5,0);
%ligne $ye$
\draw [thick, blue] (3,0) -- (2,-2);
  \fill [black,opacity=.5] (1.2, 2.4) circle (2pt);
    \fill [black,opacity=.5] (3,0) circle (2pt);
    
    \node at (4.5,-2.5)[left] {$\rm (i)$};
\end{tikzpicture}
\qquad\qquad
 \begin{tikzpicture}  [scale= 0.6]    

\coordinate [label=right:$a$]  (a) at (1,2.5);
\coordinate [label=right:$b$]  (b) at (3,0);
\coordinate [label=left:$c$]  (c) at (-1,0);
\coordinate [label=left:$d$]  (d) at (-1.8,1.9);
\coordinate [label=left:$e$]  (e) at (-0.5,-2);
\coordinate [label=above:$\sigma$]  (s) at (1.8,0.5);
\coordinate [label=above:$\tau$]  (t) at (0.7,0.6);

\coordinate   (l2) at (1,0);

\draw (a)--(b) -- (e) -- (c) -- (b);
\draw (a)--(c) -- (d) -- (a);
\draw [thick,red] (a)--(l2);
\draw [thick,blue] (l2) -- (e);

    \fill [black,opacity=.5] (l2) circle (2pt);
     \node at (2.5,-2.5)[left] {$\rm (ii)$};
\end{tikzpicture}
\caption{Sub-triangula\c c\~ao $K'$}\label{subdividi}
\end{figure}
Denotamos ${a,b,c}$ os v\'ertices de $\sigma$. Suponhamos que $\tau$ 
seja o segmento $(x,y)$ com $x\in (a,c)$ e $y\in (b,c)$.
Definamos a triangula\c c\~ao $K'$ como seguinte: 
O tri\^angulo $(x,y,c)$ \'e um simplexo de dimens\~ao 2 de $K'$.
O quadril\'atero $(a,b,y,x)$ pode ser dividido 
em dois tri\^angulos, por exemplo $(a,b,x)$ e $(b,x,y)$, 
 que fornecem 
dois simplexos de dimens\~ao 2 de $K'$. 
Agora, $(a,c)$ \'e uma face  de um simplexo 
de dimens\~ao 2 de $K$, seja $(a,c,d)$. 
Dividimos este simplexo em dois simplexos de 
dimens\~ao 2,  que s\~ao $(a,d,x)$ e $(d,x,c)$. 
Analogamente, o simplexo $(c,b,e)$ ser\'a dividido em dois 
simplexos de dimens\~ao 2.  Assim obtemos a nova triangula\c c\~ao $K'$ 
no qual $n_0$ aumenta de $2$ e $n_1$ aumenta de 6, 
e $n_2$ aumenta de $4$, ent\~ao a
soma  $n_0 -n_1 + n_2$ mant\'em-se igual para $K$ e $K'$.

Agora, a segunda situa\c c\~ao (figura \ref{subdividi} (ii)) \'e \'obvia.
\end{proof} 

\begin{proof}[Demonstra\c c\~ao do teorema \ref{teo1} usando apenas o m\'etodo de Cauchy] 

Sendo um pol\'igono $K$ no plano $Q \cong \R^2$, triangulado e homeomorfo a um disco $D$, 
com poss\'iveis identifica\c c\~oes dos simplexos 
na borda  de $K$, o que denotamos por $K_0$,
a prova consiste em seis seguintes etapas.

\medskip 
1) Etapa 1:  A primeira etapa \'e construir 
um ``elongamento'' de $K$ na forma de uma pir\^amide. 
Aqui, chamaremos de pir\^amide somente a parte superficial (de dimens\~ao 2)
 da pir\^amide, isto \'e, a uni\~ao das faces da pir\^amide, menos a base.

   \begin{figure}[H] 
\begin{tikzpicture}  [scale= 0.5]

% Le bord du disque
\coordinate [label=left:$b_1$] (B1) at (-6,-0.5);
\coordinate [label=below:$b_6$] (B2) at (-3.5,-4);
\coordinate [label=right:$b_5$] (B3) at (4.5,-4);
\coordinate [label=right:$b_4$] (B4) at (6,-1);
\coordinate [label=right:$b_3$] (B5) at (3.5,4);
\coordinate [label=above:$b_2$] (B6) at (-3,4);
\draw [very thick] (B1)  -- (B2) -- (B3) -- (B4)--(B5) -- (B6) -- (B1);

% Le triangle
\coordinate [label=right:$a_2$]  (A1) at (0,2);
\coordinate [label=left:$a_1$] (A2) at (-1,0);
\coordinate [label=below:$a_3$] (A3) at (2,0);
\draw [very thick] (A1)  -- (A2) -- (A3)-- (A1) ;
\coordinate [label=right:$0$]  (0) at (0.3,0.7);
 \foreach \point in {0}
    \fill [black,opacity=.5] (\point) circle (2pt);

% Les points rouge
\coordinate [label=below:$y_2$] (X1) at (0.8,-2);
\coordinate  [label=left:$y_1$] (X2) at (-2,1);

\draw (B1) -- (X2);
\draw (A2) -- (X2);
\draw (A1) -- (X2);
\draw (B6) -- (X2);
\draw (B2) -- (X2);
\draw (B6) -- (A1);

\draw (A2) -- (X1);
\draw (A3) -- (X1);
\draw (B2) -- (X1);
\draw (B3) -- (X1);
\draw (B3) -- (A3);
\draw (B4) -- (A3);
\draw (B5) -- (A3);
\draw (B5) -- (A1);
\draw (B2) -- (A2);

 \end{tikzpicture}
\caption{A triangula\c c\~ao do pol\'igono $K$.}\label{Pyramide70}
\end{figure}
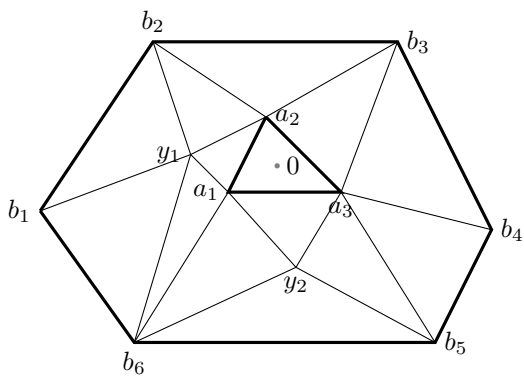

Podemos supor que a origem $0$ de $\R^2$ est\'a em um simplexo $\sigma_0$ 
de dimens\~ao 2 (um tri\^angulo)  no interior do pol\'igono $K$. 
Dentro de $\R^2$ consideramos a m\'etrica Euclidiana. 
Podemos supor que as dist\^ancias da origem aos v\'ertices 
da triangula\c c\~ao sejam diferentes, 
pois caso contr\'ario, isso pode ser realizado por uma pequena perturba\c c\~ao 
sem mudar a estrutura do complexo simplicial 
e a demonstra\c c\~ao do teorema mante-se v\'alida.  

Denotamos por $a_1, a_2, a_3$ os v\'ertices
de $\sigma_0$ e  por $b_1, \ldots , b_k$ os v\'ertices da  triangula\c c\~ao de $K_0$. 
Os  demais outros v\'ertices est\~ao denotados
da seguinte maneira: Chamamos de $y_1$ o  v\'ertice mais perto de $0$ 
e $y_{2}, \ldots, y_{n}$ os 
v\'ertices na ordem crescente da dist\^ancia de $0$. 
Podemos supor tamb\'em que a dist\^ancia da origem ao v\'ertice  
$y_1$ \'e maior que as dist\^ancias da origem aos v\'ertices de $\sigma_0$.

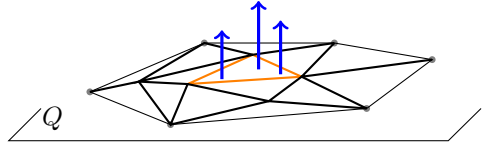
\begin{figure}[H] 
\begin{tikzpicture}  [scale= 0.65]
% Le plan en bas
\draw (-3.76,-3.2) node {$Q$};
\draw (-4,-3) --(-4.66, -3.66) -- (4.33,-3.66) -- (5,-3);
\coordinate (Y1) at (-3,-2.66);
\coordinate (Y2) at (-0.66,-1.66);
\coordinate  (Y3) at (2,-1.66);
\coordinate  (Y4) at (4,-2);
\coordinate  (Y5) at (2.66,-3);
\coordinate   (Y6) at (-1.35,-3.33);
\draw (Y1) -- (Y2) -- (Y3) -- (Y4) --  (Y5) -- (Y6) -- (Y1);
 \foreach \point in {Y1,Y2,Y3,Y4,Y5,Y6}
    \fill [black,opacity=.5] (\point) circle (2pt);
    
%La triangulation en bas
\coordinate (A1) at (-1,-2.5);
\coordinate (A2) at (0.33,-1.9);
\coordinate  (A3) at (1.33,-2.35);
\coordinate  (X1b) at (0.66,-2.85);
\coordinate  (X2b) at (-2,-2.45);
\draw  [thick,orange](A1) -- (A2) -- (A3)--(A1);
\draw[thick] (Y3)--(A2)-- (Y2) -- (X2b) --(A2);
\draw[thick] (Y1)--(X2b)-- (A1) -- (Y6) --(X2b);
\draw[thick] (A1)--(X1b)-- (Y6) ;
\draw[thick] (Y5)--(X1b)-- (A3) -- (Y5);
\draw[thick] (Y4)--(A3) -- (Y3);

\draw[very thick,blue,->] (-0.3,-2.4) -- (-0.3,-1.4);
\draw[very thick,blue,->] (0.45,-2.2) -- (0.45,-0.8);
\draw[very thick,blue,->] (0.9,-2.3) -- (0.9,-1.2);

 \end{tikzpicture}
\caption{O levantamento do pol\'igono.}\label{a}
\end{figure}

%\vfill
%\break 

Vamos construir uma pir\^amide $\Pi$ em $\R^3$, acima de $K$, fixando 
a borda $K_0$ como sendo a base da pir\^amide  no plano $Q= \R^2$ horizontal 
em $\R^3$.  

Para $i=0,\ldots, n+1$, consideramos os planos $P_i$ paralelos  e t\^em dist\^ancias $n-i+1$ 
relativamente ao plano de base $Q$ em $\R^3$. 
Agora, para $i=1,\ldots, n$, chamamos de $x_i$ a proje\c c\~ao ortogonal do ponto $y_i$ 
sobre o plano $P_i$.  No plano 
$P_{0}$, denotamos por $u_1, u_2, u_3$ 
as proje\c c\~oes ortogonais dos pontos  $a_1, a_2, a_3$
(veja figura \ref{Pyramide5} (i)).

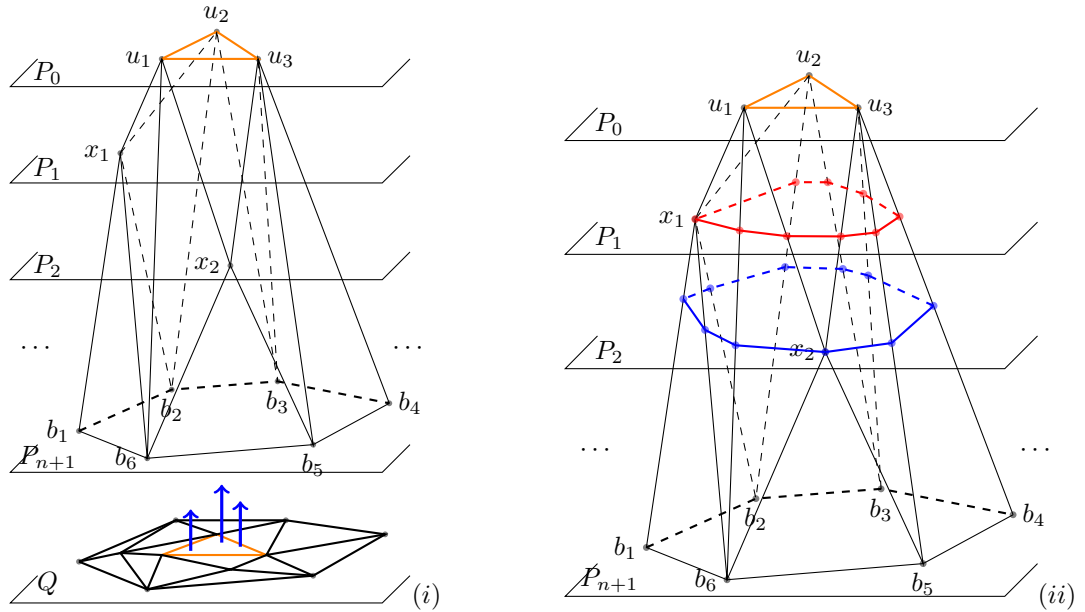
\begin{figure}[H] 
\begin{tikzpicture}  [scale= 0.55]

%Le triangle en haut
\coordinate [label=left:$u_1$]  (A1) at (-1,8);
 \coordinate [label=above:$u_2$]  (A2) at (0.33,8.66);
 \coordinate [label=right:$u_3$]  (A3) at (1.33,8);
 \foreach \point in {A1,A2,A3}
    \fill [black,opacity=.5] (\point) circle (2pt);
\draw[thick,orange]  (A1) -- (A2) -- (A3)-- (A1);

% Le ``disque" en bas
 \coordinate [label=left:$b_1$]  (B1) at (-3,-1);
 \coordinate [label=below:$b_2$]  (B2) at (-0.76,0);
 \coordinate [label=below:$b_3$]  (B3) at (1.8,0.2);
  \coordinate [label=right:$b_4$]  (B4) at (4.5,-0.33);
 \coordinate [label=below:$b_5$]  (B5) at (2.66,-1.33);
 \coordinate [label= left:$b_6$]  (B6) at (-1.35,-1.66);

 \foreach \point in {B1,B2,B3,B4,B5,B6}
    \fill [black,opacity=.5] (\point) circle (2pt);
\draw [thick, dashed] (B1)  -- (B2) -- (B3) -- (B4);
\draw  (B6) -- (B5) -- (B4);

% Les points x_1 et x_2
 \coordinate [label=left:$x_2$]  (X2) at (0.66,3);
 \coordinate [label=left:$x_1$]  (X1) at (-2,5.72);
 \foreach \point in {X1,X2}
    \fill [black,opacity=.5] (\point) circle (2pt);

%Les segments (en noir) devant et derri�re
\draw (X1) -- (B1) -- (B6) -- (X1)-- (A1) -- (B6) -- (X2) -- (A1);
\draw (B4) -- (A3)-- (B5) -- (X2) -- (A3);
\draw [dashed] (A2) -- (B2) -- (X1) -- (A2) -- (B3) -- (A3);

% Les plans
\draw (-3.76,7.66) node {$P_{0}$};
\draw (-4,8) --(-4.66, 7.33) -- (4.33,7.33) -- (5,8);

\draw (-3.76,5.33) node {$P_1$};
\draw (-4,5.66) --(-4.66, 5) -- (4.33,5) -- (5,5.66);

\draw (-3.76,3) node {$P_2$};
\draw (-4,3.33) --(-4.66, 2.66) -- (4.33,2.66) -- (5,3.33);

\draw(-4,1) node{$\cdots$};
\draw(5,1) node{$\cdots$};

\draw (-3.76,-1.66) node {$P_{n+1}$};
\draw (-4,-1.33) --(-4.66, -2) -- (4.33,-2) -- (5,-1.33);

% Le plan en bas
\draw (-3.76,-4.7) node {$Q$};
\draw (-4,-4.5) --(-4.66, -5.16) -- (4.33,-5.16) -- (5,-4.5);
\coordinate (Y1) at (-3,-4.16);
\coordinate (Y2) at (-0.66,-3.16);
\coordinate  (Y3) at (2,-3.16);
\coordinate  (Y4) at (4.4,-3.5);
\coordinate  (Y5) at (2.66,-4.5);
\coordinate   (Y6) at (-1.35,-4.83);
\draw [thick] (Y1) -- (Y2) -- (Y3) -- (Y4) --  (Y5) -- (Y6) -- (Y1);
 \foreach \point in {Y1,Y2,Y3,Y4,Y5,Y6}
    \fill [black,opacity=.5] (\point) circle (2pt);
    
%La triangulation en bas
\coordinate (A1) at (-1,-4);
\coordinate (A2) at (0.33,-3.5);
\coordinate  (A3) at (1.53,-4);
\coordinate  (X1b) at (0.66,-4.35);
\coordinate  (X2b) at (-2,-3.95);
\draw  [thick,orange](A1) -- (A2) -- (A3)--(A1);
\draw[thick] (Y3)--(A2)-- (Y2) -- (X2b) --(A2);
\draw[thick] (Y1)--(X2b)-- (A1) -- (Y6) --(X2b);
\draw[thick] (A1)--(X1b)-- (Y6) ;
\draw[thick] (Y5)--(X1b)-- (A3) -- (Y5);
\draw[thick] (Y4)--(A3) -- (Y3);

% les fleches
\draw[very thick,blue,->] (-0.3,-3.9) -- (-0.3,-2.9);
\draw[very thick,blue,->] (0.45,-3.7) -- (0.45,-2.3);
\draw[very thick,blue,->] (0.9,-3.8) -- (0.9,-2.7);

\draw (5.4,-5) node {$(i)$};
 \end{tikzpicture}
%\caption{A pir\^amide $\Pi$.}\label{Pyramide4}
\quad\quad\quad\quad
\begin{tikzpicture} [scale= 0.65]

 \coordinate [label=left:$u_1$]  (A1) at (-1,8);
 \coordinate [label=above:$u_2$]  (A2) at (0.33,8.66);
 \coordinate [label=right:$u_3$]  (A3) at (1.33,8);
 \foreach \point in {A1,A2,A3}
    \fill [black,opacity=.5] (\point) circle (2pt);
\draw[thick,orange]  (A1) -- (A2) -- (A3)-- (A1);

% Le ``disque" en bas
 \coordinate [label=left:$b_1$]  (B1) at (-3,-1);
 \coordinate [label=below:$b_2$]  (B2) at (-0.76,0);
 \coordinate [label=below:$b_3$]  (B3) at (1.8,0.2);
  \coordinate [label=right:$b_4$]  (B4) at (4.5,-0.33);
 \coordinate [label=below:$b_5$]  (B5) at (2.66,-1.33);
 \coordinate [label= left:$b_6$]  (B6) at (-1.35,-1.66);

 \foreach \point in {B1,B2,B3,B4,B5,B6}
    \fill [black,opacity=.5] (\point) circle (2pt);
\draw [thick, dashed] (B1)  -- (B2) -- (B3) -- (B4);
\draw  (B6) -- (B5) -- (B4);

% Les points x_1 et x_2
 \coordinate [label=left:$x_2$]  (X2) at (0.66,3);
 \coordinate [label=left:$x_1$]  (X1) at (-2,5.72);
 \foreach \point in {X1,X2}
    \fill [black,opacity=.5] (\point) circle (2pt);

%Les segments (en noir) devant et derri�re
\draw (X1) -- (B1) -- (B6) -- (X1)-- (A1) -- (B6) -- (X2) -- (A1);
\draw (B4) -- (A3)-- (B5) -- (X2) -- (A3);
\draw [dashed] (A2) -- (B2) -- (X1) -- (A2) -- (B3) -- (A3);

% Les plans
\draw (-3.76,7.66) node {$P_{0}$};
\draw (-4,8) --(-4.66, 7.33) -- (4.33,7.33) -- (5,8);

\draw (-3.76,5.33) node {$P_1$};
\draw (-4,5.66) --(-4.66, 5) -- (4.33,5) -- (5,5.66);

\draw (-3.76,3) node {$P_2$};
\draw (-4,3.33) --(-4.66, 2.66) -- (4.33,2.66) -- (5,3.33);

\draw(-4,1) node{$\cdots$};
\draw(5,1) node{$\cdots$};

\draw (-3.76,-1.66) node {$P_{n+1}$};
\draw (-4,-1.33) --(-4.66, -2) -- (4.33,-2) -- (5,-1.33);

\coordinate(J1) at (8,3.16);
\coordinate (C1) at (intersection of A1--B6 and X1--J1);
\draw [thick,red](X1) -- (C1);

\coordinate  (J2) at (8,4.4);
\coordinate(C2) at (intersection of A1--X2 and C1--J2);
\draw [thick,red](C1) -- (C2);

\coordinate (J3) at (4,5.35);
\coordinate (C3) at (intersection of A3--X2 and C2--J3);
\draw [thick,red](C2) -- (C3);

\coordinate (J4) at (4,5.7);
\coordinate (C4) at (intersection of A3--B5 and C3--J4);
\draw [thick,red](C3) -- (C4);

\coordinate (J5) at (4,7);
\coordinate (C5) at (intersection of A3--B4 and C4--J5);
\draw [thick,red](C4) -- (C5);

\coordinate (J6) at (-2,8.4);
\coordinate (C6) at (intersection of A3--B3 and C5--J6);
\draw [thick,dashed,red](C5) -- (C6);

\coordinate (J7) at (-4,8);
\coordinate (C7) at (intersection of A2--B3 and C6--J7);
\draw [thick,dashed,red](C6) -- (C7);

\coordinate (J8) at (-8,6.5);
\coordinate (C8) at (intersection of A2--B2 and C7--J8);
\draw [thick,dashed,red](C7) -- (C8);

\draw [thick,dashed,red](C8) -- (X1);

 \foreach \point in {X1,C1,C2,C3,C4,C5,C6,C7,C8}
    \fill [red,opacity=.5] (\point) circle (2.3pt);
    
% jles points bleus  $Di$

\coordinate (K1) at (8,4);
\coordinate (D1) at (intersection of A3--B5 and X2--K1);
\draw [thick,blue](X2) -- (D1);

\coordinate (K2) at (8,8.5);
\coordinate  (D2) at (intersection of A3--B4 and D1--K2);
\draw [thick,blue](D1) -- (D2);

\coordinate (K3) at (-8,9);
\coordinate (D3) at (intersection of A3--B3 and D2--K3);
\draw [thick,dashed,blue](D2) -- (D3);

\coordinate (K4) at (-8,7);
\coordinate (D4) at (intersection of A2--B3 and D3--K4);
\draw [thick,dashed,blue](D3) -- (D4);

\coordinate (K5) at (-8,5);
\coordinate (D5) at (intersection of A2--B2 and D4--K5);
\draw [thick,dashed,blue](D4) -- (D5);

\coordinate (K6) at (-8,2.5);
\coordinate (D6) at (intersection of X1--B2 and D5--K6);
\draw [thick,dashed,blue](D5) -- (D6);

\coordinate (K7) at (-8,1.8);
\coordinate (D7) at (intersection of X1--B1 and D6--K7);
\draw [thick,dashed,blue](D6) -- (D7);

\coordinate (K8) at (8,-10.5);
\coordinate (D8) at (intersection of X1--B6  and D7--K8);
\draw [thick,blue](D7) -- (D8);

\coordinate (K9) at (8,-1.5);
\coordinate (D9) at (intersection of A1--B6 and D8--K9);
\draw [thick,blue](D8) -- (D9);

\draw [thick,blue](D9) -- (X2);

 \foreach \point in {X2,D1,D2,D3,D4,D5,D6,D7,D8,D9}
    \fill [blue,opacity=.5] (\point) circle (2.3pt);
    
    \draw (5.4,-2) node {$(ii)$};
    
 \end{tikzpicture}
 \bigskip
 \caption{A pir\^amide $\Pi$ e a decomposi\c c\~ao $L'$ da pir\^amide $\Pi$.}\label{Pyramide5}
\end{figure}

A triangula\c c\~ao do pol\'igono 
induz uma triangula\c c\~ao $L$ na pir\^amide $\Pi$ 
levantando cada 
simplexo $[b_i, y_j]$ em $[b_i, x_j]$, cada $[y_i, y_j]$ em $[x_i, x_j]$ e cada $[y_i, a_j]$ em 
$[x_i, u_j]$. Da mesma maneira, levantamos os simplexos de dimens\~ao 2. 

\bigskip 
%\vfill\break

2) Etapa 2: Constru\'imos uma sub-decomposi\c c\~ao $L'$ da triangula\c c\~ao $L$ 
da pir\^amide $\Pi$, tal que a interse\c c\~ao dos planos $P_i$ com a pir\^amide seja triangulada da seguinte maneira: 
Definimos novos v\'ertices 
de $L'$ como interse\c c\~oes de $1$-simplexos de $L$ com planos $P_i$.
 Da mesma maneira, 
definimos tamb\'em novos $1$-simplexos de $L'$ 
como interse\c c\~oes de $2$-simplexos de $L$ com planos $P_i$. 
A decomposi\c c\~ao $L'$ da pir\^amide 
cont\'em v\'ertices, arestas (ou $1$-simplexos), 
e faces que podem ser tri\^angulos ou quadril\'ateros (veja figura 
\ref{Pyramide5} (ii)).

3) Etapa 3: Definamos uma sub-triangula\c c\~ao $L''$ de $L$ da seguinte maneira: 
 Usando o lema 
  \ref{subtriang} com $\tau= P_i \cap \sigma$ para cada $2$-simplexo $\sigma$ de $L$ 
e cada plano $P_i$, obtemos  uma sub-triangula\c c\~ao $L''$ de $L$ de tal modo que todos os 
v\'ertices de $L''$ est\~ao situados nos planos $P_i$. Al\'em disso, o lema \ref{subtriang} implica que a soma
$n_0 -n_1 + n_2$ \'e a mesma para as triangula\c c\~oes $L$ e $L''$ (veja figura \ref{Pyramide7}). 

\begin{figure}[H] 
\begin{tikzpicture} [scale= 0.8]

 \coordinate [label=left:$u_1$]  (A1) at (-1,8);
 \coordinate [label=above:$u_2$]  (A2) at (0.33,8.66);
 \coordinate [label=right:$u_3$]  (A3) at (1.33,8);
 \foreach \point in {A1,A2,A3}
    \fill [black,opacity=.5] (\point) circle (2pt);
\draw[thick,orange]  (A1) -- (A2) -- (A3)-- (A1);

% Le ``disque" en bas
 \coordinate [label=left:$b_1$]  (B1) at (-3,-1);
 \coordinate [label=below:$b_2$]  (B2) at (-0.76,0);
 \coordinate [label=below:$b_3$]  (B3) at (1.8,0.2);
  \coordinate [label=right:$b_4$]  (B4) at (4.5,-0.33);
 \coordinate [label=below:$b_5$]  (B5) at (2.66,-1.33);
 \coordinate [label= left:$b_6$]  (B6) at (-1.35,-1.66);

 \foreach \point in {B1,B2,B3,B4,B5,B6}
    \fill [black,opacity=.5] (\point) circle (2pt);
\draw [thick, dashed] (B1)  -- (B2) -- (B3) -- (B4);
\draw  (B6) -- (B5) -- (B4);

% Les points x_1 et x_2
 \coordinate [label=left:$x_2$]  (X2) at (0.66,3);
 \coordinate [label=left:$x_1$]  (X1) at (-2,5.72);
 \foreach \point in {X1,X2}
    \fill [black,opacity=.5] (\point) circle (2pt);

%Les segments (en noir) devant et derri�re
\draw (X1) -- (B1) -- (B6) -- (X1)-- (A1) -- (B6) -- (X2) -- (A1);
\draw (B4) -- (A3)-- (B5) -- (X2) -- (A3);
\draw [dashed] (A2) -- (B2) -- (X1) -- (A2) -- (B3) -- (A3); 

% Les plans
\draw (-3.76,7.66) node {$P_{0}$};
\draw (-4,8) --(-4.66, 7.33) -- (4.33,7.33) -- (5,8);

\draw (-3.76,5.33) node {$P_1$};
\draw (-4,5.66) --(-4.66, 5) -- (4.33,5) -- (5,5.66);

\draw (-3.76,3) node {$P_2$};
\draw (-4,3.33) --(-4.66, 2.66) -- (4.33,2.66) -- (5,3.33);

\draw(-4,1) node{$\cdots$};
\draw(5,1) node{$\cdots$};

\draw (-3.76,-1.66) node {$P_{n+1}$};
\draw (-4,-1.33) --(-4.66, -2) -- (4.33,-2) -- (5,-1.33);

% les points  rouges $Ci$

\coordinate(J1) at (8,3.16);
\coordinate (C1) at (intersection of A1--B6 and X1--J1);
\draw [thick,red](X1) -- (C1);

\coordinate  (J2) at (8,4.4);
\coordinate(C2) at (intersection of A1--X2 and C1--J2);
\draw [thick,red](C1) -- (C2);

\coordinate (J3) at (4,5.35);
\coordinate (C3) at (intersection of A3--X2 and C2--J3);
\draw [thick,red](C2) -- (C3);

\coordinate (J4) at (4,5.7);
\coordinate (C4) at (intersection of A3--B5 and C3--J4);
\draw [thick,red](C3) -- (C4);

\coordinate (J5) at (4,7);
\coordinate (C5) at (intersection of A3--B4 and C4--J5);
\draw [thick,red](C4) -- (C5);

\coordinate (J6) at (-2,8.4);
\coordinate (C6) at (intersection of A3--B3 and C5--J6);
\draw [thick,dashed,red](C5) -- (C6);

\coordinate (J7) at (-4,8);
\coordinate (C7) at (intersection of A2--B3 and C6--J7);
\draw [thick,dashed,red](C6) -- (C7);

\coordinate (J8) at (-8,6.5);
\coordinate (C8) at (intersection of A2--B2 and C7--J8);
\draw [thick,dashed,red](C7) -- (C8);

\draw [thick,dashed,red](C8) -- (X1);

% jles points bleus  $Di$

\coordinate (K1) at (8,4);
\coordinate (D1) at (intersection of A3--B5 and X2--K1);
\draw [thick,blue](X2) -- (D1);

\coordinate (K2) at (8,8.5);
\coordinate  (D2) at (intersection of A3--B4 and D1--K2);
\draw [thick,blue](D1) -- (D2);

\coordinate (K3) at (-8,9);
\coordinate (D3) at (intersection of A3--B3 and D2--K3);
\draw [thick,dashed,blue](D2) -- (D3);

\coordinate (K4) at (-8,7);
\coordinate (D4) at (intersection of A2--B3 and D3--K4);
\draw [thick,dashed,blue](D3) -- (D4);

\coordinate (K5) at (-8,5);
\coordinate (D5) at (intersection of A2--B2 and D4--K5);
\draw [thick,dashed,blue](D4) -- (D5);

\coordinate (K6) at (-8,2.5);
\coordinate (D6) at (intersection of X1--B2 and D5--K6);
\draw [thick,dashed,blue](D5) -- (D6);

\coordinate (K7) at (-8,1.8);
\coordinate (D7) at (intersection of X1--B1 and D6--K7);
\draw [thick,dashed,blue](D6) -- (D7);

\coordinate (K8) at (8,-10.5);
\coordinate (D8) at (intersection of X1--B6  and D7--K8);
\draw [thick,blue](D7) -- (D8);

\coordinate (K9) at (8,-1.5);
\coordinate (D9) at (intersection of A1--B6 and D8--K9);
\draw [thick,blue](D8) -- (D9);

\draw [thick,blue](D9) -- (X2);

%Les "diagonales"

\draw [thick,green](A1) -- (C3);
%\draw [thick,green](C1) -- (D8);
\draw [thick,green](X1) -- (D9);
\draw [thick,green](C1) -- (X2);
\draw [thick,green](C4) -- (X2);
\draw [thick,green](C5) -- (D1);
\draw [thick,green](D7) -- (B6);
\draw [thick,green](D2) -- (B5);

\draw [thick,dashed,green](A3) -- (C7);
\draw [thick,dashed,green](C6) -- (D2);
\draw [thick,dashed,green](C7) -- (D3);
\draw [thick,dashed,green](C8) -- (D4);
\draw [thick,dashed,green](C8) -- (D6);
\draw [thick,dashed,green](D3) -- (B4);
\draw [thick,dashed,green](D4) -- (B2);
\draw [thick,dashed,green](D6) -- (B1);

 \end{tikzpicture}
\caption{A sub-triangula\c c\~ao $L''$ da pir\^amide $\Pi$.}\label{Pyramide7}
\end{figure}
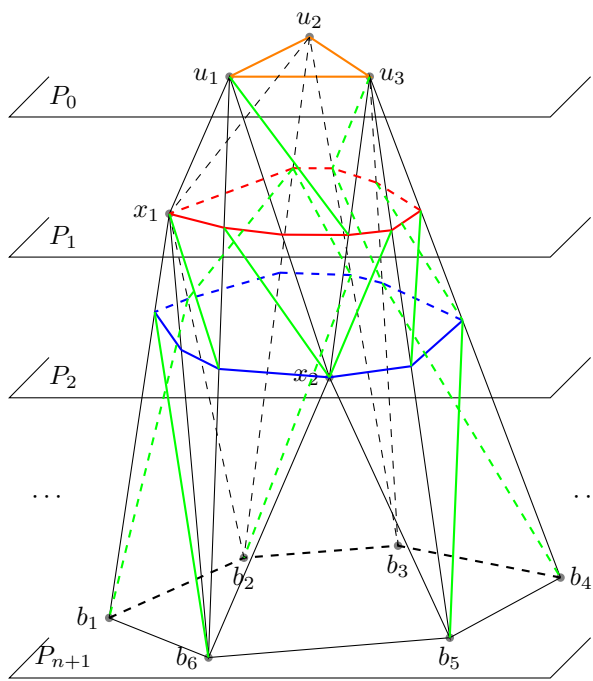

A proje\c c\~ao de $L''$ sobre o plano $P_0$  induz uma sub-triangula\c c\~ao $K''$ da triangula\c c\~ao 
$K$ (figura \ref{Pyramide71}). 

\begin{figure}[H] 
\begin{tikzpicture}  [scale= 0.5]

% Le bord du disque
\coordinate [label=left:$b_1$] (B1) at (-6,-0.5);
\coordinate [label=below:$b_6$] (B2) at (-3.5,-4);
\coordinate [label=right:$b_5$] (B3) at (4.5,-4);
\coordinate [label=right:$b_4$] (B4) at (6,-1);
\coordinate [label=right:$b_3$] (B5) at (3.5,4);
\coordinate [label=above:$b_2$] (B6) at (-3,4);
\draw [very thick] (B1)  -- (B2) -- (B3) -- (B4)--(B5) -- (B6) -- (B1);

% Le triangle
\coordinate [label=right:$a_2$]  (A1) at (0,2);
\coordinate [label=left:$a_1$] (A2) at (-1,0);
\coordinate [label=below:$a_3$] (A3) at (2,0);
\draw [very thick] (A1)  -- (A2) -- (A3)-- (A1) ;

% Les points rouge
\coordinate [label=below:$y_2$] (X1) at (0.8,-2);
\coordinate  [label=left:$y_1$] (X2) at (-2,1);
\coordinate (C11) at (-1.5,-1.5);
\coordinate (C1) at (intersection of X2--C11 and A2--B2);
\coordinate (C21) at (0,-0.7);
\coordinate (C2) at (intersection of C1--C21 and A2--X1);
\coordinate  (C31) at (2,-0.68);
\coordinate  (C3) at (intersection of C2--C31 and A3--X1);
\coordinate (C41) at (3,-1.1);
\coordinate (C4) at (intersection of C3-- C41 and A3--B3);
\coordinate (C51) at (3,-0.3);
\coordinate (C61) at (2.5,1.5);
\coordinate (C71) at (1.5,2.7);
\coordinate (C81) at (-1,2.5);
\coordinate (C5) at (intersection of C4-- C51 and A3--B4);
\coordinate (C6) at (intersection of C5-- C61 and A3--B5);
\coordinate (C7) at (intersection of C6-- C71 and A1--B5);
\coordinate (C8) at (intersection of C7-- C81 and A1--B6);

\draw [very thick,red] (X2) -- (C1)  -- (C2) -- (C3) -- (C4) -- (C5) -- (C6) -- (C7) -- (C8) -- (X2);

% Les points bleus
%\coordinate (X1) at (1,-1.5);
\coordinate (X1) at (0.8,-2);
\coordinate (D11) at (3.2,-2);
\coordinate (D21) at (4.2,-0.5);
\coordinate (D31) at (3,3);
\coordinate (D41) at (2.5,3.5);
%\coordinate (D5) at (-2.3,3.2);
\coordinate (D61) at (-2.1,3.2);
\coordinate (D6) at (-3.2,2.5);
\coordinate (D51) at (-2.3,3.2);
\coordinate (D71) at (-4.55,0.2);
\coordinate (D81) at (-3.2,-1.2);
\coordinate (D91) at (-2.8,-2.5);
\coordinate (D9) at (intersection of X1--D91 and A2--B2);
\coordinate (D8) at (intersection of D9--D81 and X2--B2);
\coordinate (D7) at (intersection of D8--D71 and X2--B1);
\coordinate (D6) at (intersection of D7--D51 and X2--B6);
\coordinate (D5) at (intersection of D6--D61 and A1--B6);
\coordinate (D4) at (intersection of D5--D41 and A1--B5);
\coordinate (D3) at (intersection of D4--D31 and A3--B5);
\coordinate (D2) at (intersection of D3--D21 and A3--B4);
\coordinate (D1) at (intersection of D2--D11 and A3--B3);
\draw [very thick,blue] (X1) -- (D1)  -- (D2) -- (D3) -- (D4)--(D5) -- (D6) -- (D7) -- (D8) -- (D9)--(X1);

\draw (B1) -- (X2);
\draw (A2) -- (X2);
\draw (A1) -- (X2);
\draw (B6) -- (X2);
\draw (B2) -- (X2);
\draw (B6) -- (A1);

\draw (A2) -- (X1);
\draw (A3) -- (X1);
\draw (B2) -- (X1);
\draw (B3) -- (X1);
\draw (B3) -- (A3);
\draw (B4) -- (A3);
\draw (B5) -- (A3);
\draw (B5) -- (A1);
\draw (B2) -- (A2);

\draw [thick,green](A2) -- (C3);
\draw [thick,green](A3) -- (C7);
%\draw [thick,green](C1) -- (D8);
\draw [thick,green](X2) -- (D9);
\draw [thick,green](C1) -- (X1);
\draw [thick,green](C4) -- (X1);
\draw [thick,green](C5) -- (D1);
\draw [thick,green](C6) -- (D2);
\draw [thick,green](C7) -- (D3);
\draw [thick,green](C8) -- (D4);
\draw [thick,green](C8) -- (D6);
\draw [thick,green](D7) -- (B2);
\draw [thick,green](D2) -- (B3);
\draw [thick,green](D3) -- (B4);
\draw [thick,green](D4) -- (B6);
\draw [thick,green](D6) -- (B1);

 \end{tikzpicture}
\caption{A sub-triangula\c c\~ao $K''$ do pol\'igono.}\label{Pyramide71}
\end{figure}
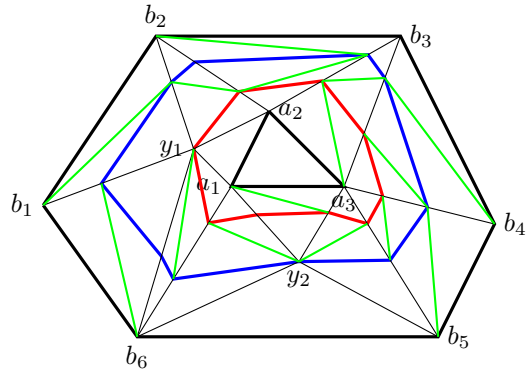

\bigskip 

 4) Etapa 4: Mostramos por indu\c c\~ao 
que a interse\c c\~ao de $L''$ com cada plano $P_i$ \'e uma curva homeomorfa a uma circunfer\^encia. 

Observamos que a proje\c c\~ao ortogonal $\pi$ 
da pir\^amide sobre $Q$ \'e uma bije\c c\~ao entre a triangula\c c\~ao $L''$ (da pir\^amide $\Pi$) e a triangula\c c\~ao $K''$ (do pol\'igono $K$). 

De fato, por constru\c c\~ao, como n\~ao tem arestas verticais na pir\^amide, 
a proje\c c\~ao $\pi : \Pi \to K$ \'e uma bije\c c\~ao entre as triangula\c c\~oes $L''$ e $K''$. 
Cada v\'ertice de $K''$ corresponde a um v\'ertice de $L''$, e da mesma maneira 
para as arestas e os tri\^angulos de $L''$ e $K''$, respectivamente. 
Isso implica que $K''$ \'e uma triangula\c c\~ao. 

Vamos mostrar que todas as interse\c c\~oes $B_i = L \cap P_i$ est\~ao homeomorfas a circunfer\^encias. 
Iniciando com $L \cap P_{0}$, que \'e a borda denotada por  
$B_{0}$ do tri\^angulo $\sigma_0$, e que \'e homeomorfa a uma circunfer\^encia. 

Note que, por um lado, sabemos que a proje\c c\~ao $\Pi \to D$ \'e uma bije\c c\~ao, e por outro lado,
na triangula\c c\~ao  $K''$, e tamb\'em $L''$, toda aresta \'e borda de exatamente dois tri\^angulos. 
Suponhamos que $B_i$ seja homeomorfa a uma circunfer\^encia. Se $B_{i+1}$ n\~ao 
for homeomorfa a uma circunfer\^encia, ter\'a um ponto m\'ultiplo 
(como na figura \ref{duplo}).  
Neste caso,  a aresta $(a,b)$, que \'e uma aresta da $L''$, \'e a  borda de 
tr\^es  tri\^angulos (simplexos de dimens\~ao 2),
que \'e contradi\c c\~ao. 

\begin{figure}[H] 
\begin{tikzpicture} [scale= 0.6]
% plano $P_{i-1}$
\draw (1, 1.5) -- (0,0) -- (8,0) -- (9, 1.5);
\node at (1,0.5) {$P_{i+1}$};
% plano $P_{i}$
\draw (1, 5.5) -- (0,4) -- (8,4) -- (9, 5.5);
\node at (1,4.5) {$P_{i}$};

%b - b_1 - x_1

\draw[thick,densely dotted] (4, 5) -- (3.7619,4); 
\draw (3.7619,4) -- (3, 0.8); 
\node at (4,5) [above]{$b$};
%\node at (3.8,3.8) [left]{$b_1$};
%\node at (3, 0.8) [below]{$x_1$};

%b - b_2 - a

\draw[thick,densely dotted] (4, 5) -- (4.25,4); 
\draw (4.25,4) -- (5, 1); 
%\node at (3.9,3.8) [right]{$b_2$};
\node at (5,1) [below]{$a$};

%b - b_3 - x_2

\draw[thick,densely dotted] (4, 5) -- (4.6666,4); 
\draw (4.6666,4) -- (7, 0.5); 
%\node at (4.4,3.8) [right]{$b_3$};
%\node at (7,0.5) [below]{$x_2$};

%b - b_4 - x_3
\draw[thick,densely dotted] (4, 5) -- (5.0811,4); 
\draw (5.0811,4) -- (8, 1); 
%\node at (4.9,3.8) [right]{$b_4$};
%\node at (8,1) [below]{$x_3$};

%x_1 -- a
\draw (3, 0.8) -- (5,1); 

% a -- x_2
\draw (5,1) -- (7,0.5); 

% a -- x_3 
\draw[thick,densely dotted] (5,1) -- (6.6667, 1); 
\draw (6.6667, 1) -- (8,1); 

\end{tikzpicture}
\bigskip
\caption{Figura n\~ao admissivel: A interse\c c\~ao $B_i$ de $L$ com cada plano $P_i$ n\~ao tem ponto  m\'ultiplo.} \label{duplo}
\end{figure}
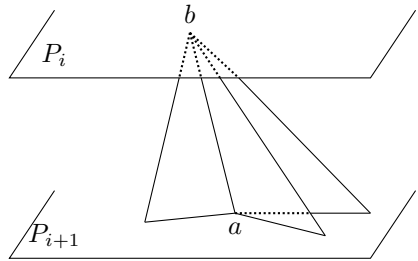

5) Etapa 5: Aplicar o m\'etodo de Cauchy  (se\c c\~ao \ref{methode_Cauchy}): Agora, vamos 
aplicar o m\'etodo de Cauchy na pir\^amide iniciando removendo o tri\^angulo $\sigma_0$.
Suponhamos que j\'a removemos todos tri\^angulos acima do plano $P_i$. 
Vamos mostrar que removendo todos os tri\^angulos na faixa situada entre $P_i$ e $P_{i+1}$, 
para $i=0, \ldots, n$, n\~ao muda a soma $n_0 - n_1 + n_2$. Isso se pode fazer 
pois a faixa (aberta) entre $B_{i}$ e $B_{i+1}$ n\~ao possui v\'ertices, {como seguinte:}
 
Escolhemos um tri\^angulo $(\alpha_0, \alpha_1,\beta_0)$  da faixa entre $B_{i}$ e $B_{i+1}$, 
onde os v\'ertices $\alpha_0$ e $\alpha_1$ pertencem a $B_i$ e $\beta_0$ pertence a $B_{i+1}$. 
Primeiramente, removemos o tri\^angulo $(\alpha_0, \alpha_1,\beta_0)$ pela ``opera\c c\~ao I'', 
sem mudar a soma $n_0 - n_1 + n_2$. Agora, a aresta  $(\alpha_1,\beta_0)$ tamb\'em \'e  a 
borda seja de um tri\^angulo $(\alpha_1,\beta_0, \beta_1)$, com $\beta_1 \in B_{i+1}$
(veja figura \ref{situa} (a)), 
seja de um tri\^angulo $(\alpha_1, \alpha_2,\beta_0)$, com $\alpha_2 \in B_{i}$ (veja figura \ref{situa} (b)). 
No primeiro caso, o tri\^angulo $(\alpha_1,\beta_0, \beta_1)$ pode ser removido pela ``opera\c c\~ao I'' 
de Cauchy, e no segundo caso o tri\^angulo $(\alpha_1, \alpha_2,\beta_0)$ 
pode ser removido pela ``opera\c c\~ao II'' de Cauchy. Nos dois casos, 
a soma $n_0 - n_1 + n_2$ n\~ao est\'a alterada.

\begin{figure} [H]

 \begin{tikzpicture}  [scale= 0.55]
 
 \node (1) at (-6,2.5) [rotate=340] {$\Rightarrow$};
\coordinate (a) at (-5.3,4) ;
\coordinate (ak) at (-3,03.45);
\coordinate (a0) at (0,3.2) ;
\coordinate (a1) at (2,3.5) ;
\coordinate (a2) at (4.4,4.45) ;

\node (Bi) at (7.5,4.45) {$B_i$};
          
\coordinate (b) at (-5,1);
\coordinate (bj) at (-3,0.45); 
\coordinate (b0) at (0,0.2) ;
\coordinate (b1) at (2,0.5) ;
\coordinate (b2) at (4.6,1.45) ;     

\node (Bi1) at (7.5,1.45) {$B_{i+1}$};

\node (2) at (5.5,3.2)[rotate=30]  {$\Rightarrow$};

 \fill[fill=pink]   (a0) -- (b0)-- (a1) -- (a0);
 
\draw[thick] plot[smooth] coordinates {(-6,4.4)(-5.3,4) (-3,03.45) (0,3.2) (2,3.5) (4.4,4.45) (6,5.3)};     
\draw[thick] plot[smooth] coordinates {(-6.2,1.5)(-5,1) (-3,0.45) (0,0.2) (2,0.5) (4.6,1.45) (6,2.3)};     

\coordinate (x1) at (-6,6.5);
\coordinate (x2) at (-3,5.8);
\coordinate (y) at (0,5.5);
\coordinate (x3) at (3,6.3);
\coordinate (x4) at (6,7.5);
 \fill[fill=blue!20]  (x1) -- (-6,4.4)-- (-5.3,4)-- (-3,03.45) --(0,3.2) -- (2,3.5)-- (4.4,4.45) --(6,5.3)-- 
 (x4) -- (x3) -- (y) -- (x2) -- (x1) ;

\draw (a) -- (b) --(ak) --(bj) ;
 \draw[very thick,red] (a0) -- (bj);
\draw (a0) -- (b0) -- (a1) -- (b2) -- (a2);
\draw (ak) -- (bj);
\draw (a1) -- (b1);

\filldraw (a) node {$\bullet$} ;
\filldraw (ak) node {$\bullet$} node [above] {$\alpha_k$};
\filldraw (a0) node {$\bullet$} node [above] {$\alpha_0$};
\filldraw (a1) node {$\bullet$} node [above] {$\alpha_1$};
\filldraw (a2) node {$\bullet$} node [above] {$\alpha_2$};

\filldraw (b) node {$\bullet$} ;
\filldraw (bj) node {$\bullet$} node [below] {$\beta_j$};
\filldraw (b0) node {$\bullet$} node [below] {$\beta_0$};
\filldraw (b1) node {$\bullet$} node [below] {$\beta_1$};
\filldraw (b2) node {$\bullet$} node [below] {$\beta_2$};

\node at (4.6,-1) {(a)};
 \end{tikzpicture}
 \begin{tikzpicture} [scale= 0.55]

 \node (1) at (-6,2.5) [rotate=340]{$\Rightarrow$};
\coordinate (a) at (-5.3,4) ;
\coordinate (ak) at (-3,03.45);
\coordinate (a0) at (0,3.2) ;
\coordinate (a1) at (2,3.5) ;
\coordinate (a2) at (4.4,4.45) ;
          
\coordinate (b) at (-5,1);
\coordinate (bj) at (-3,0.45); 
\coordinate (b0) at (0,0.2) ;
\coordinate (b1) at (2,0.5) ;
\coordinate (b2) at (4.6,1.45) ;         

\node (2) at (5.5,3.2)[rotate=30] {$\Rightarrow$};

 \fill[fill=pink]   (a0) -- (b0)-- (a1) -- (a0);
 
\draw[thick] plot[smooth] coordinates {(-6,4.4)(-5.3,4) (-3,03.45) (0,3.2) (2,3.5) (4.4,4.45) (6,5.3)};     
\draw[thick] plot[smooth] coordinates {(-6.2,1.5)(-5,1) (-3,0.45) (0,0.2) (2,0.5) (4.6,1.45) (6,2.3)};     

\coordinate (x1) at (-6,6.5);
\coordinate (x2) at (-3,5.8);
\coordinate (y) at (0,5.5);
\coordinate (x3) at (3,6.3);
\coordinate (x4) at (6,7.5);
 \fill[fill=blue!20]  (x1) -- (-6,4.4)-- (-5.3,4)-- (-3,03.45) --(0,3.2) -- (2,3.5)-- (4.4,4.45) --(6,5.3)-- 
 (x4) -- (x3) -- (y) -- (x2) -- (x1) ;

\draw (a) -- (b) --(ak) --(bj) ;
 \draw[very thick,red] (ak) -- (b0);
\draw (a0) -- (b0) -- (a1);
\draw (a0) -- (b0) -- (a1);
\draw (b0) -- (a2);
\draw (ak) -- (bj);
\draw (b2)--(a2) -- (b1);

\filldraw (a) node {$\bullet$} ; 
\filldraw (ak) node {$\bullet$} node [above] {$\alpha_k$};
\filldraw (a0) node {$\bullet$} node [above] {$\alpha_0$};
\filldraw (a1) node {$\bullet$} node [above] {$\alpha_1$};
\filldraw (a2) node {$\bullet$} node [above] {$\alpha_2$};

\filldraw (b) node {$\bullet$} ;
\filldraw (bj) node {$\bullet$} node [below] {$\beta_j$};
\filldraw (b0) node {$\bullet$} node [below] {$\beta_0$};
\filldraw (b1) node {$\bullet$} node [below] {$\beta_1$};
\filldraw (b2) node {$\bullet$} node [below] {$\beta_2$};
\node at (4.6,-1) {(b)};
  \end{tikzpicture}
  
%  Situation A  \hskip 7truecm   Situa\c c\~ao B 
\caption{Indo de $B_i$ a $B_{i+1}$.}\label{situa}
\end{figure}
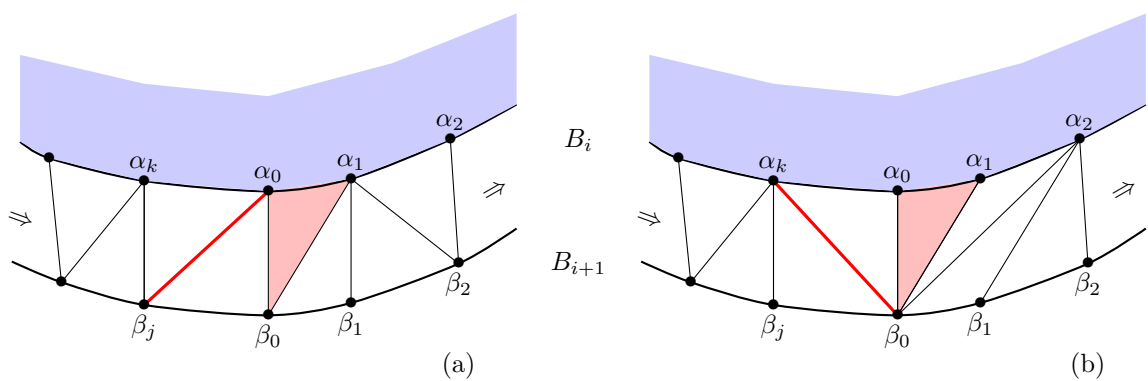

Continuamos o processo para os tri\^angulos da faixa, que s\~ao todos de um caso ou do outro, 
at\'e chegar aos \'ultimos v\'ertices de $B_i$ e $B_{i+1}$ situados 
antes de voltar a $\alpha_0$ e $\beta_0$, respectivamente. 
Chamamos estes de $\alpha_k$ e $\beta_j$. Temos duas situa\c c\~oes  poss\'iveis (a) e (b) 
(veja figura \ref{situa}). 
Na situa\c c\~ao (a), os \'ultimos tri\^angulos restantes  est\~ao  $(\alpha_k, \alpha_0, \beta_j)$ 
e $(\alpha_0, \beta_j, \beta_0)$. Neste caso os tri\^angulos $(\alpha_k, \alpha_0, \beta_j)$ 
e $(\alpha_0, \beta_j, \beta_0)$ podem ser  removidos nesta ordem  com a 
``opera\c c\~ao II'' de Cauchy.  Na situa\c c\~ao (b), os \'ultimos tri\^angulos restantes
s\~ao $(\alpha_k, \beta_j,\beta_0)$ 
e $(\alpha_k, \alpha_0, \beta_0)$. Neste caso os tri\^angulos $(\alpha_k, \alpha_0, \beta_0)$ 
e $(\alpha_k, \beta_j, \beta_0)$ podem ser  removidos nesta ordem {tamb\'em} com a 
``opera\c c\~ao II'' de Cauchy. Em todos {os dois} casos, a soma $n_0 - n_1 + n_2$ 
fica a mesma. 

Nesta prova  usamos apenas as opera\c c\~oes de Cauchy. 
\medskip 

6) Etapa 6: A conclus\~ao. 
\medskip

O processo vale at\'e a borda $B_{n+1}$ do \'ultimo buraco, que \'e a borda $K_0$ de $K$, 
e tamb\'em \'e a interse\c c\~ao  de $L''$ com o plano $Q=P_{n+1}$. 

Denotamos por $n_0^K$, $n_1^K$ e $n_2^K$ respectivamente 
os n\'umeros de v\'ertices, arestas e tri\^angulos da triangula\c c\~ao $K$ e 
 usamos a mesma nota\c c\~ao 
para a triangula\c c\~ao $K''$ do disco e a triangula\c c\~ao $K_0$  da borda. Temos 
$$n_0^K - n_1^K + n_2^K = n_0^{K''} - n_1^{K''}  + n_2^{K''} = n_0^{K_0}  - n_1^{K_0} +1.$$ 
 A primeira igualdade vem do fato que o processo de sub-triangula\c c\~ao 
n\~ao muda a soma alterada e a segunda igualdade vem do fato que 
o tri\^angulo $\sigma_0$ foi removido no in\'icio e que $n_2^{K_0} = 0$. 
Aqui tomamos em conta as identifica\c c\~oes 
dos simplexos de $K_0$.

Sendo uma triangula\c c\~ao  do pol\'igono $K$, 
o resultado n\~ao depende de escolhas efetuadas.
\end{proof}

\begin{proof}[{\bf Prova do Teorema \ref{Euler's Theorem} usando  o m\'etodo de Cauchy}]
Seja $P$ um poliedro convexo.
Prosseguimos a representa\c c\~ao planar $K$ de $P$, de acordo com o
primeiro passo da prova de Cauchy (veja a figura \ref{rep-Cauchy}), removendo 
um pol\'igono ${\mathcal F}$. 
Observe que no caso de um poliedro convexo, 
$K$ \'e um pol\'igono sem nenhuma identifica\c c\~ao de 
simplexos na sua borda $K_0$.
Ent\~ao $n_0^{K_0} - n_1^{K_0} + n_2^{K_0} = 0.$ 
O teorema \ref{teo1} implica que $n_0^{\widehat{K}} - n_1^{\widehat{K}} + n_2^{\widehat{K}} = +2$, 
tomando em considera\c c\~ao o pol\'igono removido ${\mathcal F}$ no primeiro passo da prova de Cauchy.
Como o teorema \ref{teo1} \'e provado usando apenas o m\'etodo de Cauchy, a f\'ormula de Euler tamb\'em 
\'e  provada usando apenas o m\'etodo de Cauchy.
\end{proof} 

\begin{remark} 
Como ressaltamos no in\'icio desta se\c c\~ao, usamos na demonstra\c c\~ao acima 
apenas o m\'etodo de Cauchy (se\c c\~ao \ref{methode_Cauchy}),  
sem outras ferramentas. Existem outras maneiras 
para provar o teorema \ref{Euler's Theorem}  (veja, por exemple, \cite{Epp}). 
Por\'em, nestas demonstra\c c\~oes 
s\~ao usadas ferramentas que aparecem depois da \'epoca de Cauchy.
\end{remark}

Antes de prosseguir com as aplica\c c\~oes, damos algumas observa\c c\~oes sobre a prova do Teorema \ref{teo1}.

\begin{remark}
A prova da f\'ormula de Euler usando a  proje\c c\~ao estereogr\'afica \'e um caso particular da prova.
\end{remark}

\begin{remark}
Existem outras maneiras de definir uma ordem de v\'ertices do pol\'igono (triangulado) para construir 
a pir\^amide, sem usar a dist\^ancia euclidiana em $\R^2$. 

Uma maneira poss\'ivel \'e usar a no\c c\~ao de dist\^ancia entre dois v\'ertices como
o menor n\'umero de arestas em um caminho de arestas que une estes  v\'ertices. 
Como na prova do Teorema \ref{teo1}, escolha um simplexo $\sigma_0$ de dimens\~ao $2$
 no interior do pol\'igono $K$ e defina a dist\^ancia $0$ para os tr\^es v\'ertices de 
$\sigma_0$. Decide qualquer ordem entre os v\'ertices cuja dist\^ancia aos v\'ertices de $\sigma_0$  \'e 1,
continua-se  decidindo qualquer ordem entre
v\'ertices cuja dist\^ancia aos v\'ertices de $\sigma_0$ \'e 2, etc. Em seguida, prossegue-se a 
constru\c c\~ao da pir\^amide.

Outra maneira seria iniciar a prova do Teorema \ref{Euler's Theorem}
com o poliedro convexo e ordenar as faces de  dimens\~ao $2$ 
de acordo com o processo de ``shelling'' (veja \cite{Zie} e \cite {BM}). 
As faces de  dimens\~ao $2$  do poliedro s\~ao duais dos v\'ertices do poliedro polar.
Obtemos uma ordem pelos v\'ertices do poliedro polar. Continuamos a prova usando o poliedro polar
ao lugar do poliedro original, sabendo que a soma $n_0 - n_1 + n_2$ \'e a mesma para o 
poliedro e para o seu polar. Note que o ``shelling'' \'e uma ferramenta que foi definida bem 
depois de Cauchy, portanto n\~ao \'e aceit\'avel em nosso contexto. 
Mencionamos isso apenas por uma quest\~ao de completude. 
\end{remark}

\begin{remark}
Na etapa 4 da prova, a proje\c c\~ao no plano $Q$ da interse\c c\~ao de cada plano $P_i$
com a pir\^amide \'e uma curva de Jordan passando por $y_i$.
Al\'em disso, na etapa 5 da nossa prova, deixamos claro que, se o limite do buraco estendido for
homeomorfo a uma circunf\^erencia, ent\~ao o processo de Cauchy funciona.
Neste caso, \'e poss\'ivel prosseguir na etapa 5 seja com a sub-triangula\c c\~ao $L''$ da 
pir\^amide, ou seja com a sub-triangula\c c\~ao $K''$ do pol\'igono $K$.
\end{remark}

\begin{remark}\label{rem111}
Na etapa 5 da prova, usamos apenas as opera\c c\~oes I e II de Cauchy.
Observando que, se alterarmos a ordem de remo\c c\~ao dos \'ultimos tri\^angulos restantes,
por exemplo, na situa\c c\~ao (a), se removermos o tri\^angulo $(\alpha_0, \beta_j, \beta_0)$ e, depois,
o tri\^angulo $(\alpha_k, \alpha_0, \beta_j)$, usaremos primeiro a opera\c c\~ao I de Cauchy e 
depois a opera\c c\~ao que chamamos de opera\c c\~ao III na se\c c\~ao \ref {depois_Cauchy} 
(veja a figura \ref{figuraCE} (d)).
Aqui tamb\'em n\~ao alteramos a soma $n_0 - n_1 + n_2$.
\end{remark}

\begin{remark} 
Como enfatizamos no in\'icio desta se\c c\~ao,  usamos na prova apenas o m\'etodo de Cauchy
(se\c c\~ao \ref{methode_Cauchy}) sem outras ferramentas.
Sabemos muito bem que existem maneiras ``mais modernas e mais r\'apidas'' para provar o 
teorema \ref{teo1}.
No entanto, essas provas usam ferramentas que aparecem depois da \'epoca de Cauchy. 
Em particular, algumas provas usam o Lema de Jordan (veja \cite{Jo}) que, como vimos,
aparece como um artefato em nossa prova.
\end{remark}

{\section{ Aplica\c c\~oes}\label{applications}

A seguir, usando o teorema \ref{teo1}, 
mostramos que o valor $n_0 - n_1 + n_2$ n\~ao depende da triangula\c c\~ao nos 
casos da esfera, do toro, do plano projetivo, da garrafa de Klein e  mesmo no caso de uma superf\'icie singular: o toro pin\c cado.

Em cada caso, usaremos uma representa\c c\~ao planar da superf\'icie 
sob a forma de um pol\'igono triangulado homeomorfa
a um disco com poss\'iveis identifica\c c\~oes na borda, e usaremos um lema 
``de corte''. A id\'eia de  corte foi introduzida, em geral, por Alexander Veblen em um semin\'ario em 1915 (veja \cite{Bra}). Essa id\'eia \'e bem desenvolvida no livro de Hilbert e Cohn-Vossen 
\cite{HC}, em particular para as superf\'icies que damos como exemplos.
O seguinte lema de ``de corte'' ser\'a usado nas  pr\'oximas provas.

\begin{lem} \label{lemageneral}
Seja $\mathcal T$ uma triangula\c c\~ao 
 de uma superf\'icie $S$ compacta.  
Seja $\Gamma$  uma curva cont\'inua 
simples em $S$. Existe uma sub-triangula\c c\~ao ${\mathcal T}'$ de $\mathcal T$, 
com simplexos curvilineares,  
compat\'ivel com a curva (isto \'e, $\Gamma$  \'e uni\~ao de segmentos de ${\mathcal T}'$) tal que o n\'umero 
$n_0 - n_1 + n_2$ \'e o mesmo para $\mathcal T$ e ${\mathcal T}'$. 
\end{lem}

\begin{proof} 
Antes de tudo, podemos supor que a curva $\Gamma$  seja {transversal} a todas {as} 
 arestas de $\mathcal T$, isto \'e, a interse\c c\~ao de $\Gamma$  com {cada} aresta \'e um n\'umero 
finito de pontos. Se n\~ao for o caso, uma pequena perturba\c c\~ao de $\Gamma$  
{permite obter a} transversalidade.

Escolhemos um ponto base (ponto de partida) $x_0$ sobre a curva assim que um sentido 
de percurso da curva. Se a curva n\~ao for fechada, 
definimos o ponto base como uma das extremidades da curva. 
O que segue n\~ao depende nem do ponto de partida,
nem do sentido de orienta\c c\~ao. 

A sub-triangula\c c\~ao ${\mathcal T}'$ est\'a constru\'ida simplexo por simplexo seguindo o percurso 
da curva $\Gamma$. O primeiro simplexo a ser subdividido \'e aquele 
$\sigma_0$ que cont\'em o ponto base. Seja $y$ o (primeiro) ponto de sa\'ida 
de $\sigma_0$ pela curva. O segmento (curvilinear) $(x_0, y)$ ser\'a uma aresta de ${\mathcal T}'$ assim que 
segmentos ligando $x_0$ aos v\'ertices de $\sigma_0$, um dos quais pode ser $(x_0, y)$ se 
o ponto $y$ for um v\'ertice de $\sigma_0$ (figura \ref{subdivi 0} (ii)).

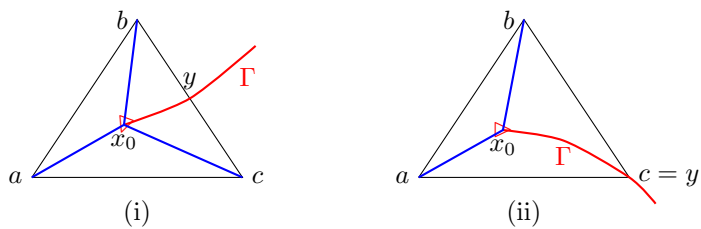
\begin{figure} [H] 
 \begin{tikzpicture}  [scale= 0.35]
  
\draw (-4,0) node[left]{$a$}   --  (0,6) node[left]{$b$} -- (4,0) node[right]{$c$} -- (-4,0);
\draw[thick,red] plot[smooth] coordinates {(-0.5,2) (2,3) (4.5,5)};     
\node at (-0.5,2) [rotate=70, red, very thick]  {$\vartriangleleft$};
\coordinate [label=above:$y$] (y) at (2,3);
\coordinate [label=below:$x_0$] (x_0) at (-0.5,2);
\draw [thick,blue](-4,0) -- (-0.5,2) -- (0,6);
\draw [thick,blue] (-0.5,2) -- (4,0);
\coordinate (u) at (-4,0);
% \foreach \point in {y}
 %   \fill [black,opacity=.5] (\point) circle (3pt);
  \draw (4.2,4.5) node[red,below] {$\Gamma$} ;
    \draw (0,-0.5) node[below] {$\rm (i)$} ;
\end{tikzpicture}
\qquad\qquad 
\begin{tikzpicture}  [scale= 0.35] 
    
\draw (-4,0) node[left]{$a$}   --  (0,6) node[left]{$b$} -- (4,0) node[right]{$c=y$} -- (-4,0);
\draw[thick,red] plot[smooth] coordinates {(-0.8,1.8)(1.5,1.4) (4,0) (5,-1)};    
\node at (-0.8,1.8) [rotate=180, red, very thick]  {$\vartriangleleft$}; 
\draw [thick,blue] (-4,0) -- (-0.8,1.8) -- (0,6);
\coordinate (u) at (-4,0);
\coordinate (v) at (4,0);
\coordinate [label=below:$x_0$] (x_0) at (-0.8,1.8);
% \foreach \point in {u,v,f}
 %   \fill [black,opacity=.5] (\point) circle (3pt);
   \draw (1.5,1.5) node[red,below] {$\Gamma$} ;
       \draw (0,-0.5) node[below] {$\rm (ii)$} ;
\end{tikzpicture}

\caption{Sub-divis\~ao do primeiro simplexo I.}\label{subdivi 0} 
\end{figure}

Agora, basta fazer a constru\c c\~ao para um simplexo $\sigma = (a,b,c)$ no qual a curva $\Gamma$  entra. 
A constru\c c\~ao a seguir sup\~oe que todos {os} simplexos 
encontrados pela curva entre o ponto base e o simplexo $\sigma$ 
{seguindo a orienta\c c\~ao dada}  j\'a est\~ao subdivididos.  
O ponto de ``entrada'' de $\Gamma$  
em $\sigma$ pode ser seja um v\'ertice, seja um ponto $d$ localizado em um lado de $\sigma$.

Se o ponto de entrada da curva $\Gamma$  em $\sigma$ \'e um v\'ertice $a$, 
a curva pode sair em um ponto $d$ localizado seja no lado oposto, 
seja em um lado que {cont\'em} 
o v\'ertice $a$, seja em um outro v\'ertice $b$. 
No primeiro caso (figura \ref{subdivi I} (i)), 
dividimos o tri\^angulo $(a,b,c)$ em dois tri\^angulos (curvilineares) $(a,b,d)$ e $(a,d,c)$. 
No segundo caso (figura \ref{subdivi I} (ii)),  seja $e\in (a,c)$ 
 o ponto {no qual} a curva $\Gamma$  saia do tri\^angulo, 
escolhemos um ponto $f$ na curva, {localizado} 
entre $a$ e $e$. Dividimos o tri\^angulo $(a,b,c)$ 
em quatro tri\^angulos (curvilineares) $(a,b,f)$, $(b,f,e)$, $(b,e,c)$ e $(a,f,e)$. 
Enfim, no \'ultimo caso (figura \ref{subdivi I} (iii)), suponha que o ponto de {sa\'ida} 
seja o v\'ertice $c$, escolhemos um ponto $f$ na curva, localizado entre $a$ e $c$. 
Dividimos o tri\^angulo $(a,b,c)$ em tr\^es tri\^angulos (curvilineares) $(a,b,f)$, $(b,f,c)$ e $(a,f,c)$.

\begin{figure} [H] 
 \begin{tikzpicture}  [scale= 0.35]
  
\draw (-4,0) node[left]{$a$}   --  (0,6) node[left]{$b$} -- (4,0) node[right]{$c$} -- (-4,0);
\draw[thick,red] plot[smooth] coordinates {(-5,-1)(-4,0)(0,2) (2,3) (4.5,5)};     
\node at (-4.85,-0.8) [rotate=240, red, very thick]  {$\vartriangleleft$};
\coordinate [label=above:$d$] (d) at (2,3);
\coordinate (u) at (-4,0);
 \foreach \point in {u,d}
    \fill [black,opacity=.5] (\point) circle (3pt);
       \draw (4.2,4.5) node[red,below] {$\Gamma$} ;
    \draw (0,-0.5) node[below] {$\rm (i)$} ;
\end{tikzpicture}
\qquad
\begin{tikzpicture}  [scale= 0.35]
  
\draw (-4,0) node[left]{$a$}   --  (0,6) node[left]{$b$} -- (4,0) node[right]{$c$} -- (-4,0);
\draw[thick,red] plot[smooth] coordinates {(-5,-1)(-4,0)(-1,1.5) (1.5,0) (2.2,-1)};     
\node at (-4.87,-0.8) [rotate=240, red, very thick]  {$\vartriangleleft$};
\draw [thick,blue](-1,1.5) -- (0,6) -- (1.5,0);
\coordinate (u) at (-4,0);
\coordinate [label=below:$e$] (e) at (1.5,0);
\coordinate [label=below:$f$] (f) at (-1,1.5);
 \foreach \point in {u,e,f}
    \fill [black,opacity=.5] (\point) circle (3pt);
   \draw (2.6,0) node[red,below] {$\Gamma$} ;
       \draw (0,-0.5) node[below] {$\rm (ii)$} ;
\end{tikzpicture}
\qquad
\begin{tikzpicture}  [scale= 0.35]
  
\draw (-4,0) node[left]{$a$}   --  (0,6) node[left]{$b$} -- (4,0) node[right]{$c$} -- (-4,0);
\draw[thick,red] plot[smooth] coordinates {(-5,-1)(-4,0)(0.8,1.8) (4,0) (5,-1)};    
\node at (-4.9,-0.8) [rotate=240, red, very thick]  {$\vartriangleleft$}; 
\draw [thick,blue] (0,6) -- (0.8,1.8);
\coordinate (u) at (-4,0);
\coordinate (v) at (4,0);
\coordinate [label=below:$f$] (f) at (0.8,1.8);
 \foreach \point in {u,v,f}
    \fill [black,opacity=.5] (\point) circle (3pt);
\draw (-1.4,1.65) node[red] {$\Gamma$} ;
       \draw (0,-0.5) node[below] {$\rm (iii)$} ;
\end{tikzpicture}

\caption{Sub-divis\~oes II.}\label{subdivi I} 
\end{figure}
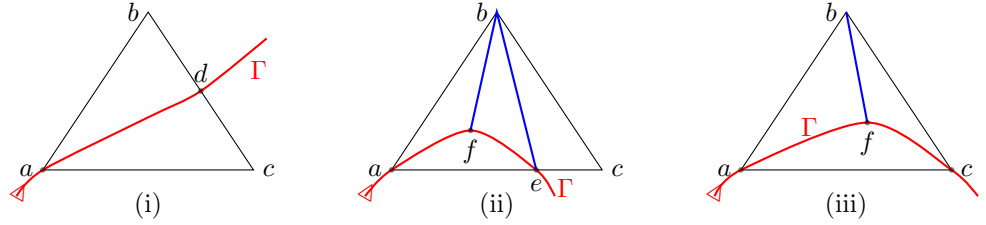

Observamos que, nos tr\^es casos acima, {a escolha}  
de sub-triangula\c c\~ao n\~ao \'e {\'unica}, 
mas a soma 
$n_0 - n_1 + n_2$ permanece inalterada independentemente da escolha. 

Se o ponto de entrada da curva $\Gamma$  em $\sigma$ estiver localizado em um lado, 
denotamos por $d$  o ponto de entrada e por $(a,c)$ o lado de $\sigma$ que cont\'em $d$. 
O pr\'oximo ponto de {sa\'ida} de $\Gamma$  pode ser 
seja em um outro lado que $(a,c)$ (por exemplo $(a,b)$), seja no mesmo lado $(a,c)$, seja um 
v\'ertice. 

\begin{figure} [H] %figure 1

 \begin{tikzpicture}  [scale= 0.35]
  
\draw (-4,0) node[left]{$a$}   --  (0,6) node[left]{$b$} -- (4,0) node[right]{$c$} -- (-4,0);
\draw [thick,red](-1,-2) arc (-10:60: 6 );
\node at (-0.87,-0.8) [rotate=270, red, very thick]  {$\vartriangleleft$};
\coordinate [label=above:$e$] (e) at (-2.2, 2.8);
\coordinate (f) at (-1, 0);
\coordinate [label=above:$d$] (d) at (-0.7, 0);
\draw [thick,blue](e) -- (4,0);
 \foreach \point in {e,f}
    \fill [black,opacity=.5] (\point) circle (3pt);
\draw (-3.5,3.3) node[red] {$\Gamma$} ;
    \draw (0.8,-0.5) node[below] {$\rm (i)$} ;
\end{tikzpicture}
\quad
 \begin{tikzpicture}  [scale= 0.35]
  
  \draw (-4,0) node[left]{$a$}   --  (0,6) node[left]{$b$} -- (4,0) node[right]{$c$} -- (-4,0);
  
\draw[thick,red] plot[smooth] coordinates {(-2.6,-1)(-2,0) (0,1.5) (1.8,0) (2.6,-1)};     
\node at (-2.5,-0.8) [rotate=240, red, thick]  {$\vartriangleleft$}; 
\draw [thick,blue](-2,0) -- (0,6) -- (1.8,0);
\draw [thick,blue](0,6) -- (0,1.5);
\coordinate [label=below:$d$] (d) at (-2, 0);
\coordinate [label=below:$f$] (f) at (0,1.5);
\coordinate [label=below:$e$] (e) at (1.8, 0);
 \foreach \point in {d,e,f}
    \fill [black,opacity=.5] (\point) circle (3pt);
   \draw (3,0) node[red,below] {$\Gamma$} ;
\draw (0,-0.5) node[below] {$\rm (ii)$} ;
\end{tikzpicture}
\quad
\begin{tikzpicture}  [scale= 0.35]
  
\draw (-4,0) node[left]{$a$}   --  (0,6) node[left]{$b$} -- (4,0) node[above]{$c$} -- (-4,0);
\draw[thick,red] plot[smooth] coordinates {(-2.2,-1)(-1.5,0) (1,1.5) (4,0)(5,-1)};
\node at (-2.13,-0.8) [rotate=240, red, thick]  {$\vartriangleleft$}; 
\draw [thick,blue](-1.5,0) -- (0,6) -- (1,1.5);
\coordinate (v) at (4,0);
\coordinate [label=below:$d$] (d) at (-1.5,0);
\coordinate [label=below:$f$] (f) at (1,1.5);
 \foreach \point in {v,d,f}
    \fill [black,opacity=.5] (\point) circle (3pt);
\draw (4,0) node[red,below] {$\Gamma$} ;
       \draw (0.8,-0.5) node[below] {$\rm (iii)$} ;
\end{tikzpicture}
\quad
\begin{tikzpicture}  [scale= 0.35]
  
\draw (-4,0) node[left]{$a$}   --  (0,6) node[left]{$b$} -- (4,0) node[right]{$c$} -- (-4,0);
\draw[thick,red] plot[smooth] coordinates {(-2.3,-1)(-1.5,0) (0.8,2) (0,6)(-0.1,6.3)};
\node at (-2.13,-0.8) [rotate=240, red, thick]  {$\vartriangleleft$}; 
\coordinate (v) at (0,6);
\coordinate [label=below:$d$] (d) at (-1.5,0);
 \foreach \point in {v,d}
    \fill [black,opacity=.5] (\point) circle (3pt);
\draw (0,2) node[red] {$\Gamma$} ;
       \draw (0.8,-0.5) node[below] {$\rm (iv)$} ;
\end{tikzpicture}

\caption{Sub-divis\~oes III. }\label{subdivi II}
\end{figure}
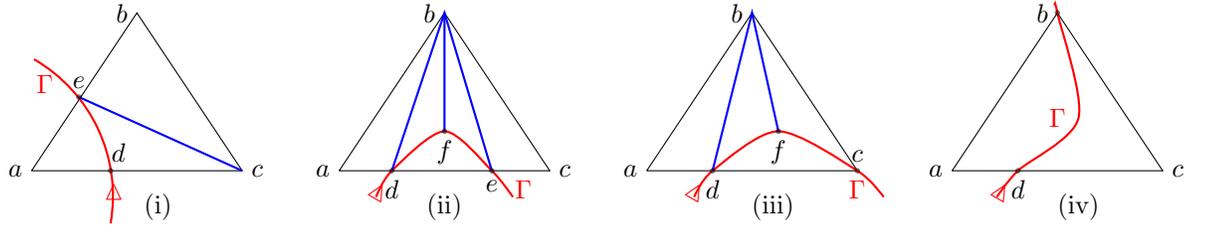

Na primeira situa\c c\~ao (figura \ref{subdivi II} (i)),
definimos (por exemplo) uma sub-triangula\c c\~ao do tri\^angulo $(a,b,c)$ como sendo 
formada dos tri\^angulos (curvilineares) $(a,d,e)$, $(c,e,d)$ e $(c,e,b)$.

Na segunda situa\c c\~ao, por exemplo se a curva $\Gamma$  entra e sai pelos ponto $d$ e $e$ localizados no mesmo lado $(a,c)$, escolhemos um ponto $f$ na curva 
localizado entre $d$ e $e$ (figura \ref{subdivi II} (ii)). 
Definimos uma sub-triangula\c c\~ao do tri\^angulo $(a,b,c)$ como sendo 
formada de cinco tri\^angulos (curvilineares) $(a,b,d)$, $(b,d,f)$, $(b,f,e)$, $(b,e,c)$ e $(d,f,e)$. 

As duas \'ultimas situa\c c\~oes (figura \ref{subdivi II} (iii) e (iv))  s\~ao semelhantes aos casos da figura 
\ref{subdivi I} (ii) e (i)  respectivamente.

Observamos que, nos quatro casos, a escolha de sub-triangula\c c\~ao n\~ao \'e \'unica, 
mas a soma 
$n_0 - n_1 + n_2$ permanece inalterada independentemente da escolha. 

Os tri\^angulos que tem um lado comum com $\sigma$ est\~ao  divididos da mesma maneira.

O processo continua para todos os $2$-simplexos 
 para os quais a curva $\Gamma$  passa. 
Eles est\~ao em n\'umero finito, mesmo depois da subdivis\~ao. 
\end{proof} 
\goodbreak

\subsection {Caso da esfera.} 
J\'a providenciamos  uma prova do Teorema \ref{Euler's Theorem} 
usando  o m\'etodo de Cauchy. No entanto, damos aqui duas provas de mais
do fato que a soma $n_0 - n_1 + n_2$ n\~ao depende da triangula\c c\~ao
da esfera. A primeira prova usa cortes, o que d\'a uma representa\c c\~ao planar da figura \ref{2spheres},
 e a segunda prova usa  a proje\c c\~ao estereogr\'afica da esfera sobre um plano. 
 Nos dois casos, usamos tamb\'em o teorema \ref{teo1}. 
%No entanto, a f\'ormula de Euler pode ser demonstrada usando uma 
%representa\c c\~ao planar da esfera na figura 
%Providenciamos as duas maneiras para mostrar o resultado.  

\medskip 

\noindent {\bf 1. Usando os cortes.} \label{caso_esfera}
\\
Seja ${\mathcal T}$ uma triangula\c c\~ao da esfera $\Sp^2$. Consideramos quatro curvas sobre a esfera: 
o equador $E$ (ou qualquer paralelo) e tr\^es curvas $\gamma_1$, $\gamma_2$ e $\gamma_3$,  
indo do p\'olo norte $N$ at\'e a curva $E$ ao longo de meridianos. Denotamos por 
$a_i$, $i=1,2,3$, os pontos $\gamma_i \cap E$ de chegada das curvas $\gamma_i$. 
Estas curvas podem ser, por exemplo,  parte norte dos meridianos $0^\circ$, $120^\circ$ leste 
e $120^\circ$ oeste 
(veja figura \ref{2spheres} (1)). Usando o lema \ref{lemageneral} 
 podemos construir uma subdivis\~ao ${\mathcal T}'$ 
da triangula\c c\~ao ${\mathcal T}$
 compat\'ivel com as quatro curvas, isto \'e, tal que a uni\~ao das quatro curvas seja 
um subcomplexo de ${\mathcal T}'$. O lema \ref{lemageneral} mostra que a soma $n_0 - n_1 + n_2$ 
permanece a mesma.

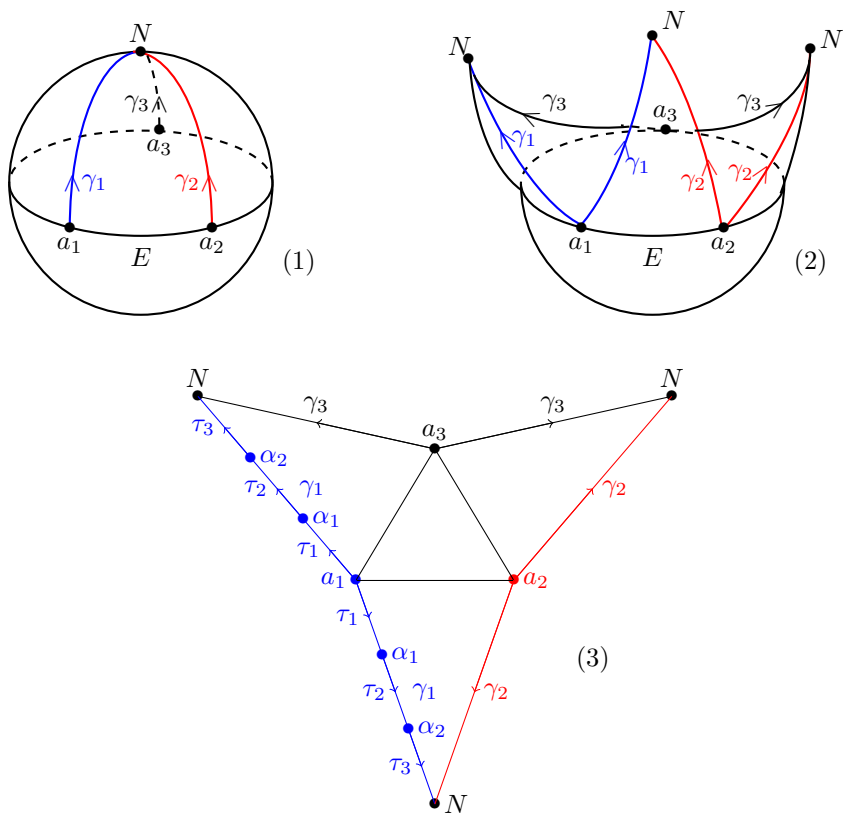
\begin{figure} [H] %figure 1
 \begin{tikzpicture}  [scale= 0.35]
 
 \draw[thick] (0,0) circle (5);
%equateur 
 \draw [thick](-5,0) arc (180:360: 5  and 2 );
  \draw [thick,dashed](5,0) arc (0:180: 5  and 2 );

% meridiens
 \draw [blue,thick](0,5) arc (90:180: 2.7  and 6.7 );
 \draw [red,thick](2.7,-1.7) arc (0:90: 3  and 6.7 );
  \draw [thick,dashed](0.7,2.2) arc (0:45: 1.8  and 4 );
  
  %points
  \filldraw (0,5) node {$\bullet$} node [above] {$N$};
  \filldraw (-2.7,-1.7) node {$\bullet$} node [below] {$a_1$};
    \filldraw (2.7,-1.7) node {$\bullet$} node [below] {$a_2$};
    \filldraw (0.7,2) node {$\bullet$} node [below] {$a_3$};
\filldraw (0,-2.1)   node [below] {$E$};

\node at (-2.6, 0)[blue]{$\wedge$};
\node at (-2.6,0)[blue,right]{$\gamma_1$};
\node at (2.6, 0)[red]{$\wedge$};
\node at (2.6,0)[red,left]{$\gamma_2$};
\node at (0.7,3){$\wedge$};
\node at (0.7,3)[left]{$\gamma_3$};

\node at (7,-3)[left]{$(1)$};

\end{tikzpicture}
\qquad \qquad
 \begin{tikzpicture}  [scale= 0.35]
 
 \draw[thick] (-5,0) arc (180:360: 5  and 5);
 
 %equateur 
 \draw [thick](-5,0) arc (180:360: 5  and 2 );
  \draw [thick,dashed](5,0) arc (0:180: 5  and 2 );

% meridiens gamma1 et 2
\draw [rotate=-10][blue,thick](-2.5,-2.2)  arc (300:366: 3  and  8 ) node[sloped,midway]{$>$};
 \draw  [rotate=10] [red,thick](2.3,-2) arc (0:57: 3  and 9 ) node[sloped,near start]{$<$};
 \node at (-1.4,0.65)[blue,right]{$\gamma_1$};
  \node at (2.4,0)[red,left]{$\gamma_2$};
  
 % meridiens gamma2 et 3
\draw[rotate=320] [red,thick](3.2,0.3) arc (0:70: 3  and  8 ) ; %node[sloped,near start] {$<$};
\draw[rotate=340] [thick](4.7,1.3) arc (0:40: 3 and  8 ) ;
\draw [thick](1.6,2) arc (270:360: 4.39 and  3 )node[sloped,midway]{$>$};
\draw [thick,dashed](0.7,2) -- (1.2,2);

\node at (4.2,0.4)[red] [rotate=240]  {$<$};
  \node at (4.2,0.4)[red,left]{$\gamma_2$};
\node at (3.7,2.4)[above]{$\gamma_3$};

  % meridiens gamma1 et 3
  \draw[rotate=20] [blue,thick](-3,-0.65) arc (250:180: 3  and  8 ) ; %node[sloped,near end]{$<$};
\draw[thick](-5,-0.25) arc (250:187: 3  and  6 );
\draw[rotate=5]  [thick](-6.4,4.7) arc (190: 270: 5.8 and  3 ) node[sloped,midway]{$<$};
\draw [thick,dashed](0.7,2) -- (-1.2,2.2);

\node at (-5.62,1.7)[blue] [rotate=300]  {$<$};
\node at (-3.7,2.4)[above]{$\gamma_3$};
 \node at (-5.7,1.7)[blue,right]{$\gamma_1$};
 
  %points
%  \filldraw (0,5) node {$\bullet$} node [above] {$N$};
  \filldraw (-2.7,-1.7) node {$\bullet$} node [below] {$a_1$};
    \filldraw (2.7,-1.7) node {$\bullet$} node [below] {$a_2$};
    \filldraw (0.5,2) node {$\bullet$} node [above] {$a_3$};
\filldraw (0,-2.1)   node [below] {$E$};
\filldraw (0,6.2)  node [right] {$N$};
\filldraw (6,5.5)   node [right] {$N$};
\filldraw (-6.5,5.2)   node [left] {$N$};
\filldraw (0,5.6)  node {$\bullet$};
\filldraw (6,5.1)   node {$\bullet$};
\filldraw (-7,4.7)   node {$\bullet$};

\node at (7,-3)[left]{$(2)$};

\end{tikzpicture}
\bigskip
\vglue0.2truecm
 \begin{tikzpicture}  [scale= 0.35]

% les points N
\coordinate (N1) at  (9,7);
\filldraw (N1) node {$\bullet$} node [above] {$N$};
\coordinate (N2) at  (-9,7);
\filldraw (N2) node {$\bullet$} node [above] {$N$};
\coordinate (N3) at  (0,-8.5);
\filldraw (N3) node {$\bullet$} node [right] {$N$};

% les points a_i
\coordinate (a1) at  (-3,0);
\filldraw (a1) node [blue] {$\bullet$} node [blue,left] {$a_1$};
\coordinate (a2) at  (3,0);
\filldraw (a2) node [red]{$\bullet$} node [red,right] {$a_2$};
\coordinate (a3) at  (0,5);
\filldraw (a3) node {$\bullet$} node [above] {$a_3$};

%les c�t�s
\draw (N2)--(a3) -- (N1);
\draw (a1) --(a3)  -- (a2)  -- (a1);
\draw [blue] (N3) -- (a1) --(N2);
\draw [red](N1) -- (a2) -- (N3);

% les sommets interm�diaires et les fl�ches
\coordinate (c1) at  (-5,2.33);
\coordinate (c2) at  (-7,4.66);
\filldraw (c1) node [blue] {$\bullet$} node [blue,right] {$\alpha_1$};
\filldraw (c2) node [blue] {$\bullet$} node [blue,right] {$\alpha_2$};

\coordinate (d1) at (-4,1.165);
\filldraw [blue,->] (a1)-- (d1);
\coordinate (d2) at (-6,3.495);
\filldraw [blue,->] (c1)-- (d2);
\coordinate (d3) at (-8,5.825);
\filldraw [blue, ->] (c2)-- (d3);

\coordinate (e1) at  (-2,-2.833);
\coordinate (e2) at  (-1,-5.666);
\filldraw (e1) node [blue]{$\bullet$} node [blue,right] {$\alpha_1$};
\filldraw (e2) node [blue]{$\bullet$} node [blue,right] {$\alpha_2$};

\coordinate (f1) at (-2.5,-1.416);
\filldraw [blue,->] (a1)-- (f1);
\coordinate (f2) at (-1.5,-4.242);
\filldraw [blue,->] (e1)-- (f2);
\coordinate (f3) at (-0.5,-7.08);
\filldraw [blue,->] (e2)-- (f3);

\filldraw (d1)  node [blue,left] {$\tau_1$};
\filldraw (d2) node [blue,left] {$\tau_2$};
\filldraw (d3)  node [blue,left] {$\tau_3$};
\filldraw (f1)  node [blue,left] {$\tau_1$};
\filldraw (f2) node [blue,left] {$\tau_2$};
\filldraw (f3) node [blue,left] {$\tau_3$};

\coordinate  (g) at (-4.5,6);
\coordinate  (h) at (4.5,6);
\coordinate  (j) at (6,3.5);
\coordinate  (k) at (1.5,-4.25);

\filldraw [->] (a3)-- (g);
\filldraw [->] (a3)-- (h);
\filldraw [red,->] (a2)-- (j);
\filldraw [red,->] (a2)-- (k);

\filldraw (g)  node [above] {$\gamma_3$};
\filldraw (h) node [above] {$\gamma_3$};
\filldraw (j)  node [red,right] {$\gamma_2$};
\filldraw (k)  node [red,right] {$\gamma_2$};
\filldraw (-5.5,3.5) node [blue,right] {$\gamma_1$};
\filldraw (-1.2,-4.25) node [blue,right] {$\gamma_1$};

\node at (7,-3)[left]{$(3)$};
\end{tikzpicture}
\caption{Representa\c c\~ao planar da esfera }\label{2spheres}
\end{figure}

Agora, cortamos a esfera ao longo das curvas $\gamma_1$, $\gamma_2$ e $\gamma_3$,  
de tal modo que 
podemos projetar a figura no plano contendo o equador (veja figura \ref{2spheres} (2)). 
Obtemos uma representa\c c\~ao planar $K$ da esfera homeomorfa a um disco 
e com identifica\c c\~oes dos simplexos na borda $K_0$ correspondentes aos cortes feitos. 
A triangula\c c\~ao de $K$ corresponde simplexo por simplexo 
\`a triangula\c c\~ao ${\mathcal T}'$.

O teorema \ref{teo1} diz que a soma $n_0^K - n_1^K + n_2^K$ pela triangula\c c\~ao $K$ \'e igual \`a soma 
$n_0^{K_0} - n_1^{K_0} +1$. 
Usando as mesmas nota\c c\~oes na esfera e na representa\c c\~ao planar, 
o v\'ertice $N$ \'e comum \`as todas as curvas $\gamma_i$ e deve ser identificado. Fora
deste v\'ertice, em cada curva $\gamma_i$ temos um n\' umero de v\'ertices igual ao n\'umero de 
arestas (veja figura \ref{2spheres} (3) onde desenhamos um exemplo de sub-triangula\c c\~ao da aresta 
$\gamma_1$ com tr\^es v\'ertices e tr\^es arestas). Ent\~ao pela borda da representa\c c\~ao planar temos 
$n_0^{K_0} - n_1^{K_0} = +1$  e pela triangula\c c\~ao $\mathcal T$ da esfera temos 
$$n_0 - n_1 + n_2 = +2.$$

\medskip

\noindent {\bf 2. Usando a proje\c c\~ao estereogr\'afica.}

Seja $\mathcal T$ uma triangula\c c\~ao da esfera $\Sp^2$ e $\sigma=(a,b,c)$ o tri\^angulo (aberto)
da triangula\c c\~ao contendo o p\'olo norte.
A primeira etapa \'e fazer a proje\c c\~ao estereogr\'afica $p: \Sp^2 \setminus \sigma \to Q$ da esfera removida o tri\^angulo $\sigma$ 
onde $Q$ \'e o plano tangente da esfera ao p\'olo sul. Denotamos $U = \Sp^2 \setminus \sigma$. 
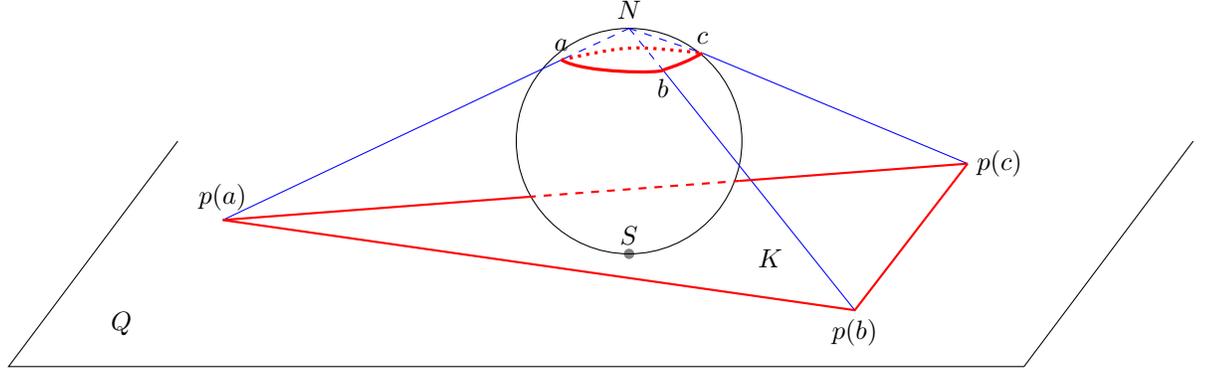
\begin{figure}[H] 
\begin{tikzpicture}  [scale= 1.5]    

\draw (0,1) circle (1cm);
\coordinate [label=above:$a$]  (a) at (-0.6,1.72);
\coordinate [label=below:$b$]  (b) at (0.3,1.63);
\coordinate [label=above:$c$]  (c) at (0.65,1.78);
\coordinate [label=above:$p(a)$]  (A) at (-3.6, 0.3);
\coordinate [label=below:$p(b)$]  (B) at (2, -0.5);
\coordinate [label=right:$p(c)$]  (C) at (3, 0.8);
\coordinate [label=above:$N$]  (N) at (0,2);
\coordinate [label=above:$S$]  (S) at (0,0);xfdisco

\coordinate [label=above:$K$]  (T) at (1.25,-0.2);
\coordinate [label=above:$Q$]  (P) at (-4.5,-0.8);
 \foreach \point in {S}
    \fill [black,opacity=.5] (\point) circle (1.3pt);

\draw (-4,1)--(-5.5,-1) -- (3.5, -1) -- (5,1);
\draw[dashed,blue] (N) -- (a);
\draw [blue] (a) --(A);
\draw[dashed,blue] (N) -- (b);
\draw [blue] (b) -- (B);
\draw[dashed,blue] (N) -- (c);
\draw [blue] (c) -- (C);

\draw [very thick, red] (a) .. controls (-0.4,1.6) and (0.3,1.6) .. (b);
\draw [very thick,red] (b) .. controls (0.5,1.7) .. (c);
\draw [very thick,dotted,red] (a) .. controls (0,1.85) .. (c);

\coordinate (D1) at (-1.2,0);
\coordinate(D) at (intersection of A--C and N--D1);
\coordinate (E1) at (1.8,-0.5);
\coordinate (E) at (intersection of A--C and N--E1);
\draw [thick,red](A)--(B)--(C);
\draw [thick,red](A)--(D);
\draw [thick,dashed,red] (D)--(E);
\draw [thick,red](E)--(C);
 \end{tikzpicture}
\caption{A proje\c c\~ao estereogr\'afica.}\label{Estereo}
\end{figure}

A proje\c c\~ao  estereogr\'afica  de $U = \Sp^2 \setminus \sigma$ 
sobre o plano $Q$ \'e um tri\^angulo $K$ 
 homeomorfa a um disco $D$, e, neste caso, sem identifica\c c\~oes na borda $K_0$.
A triangula\c c\~ao da esfera induz uma triangula\c c\~ao do tri\^angulo $K$,
cada simplexo de $K$ sendo a imagem de um simplexo de ${\mathcal T}$ pela proje\c c\~ao $p$. 
Ent\~ao a quantidade $n_0 - n_1 + n_2$ satisfaz :
$$n_0^{\mathcal T} - n_1^{\mathcal T} + n_2^{\mathcal T}  = n_0^U - n_1^U + n_2^U + 1,$$
onde ``+1'' \'e a contribu\c c\~ao  do tri\^angulo $\sigma$. 

Agora, pelo teorema \ref{teo1}, temos 
$$n_0^U - n_1^U + n_2^U = n_0^K - n_1^K + n_2^K 
= n_0^{K_0} - n_1^{K_0} + 1 = +1,$$
observando que $K_0$ \'e a borda do tri\^angulo $K$ 
com $n_0^{K_0} = n_1^{K_0} = 3$. 

Finalmente, obtemos 
$$n_0^{\mathcal T} - n_1^{\mathcal T} + n_2^{\mathcal T} = +2$$
para qualquer triangula\c c\~ao da esfera. 

\subsection{Caso do toro.}

Seja $\mathcal T$ uma triangula\c c\~ao qualquer do toro $\T = \Sp^1 \times \Sp^1$.

Escolhamos um meridiano $M=\Sp^1 \times \{0\}$ e um paralelo $P=\{0\} \times \Sp^1$. 
Eles se cruzam em um ponto $a = \{0\} \times\{0\}$. Observamos que,
sem perda de generalidade, podemos escolh\^e-los transversalmente a todas as arestas
(simplexos de dimens\~ao 1) de ${\mathcal T}$. 
Definimos uma sub-triangula\c c\~ao ${\mathcal T}'$ de 
$\mathcal T$, da  seguinte maneira (veja figura \ref{letore}): 
Cada tri\^angulo $\sigma$ (simplexo de dimens\~ao 2) de ${\mathcal T}$
que encontra $M$ ou $P$ est\'a subdividido de maneira que $\sigma \cap M$ (ou  $\sigma \cap P$) seja uma  
aresta de ${\mathcal T}'$. Conclu\'\i mos a subdivis\~ao obtendo uma sub-triangula\c c\~ao ${\mathcal T}'$ 
de ${\mathcal T}$. 
O lema \ref{lemageneral} mostra que a soma $n_0 - n_1 + n_2$ permanece a mesma
para ${\mathcal T}$ e ${\mathcal T}'$.

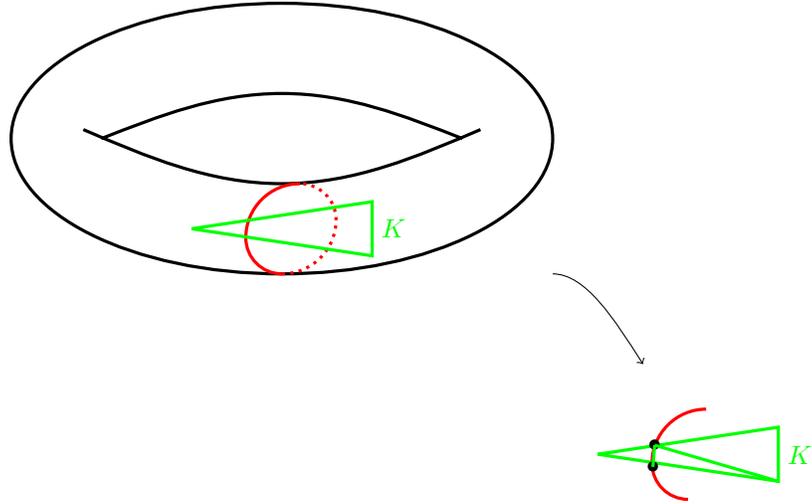
\begin{figure}[H]
\begin{tikzpicture}[scale=1.2]

\draw[very thick] (3, 1.5) ellipse (3cm and 1.5 cm);
\draw [very thick] (0.8, 1.6) sin (3, 1) cos (5.2, 1.6);
\draw [very thick] (1, 1.5) sin  (3, 2) cos (5, 1.5);); 

\draw[dotted, very thick, red] (3,0) arc (-90:0:0.6) arc (0:90: 0.4cm) ;
\draw[very thick,red] (3,0) arc (90:0:-0.4) arc (0:-90: -0.6cm) ;

% Triangulation $K$
\draw[green, very thick] (2, 0.5) -- (4, 0.2) -- (4, 0.8) -- (2, 0.5); 
\node at (4, 0.5) [right] {\textcolor{green}{$K$}};

%Triangulation $K'$ (= $K$ + (+4.5, -2.5))
\draw[very thick,red] (7.5,-2.5) arc (90:0:-0.4) arc (0:-90: -0.6cm) ;
\draw[green, very thick] (6.5, -2) -- (8.5, -2.3) -- (8.5, -1.7) -- (6.5, -2); 
\node at (7.11, -2.14) {$\bullet$}; % Points d'intersection
\node at (7.13, -1.9) {$\bullet$};

\draw [green, very thick] (7.11, -2.14) -- (7.13, -1.9); 
\draw [green, very thick] (7.13, -1.9) -- (8.5, -2.3); 
\node at (8.5, -2) [right] {\textcolor{green}{$K'$}};

%SETA

\draw [->] (6,0) cos ( 7, -1); 
\end{tikzpicture}
\caption{Sub-triangula\c c\~ao $K'$ de $K$. }
\label{letore}
\end{figure}

\begin{figure}[H]
\begin{tikzpicture}[scale=0.6]

%Rectangulo
 \draw [blue, thick]  (0, 0) -- (6,0) ;
 \draw [red, thick]  (6,0) -- (6,6) ; 
 \draw [red, thick] (0,6) -- (0,0); 
\draw [blue, thick]  (6,6) --  (0,6) ;
 
\coordinate  (0) at (0,0);
\coordinate  (10) at (0,1.5);
\coordinate  (11) at (0,2.8);
\coordinate  (12) at (0,3.5);
\coordinate  (13) at (0,6);

\coordinate  (9) at (6,0);
\coordinate  (8) at (6,1.5);
\coordinate  (19) at (6,2.8);
\coordinate  (20) at (6,3.5);
\coordinate  (21) at (6,6);

\coordinate  (1) at (2,0);
\coordinate  (4) at (3,0);
\coordinate  (6) at (4.8,0);

\coordinate  (22) at (2,6);
\coordinate  (23) at (3,6);
\coordinate  (24) at (4.8,6);

\coordinate  (2) at (1.2,2);
\coordinate  (3) at (2.4,2);
\coordinate  (5) at (3.7,2.5);
\coordinate  (7) at (5,2.2);

\coordinate  (14) at (1.5,5);
\coordinate  (15) at (2.5,4);
\coordinate  (16) at (4,3.8);
\coordinate  (17) at (4.6,3);
\coordinate  (18) at (4.9,5);

\draw [green, thick]  (0) -- (2) -- (1) -- (3) -- (2);
\draw [green, thick]  (5) -- (4) -- (3) -- (5) -- (6) -- (7) -- (5);
\draw [green, thick] (9) -- (7) -- (8);
\draw [green, thick] (10) -- (2) -- (11);
\draw [green, thick] (12) -- (14) -- (13);
\draw [green, thick] (14) -- (22) -- (15) -- (14) -- (2) -- (15) -- (3);
\draw [green, thick] (21) -- (18) -- (24) -- (16) -- (18) -- (17) -- (16) -- (23) -- (15) -- (16) -- (5) -- (15);
\draw [green, thick] (18) -- (20) -- (17) -- (19) -- (7) -- (17) -- (5);

%setas
\node at (1,0) [blue] {>};
\node at (1,0)[below]{$a$};
\node at (2.5,0) [blue] {>};
\node at (2.5,0)[below]{$b$};
\node at (4,0)[blue] {>};
\node at (4,0)[below]{$c$};
\node at (5.5,0)[blue] {>};
\node at (5.5,0)[below]{$d$};

\node at (1,6)[blue] {>};
\node at (1,6)[above]{$a$};
\node at (2.5,6)[blue] {>};
\node at (2.5,6)[above]{$b$};
\node at (4,6)[blue] {>};
\node at (4,6)[above]{$c$};
\node at (5.5,6)[blue] {>};
\node at (5.5,6)[above]{$d$};

\node at (0, 0.9)[red]{$\wedge$};
\node at (0,0.9)[left]{$e$};
\node at (0, 2.2)[red]{$\wedge$};
\node at (0,2.2)[left]{$f$};
\node at (0, 3.1)[red]{$\wedge$};
\node at (0,3.1)[left]{$g$};
\node at (0, 4.7)[red]{$\wedge$};
\node at (0,4.7)[left]{$h$};

\node at (6, 0.9)[red]{$\wedge$};
\node at (6,0.9)[right]{$e$};
\node at (6, 2.2)[red]{$\wedge$};
\node at (6,2.2)[right]{$f$};
\node at (6, 3.1)[red]{$\wedge$};
\node at (6,3.1)[right]{$g$};
\node at (6, 4.7)[red]{$\wedge$};
\node at (6,4.7)[right]{$h$};

\node at (0, 0)[left]{$A$};
\node at (6, 0)[right]{$A$};
\node at (0, 6)[left]{$A$};
\node at (6, 6)[right]{$A$};

\end{tikzpicture}
\caption{Uma representa\c c\~ao planar $K$ do toro.}  \label{Toroplanar} 
\end{figure}
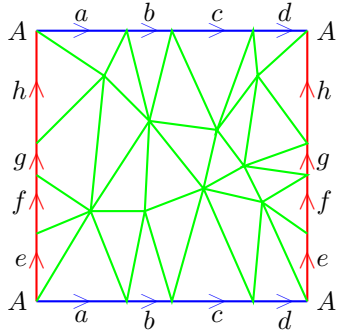

Agora, cortando  o toro ao longo $M$ e $P$, obtemos uma representa\c c\~ao planar $K$ do toro 
 homeomorfa a um disco, com identifica\c c\~oes correspondante ao corte. 
 Esta representa\c c\~ao \'e feita da mesma maneira da representa\c c\~ao planar 
 na figura \ref{o toro} mas triangulada 
como a triangula\c c\~ao $K$ correspondente a sub-triangula\c c\~ao ${\mathcal T}'$  
de ${\mathcal T}$.  
 Portanto, pelo lema \ref{lemageneral}, 
o n\'umero $n_0 - n_1 + n_2$ permanece inalterada. 
 Usando o teorema \ref{teo1}, temos 
 $$n_0^{\mathcal T} - n_1^{\mathcal T} + n_2^{\mathcal T} = 
 n_0^{{\mathcal T}'}  - n_1^{{\mathcal T}'}  + n_2^{{\mathcal T}'}  = 
 n_0^K - n_1^K + n_2^K =
 n_0^{K_0} - n_1^{K_0} + 1.$$
 Agora, sendo as identifica\c c\~oes na borda 
 $K_0$, temos: $n_1^{K_0} = n_0^{K_0} +1$. Finalmente, 
 $$n_0^{\mathcal T} - n_1^{\mathcal T} + n_2^{\mathcal T} = 0$$ 
 para qualquer triangula\c c\~ao do toro. 

A mesma prova vale para o toro de qualquer g\^enero $g$. 
Por exemplo, tomemos $g=3$. 
Sendo uma triangula\c c\~ao ${\mathcal T}$, fixamos um ponto $x$ e, ao redor de cada ``buraco'' do toro, fixamos um ``meridiano'' 
($a_1$, $a_2$, $a_3$ nas figuras \ref{toreg}) e um  ``paralelo'' ({$b_1$, $b_2$, $b_3$} nas figuras \ref{toreg}). 
Constru\'imos uma sub-triangula\c c\~ao ${\mathcal T}'$ de ${\mathcal T}$ pelo mesmo m\'etodo que no caso  
do toro $T$. Cortando o toro de g\^enero $g$ ao longo dos meridianos e dos paralelos, 
obtemos uma representa\c c\~ao planar $K$ do toro de g\^enero $3$, 
que \'e um pol\'igino triangulado pela imagem da 
triangula\c c\~ao ${\mathcal T}'$.

\begin{figure}% [H] %figure 1

 \begin{tikzpicture}  [scale= 0.20]
 
 \draw[thick] (0,0) ellipse (12 and 6);
  \draw [thick](-8,0.6) arc (170:370: 3  and 1 );
  
  %a1 et a2
  \draw[very thick,blue] plot[smooth] coordinates {(0,-3)(-0.3,0)(-0.9,2.3) (-2.5,3.5) (-4.5,3.8)(-6,3.6)(-7.5,3.1) (-9.5,1.5)
   (-9.5,0)(-7.7,-2)(-4,-2.8) (0,-3)};  
   \draw[very thick,blue] plot[smooth] coordinates {(0,-3)(0.3,0)(0.9,2.3) (2.5,3.5) (4.5,3.8)(6,3.6)(7.5,3.1) (9.5,1.5)
   (9.5,0)(7.7,-2)(4,-2.8) (0,-3)};  
   
   %b1
   \draw[very thick,red] plot[smooth] coordinates {(0,-3)(-1.25,-1.6)(-2.5,-0.8)(-4.7,-0.6)};  
    \draw[very thick,dashed,red] plot[smooth] coordinates {(-4.7,-0.6)(-5.5,-1) (-5.8,-1.3)(-6,-1.6)(-6.4,-3)(-6,-4.4)(-5.5,-5.2)
    (-4.7,-5.5)};  
       \draw[very thick,red] plot[smooth] coordinates {(-4.7,-5.5)(-4,-5.5) (-3,-5.35)(-2,-4.8)(-1.4,-4.3)(0,-3)};  
 %b2      
       \draw[very thick,red] plot[smooth] coordinates {(0,-3)(1.25,-1.6)(2.5,-0.8)(4.7,-0.6)};  
    \draw[very thick,red,dashed] plot[smooth] coordinates {(4.7,-0.6)(5.5,-1) (5.8,-1.3)(6,-1.6)(6.4,-3)(6,-4.4)(5.5,-5.2)
    (4.7,-5.5)};  
       \draw[very thick,red] plot[smooth] coordinates {(4.7,-5.5)(4,-5.5) (3,-5.35)(2,-4.8)(1.4,-4.3)(0,-3)};  
   
   \coordinate (a1) at (-4.5,3.8);
   \filldraw (a1) node [blue] {$<$} node [above] {$a_1$};
     \coordinate (ra1) at (-4,-2.8);
    %  \filldraw (ra1) node {$>$};
      
        \coordinate (b1) at (-2.5,-0.8);
   \filldraw (b1) node {\textcolor{red}{$>$}} node [below] {$b_1$};
     \coordinate (rb1) at (-3,-5.35);
   %   \filldraw (rb1) node {$<$};

 \coordinate (a2) at (4.5,3.8);
   \filldraw (a2) node  [blue]  {$>$} node [above] {$a_2$};
     \coordinate (ra2) at (4,-2.8);
   %   \filldraw (ra2) node {$<$};
      
        \coordinate (b2) at (2.5,-0.8);
   \filldraw (b2) node {\textcolor{red}{$<$}} node [below] {$b_2$};
     \coordinate (rb2) at (3,-5.35);
 %     \filldraw (rb2) node {$<$};

    \draw [thick](-2.45,-0.15) arc (17:163: 2.7  and 1 );
 \draw [thick](2,0.6) arc (170:370: 3  and 1 );
    \draw [thick](7.65,0) arc (17:168: 2.7  and 1 );
      
   \end{tikzpicture}
   \qquad 
   \begin{tikzpicture}  [scale= 0.3]

  \coordinate  (a) at (-2,5) ;
  \coordinate  (b) at (2,5) ;
  \coordinate  (c) at (5,2) ;
  \coordinate  (d) at (5,-2) ;
  \coordinate  (e) at (2,-5) ;
 \coordinate  (f) at (-2,-5) ;
 \coordinate  (g) at (-5,-2) ;
 \coordinate  (h) at (-5,2) ;
 
 \draw [blue] (a)--(b) node[blue,sloped,midway] {$>$};
 \node at (0,5) [above]{$a_1$};
 
  \draw [red] (b)--(c) node[sloped,midway] {$>$};
  \node at (3.6,3.6) [right]{$b_1$};
  
 \draw [blue]  (c)--(d) node[sloped,midway] {$<$};
 \node at (5,0) [right]{$a_1$};  
  
    \draw [red] (d)--(e) node[sloped,midway] {$>$};
   \node at (3.6,-3.6) [right]{$b_1$}; 
   
 \draw [blue]  (e)--(f) node[sloped,midway] {$<$};
   \node at (0,-5) [below]{$a_2$};   
   
  \draw[red]  (f)--(g) node[sloped,midway] {$<$};
    \node at (-3.6,-3.6) [left]{$b_2$};
    
 \draw [blue]  (g)--(h) node[sloped,midway] {$<$};
      \node at (-5,0) [left]{$a_2$};     
      
        \draw[red]  (h)--(a) node[sloped,midway] {$<$};
     \node at (-3.6,3.6) [left]{$b_2$};
  
 \end{tikzpicture}

\hglue 1truecm

\begin{tikzpicture}  [scale= 0.17]

 \draw[thick] (14,0) arc (00:155 : 14 and 8.1);
\draw[rotate=120] [thick](9.5,8) arc (0:180: 14 and  8 ) ;
 \draw[rotate=240] [thick](18,10) arc (20:150: 14.5 and  9 ) ;
 
  %\draw [thick](-8,0.6) arc (170:370: 3  and 1 );
  
  %a1 et a2 et a3
  \draw[blue,very thick] plot[smooth] coordinates {(0,-3)(-0.3,0)(-0.9,2.3) (-2.5,3.5) (-4.5,3.8)(-6,3.6)(-7.5,3.1) (-9.5,1.5)
   (-9.5,0)(-7.7,-2)(-4,-2.8) (0,-3)};  
   \draw[blue,very thick] plot[smooth] coordinates {(0,-3)(0.3,0)(0.9,2.3) (2.5,3.5) (4.5,3.8)(6,3.6)(7.5,3.1) (9.5,1.5)
   (9.5,0)(7.7,-2)(4,-2.8) (0,-3)};  
     \draw[blue,very thick] plot[smooth] coordinates {(0,-3)(-3,-7)(-4,-9)(-5.5,-14) (-5,-16)(-2.5,-16.8)(0,-17)
     (2.5,-16.8)(5,-16)(5.5,-14) (4,-9)(3,-7)(0,-3) };   
   
   %b1
   \draw[very thick,red] plot[smooth] coordinates {(0,-3)(-1.25,-1.6)(-2.5,-0.8)(-4.7,-0.6)};  
    \draw[very thick,dashed,red] plot[smooth] coordinates {(-4.7,-0.6)(-6,-1) (-6.5,-1.2)
    (-9,-2.3)(-13,-4.4)(-14,-5.2)};  
   \draw[very thick,red] plot[smooth] coordinates {(-14,-5.2)(-9.5,-5.6)  (-3,-4.6)  (0,-3)};  
    %b2      
       \draw[very thick,red] plot[smooth] coordinates {(0,-3)(1.25,-1.6)(2.5,-0.8)(4.7,-0.6)};  
    \draw[very thick,dashed,red] plot[smooth] coordinates {(4.7,-0.6)(6,-1) (6.5,-1.2)
    (9,-2.3)(13,-4.4)(14,-5.2)};  
       \draw[very thick,red] plot[smooth] coordinates {(14,-5.2)(9.5,-5.6)  (3,-4.6)  (0,-3)};  
   %b3
    \draw[very thick,red] plot[smooth] coordinates {(0,-3)(3.5,-6)(7,-11)(7.5,-14)(7.5,-16)(7.2,-17.2)};  
       \draw[very thick,red,dashed] plot[smooth] coordinates {(7.2,-17.2)(6.5,-17.5) (6,-17.8)(5,-17.8)(4,-17.2)
       (3,-16)(2.3,-14)};  
    \draw[very thick,red] plot[smooth] coordinates {(0,-3)(0.2,-7)(1,-11)  (2.3,-14)};  
    
   \coordinate (a1) at (-4.5,3.8);
   \filldraw (a1) node [blue] {$<$} node [above] {$a_1$};
     \coordinate (ra1) at (-4,-2.8);
    %  \filldraw (ra1) node {$>$};
      
        \coordinate (b1) at (-9.5,-5.6); 
           \filldraw (b1) node {\textcolor{red}{$<$}} node [below] {$b_1$};
     \coordinate (rb1) at (-3,-5.35);
   %   \filldraw (rb1) node {$<$};

 \coordinate (a2) at (4.5,3.8);
   \filldraw (a2) node [blue] {$>$} node [above] {$a_2$};
     \coordinate (ra2) at (4,-2.8);
   %   \filldraw (ra2) node {$<$};
      
        \coordinate (b2) at (9.5,-5.6);
   \filldraw (b2) node {\textcolor{red}{$>$}} node [below] {$b_2$};
     \coordinate (rb2) at (3,-5.35);
 %     \filldraw (rb2) node {$<$};
 
 \coordinate (a3) at (-5.5,-14);
   \filldraw (a3) node [blue] {$\vee$} node [left] {$a_3$};
     \coordinate (ra2) at (4,-2.8);
 
         \coordinate (b3) at (7.5,-13);
   \filldraw (b3) node {\textcolor{red}{$\wedge$}} node [right] {$b_3$};
     \coordinate (rb2) at (3,-5.35);

    \draw [thick](-8,0.6) arc (170:370: 3  and 1 );    
    \draw [thick](-2.45,-0.15) arc (17:163: 2.7  and 1 );
 \draw [thick](2,0.6) arc (170:370: 3  and 1 );
    \draw [thick](7.65,0) arc (17:168: 2.7  and 1 );
        \draw [thick](-3,-13) arc (170:370: 3  and 1 );    
    \draw [thick](2.55,-13.8) arc (17:163: 2.7  and 1 );
  \end{tikzpicture}
  \qquad\qquad
 \begin{tikzpicture}  [scale= 0.55]
 
  \coordinate  (a) at (-1,4) ;
  \coordinate  (b) at (1,4) ;
  \coordinate  (c) at (3,3) ;
  \coordinate  (d) at (4,1) ;
  \coordinate  (e) at (4,-1) ;
 \coordinate  (f) at (3,-3) ;
 \coordinate  (g) at (1,-4) ;
 \coordinate  (h) at (-1,-4) ;
  \coordinate  (j) at (-3,-3) ;
 \coordinate  (k) at (-4,-1) ;
 \coordinate  (l) at (-4,1) ;
 \coordinate  (m) at (-3,3) ;

 \draw [blue]  (a)--(b) node[sloped,midway] {$>$};
 \node at (0,4) [above]{$a_1$};
 
  \draw [red] (b)--(c) node[sloped,midway] {$>$};
  \node at (2,3.8) [right]{$b_1$};
  
 \draw [blue]  (c)--(d) node[sloped,midway] {$<$};
 \node at (4.2,1.7) [above]{$a_1$};  
  
    \draw [red] (d)--(e) node[sloped,midway] {$<$};
   \node at (4,0) [right]{$b_1$}; 
   
 \draw [blue]  (e)--(f) node[sloped,midway] {$<$};
   \node at (4.2,-1.7) [below]{$a_2$};   
   
  \draw[red]  (f)--(g) node[sloped,midway] {$<$};
    \node at (2,-3.8) [right]{$b_2$};
    
 \draw [blue] (g)--(h) node[sloped,midway] {$>$};
      \node at (0,-4) [below]{$a_2$};     
      
        \draw[red]  (h)--(j) node[sloped,midway] {$>$};
     \node at (-2,-3.6) [left]{$b_2$};
  
 \draw [blue]   (j)--(k) node[sloped,midway] {$<$};
  \node at(-4.2,-1.7) [below]{$a_3$};
  
   \draw[red] (k)--(l) node[sloped,midway] {$>$};
 \node at (-4,0) [left]{$b_3$};  
  
 \draw [blue]   (l)--(m) node[sloped,midway] {$<$};
   \node at (-4.2,1.7) [above]{$a_3$}; 
   
     \draw[red] (m)--(a) node[sloped,midway] {$<$};
   \node at (-2,3.8) [left]{$b_3$};   
   
  \end{tikzpicture}
 \caption{Representa\c c\~ao planar do toro de g\^enero $g$, para $g=2,3$. 
Ao fim de ter uma figura mais leve, n\~ao 
 desenhamos aqui a triangula\c c\~ao tal que cada aresta  $a_i$ e $b_i$ 
 seja dividida em ao menos tr\^es arestas (veja observa\c c\~ao  \ref{remarque}). }\label{toreg}
\end{figure}
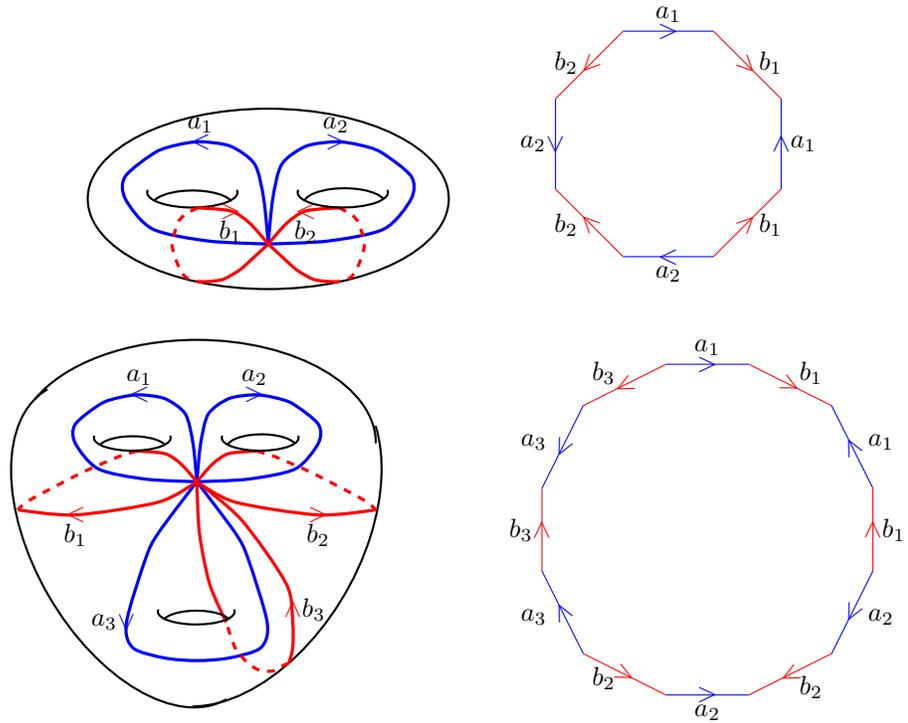

\subsection{Caso do plano projetivo.}

O plano projetivo \'e representado por uma esfera cujos pontos diametralmente opostos est\~ao 
identificados. Uma triangula\c c\~ao do plano projetivo \'e dada por uma triangula\c c\~ao da esfera 
sim\'etrica relativamente ao centro da esfera.  Consideramos a esfera em $\R^3$ (veja
figura \ref{figura17}) e seja ${\mathcal T}$ uma tal triangula\c c\~ao do plano projetivo. 

A interse\c c\~ao de ${\mathcal T}$ com o equador define uma triangula\c c\~ao $J$ do equador, sim\'etrica relativamente ao centro da esfera. Definimos uma sub-triangula\c c\~ao ${\mathcal T}'$ de ${\mathcal T}$ 
tal que simplexos de $J$ sejam simplexos de ${\mathcal T}'$ e que ${\mathcal T}'$ seja sim\'etrica
relativamente ao centro da esfera (veja figura \ref{projo}). 
Pelo lema \ref{subtriang}, a quantidade $n_0 - n_1 + n_2$ \'e a mesma para ${\mathcal T}$ e ${\mathcal T}'$.

\begin{figure}[H]
\begin{tikzpicture}[scale=0.3]

%FIRST SPHERE
\draw[thick] (0,0) circle (5);
%equateur 
% \draw [thick](-5,0) arc (180:360: 5  and 2 );
\draw [thick,red](-5,0) arc (180:270: 5  and 2 );
\draw [thick,blue](0,-2) arc (270:360: 5  and 2 );

  %\draw [thick,dashed](5,0) arc (0:180: 5  and 2 );
\draw [thick,dashed,red](5,0) arc (0:90: 5  and 2 );
\draw [thick,dashed,blue](0,2) arc (90:180: 5  and 2 );

\draw [green](1, -4) cos (3.5, -3) sin (3, 0) cos (1, -4);
\draw[dashed, thick,green] (-1, 4) cos (-3.5, 3) sin (-3, 0) cos (-1, 4);

%SECOND SPHERE
\draw[thick] (18,0) circle (5);
%equateur 

% \draw [thick](23,0) arc (180:360: 5  and 2 );
\draw [thick,red](13,0) arc (180:270: 5  and 2 );
\draw [thick,blue](18,-2) arc (270:360: 5  and 2 );

  %\draw [thick,dashed](13,0) arc (0:180: 5  and 2 );
\draw [thick,dashed,red](23,0) arc (0:90: 5  and 2 );
\draw [thick,dashed,blue](18,2) arc (90:180: 5  and 2 );

\draw [green](19, -4) cos (21.5, -3) sin (21, 0) cos (19, -4);
\draw[dashed, thick,green] (17, 4) cos (14.5, 3) sin (15, 0) cos (17, 4);

%\node at (19.7,-1.95) {$\bullet$};
\draw[thick, green]  (19.7,-1.95) sin (21.5, -3);
\draw[thick, green, dashed]  (16.3,1.95) sin (14.5, 3);

\end{tikzpicture} 
\caption{Triangula\c c\~ao ${\mathcal T}$ do plano projetivo \quad  
Sub-triangula\c c\~ao ${\mathcal T}'$ de ${\mathcal T}$ }\label{projo}
\end{figure}
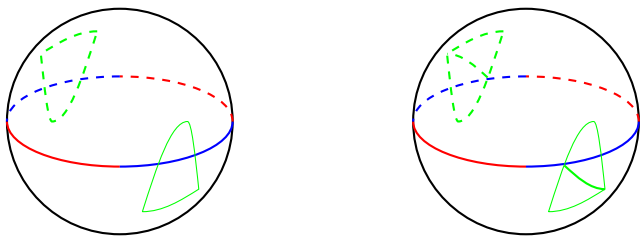

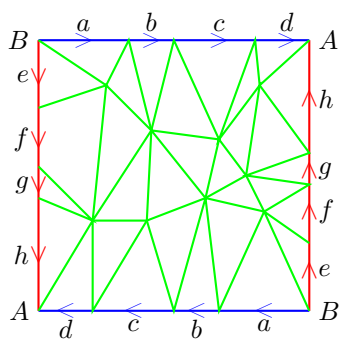
\begin{figure}[H]
\begin{tikzpicture}[scale=0.6]

%Rectangulo
 \draw [blue, thick]  (0, 0) -- (6,0) ;
 \draw [red, thick]  (6,0) -- (6,6) ; 
 \draw [red, thick] (0,6) -- (0,0); 
\draw [blue, thick]  (6,6) --  (0,6) ;
 
\coordinate  (0) at (0,0);
\coordinate  (25) at (0,4.5);
\coordinate  (26) at (0,3.2);
\coordinate  (27) at (0,2.5);
\coordinate  (13) at (0,6);

\coordinate  (9) at (6,0);
\coordinate  (8) at (6,1.5);
\coordinate  (19) at (6,2.8);
\coordinate  (20) at (6,3.5);
\coordinate  (21) at (6,6);

\coordinate  (28) at (4,0);
\coordinate  (4) at (3,0);
\coordinate  (29) at (1.2,0);

\coordinate  (22) at (2,6);
\coordinate  (23) at (3,6);
\coordinate  (24) at (4.8,6);

\coordinate  (2) at (1.2,2);
\coordinate  (3) at (2.4,2);
\coordinate  (5) at (3.7,2.5);
\coordinate  (7) at (5,2.2);

\coordinate  (14) at (1.5,5);
\coordinate  (15) at (2.5,4);
\coordinate  (16) at (4,3.8);
\coordinate  (17) at (4.6,3);
\coordinate  (18) at (4.9,5);

\draw [green, thick]  (0) -- (2) -- (29) -- (3) -- (2);
\draw [green, thick]  (5) -- (4) -- (3) -- (5) -- (28) -- (7) -- (5);
\draw [green, thick] (9) -- (7) -- (8);
\draw [green, thick] (27) -- (2) -- (26);
\draw [green, thick] (25) -- (14) -- (13);
\draw [green, thick] (14) -- (22) -- (15) -- (14) -- (2) -- (15) -- (3);
\draw [green, thick] (21) -- (18) -- (24) -- (16) -- (18) -- (17) -- (16) -- (23) -- (15) -- (16) -- (5) -- (15);
\draw [green, thick] (18) -- (20) -- (17) -- (19) -- (7) -- (17) -- (5);

%setas
\node at (5,0) [blue] {<};
\node at (5,0)[below]{$a$};
\node at (3.5,0) [blue] {<};
\node at (3.5,0)[below]{$b$};
\node at (2.1,0)[blue] {<};
\node at (2.1,0)[below]{$c$};
\node at (0.6,0)[blue] {<};
\node at (0.6,0)[below]{$d$};

\node at (1,6)[blue] {>};
\node at (1,6)[above]{$a$};
\node at (2.5,6)[blue] {>};
\node at (2.5,6)[above]{$b$};
\node at (4,6)[blue] {>};
\node at (4,6)[above]{$c$};
\node at (5.5,6)[blue] {>};
\node at (5.5,6)[above]{$d$};

\node at (0, 5.2)[red]{$\vee$};
\node at (0,5.2)[left]{$e$};
\node at (0, 3.8)[red]{$\vee$};
\node at (0,3.8)[left]{$f$};
\node at (0, 2.8)[red]{$\vee$};
\node at (0,2.8)[left]{$g$};
\node at (0, 1.25)[red]{$\vee$};
\node at (0,1.25)[left]{$h$};

\node at (6, 0.9)[red]{$\wedge$};
\node at (6,0.9)[right]{$e$};
\node at (6, 2.2)[red]{$\wedge$};
\node at (6,2.2)[right]{$f$};
\node at (6, 3.1)[red]{$\wedge$};
\node at (6,3.1)[right]{$g$};
\node at (6, 4.7)[red]{$\wedge$};
\node at (6,4.7)[right]{$h$};

\node at (0, 0)[left]{$A$};
\node at (6, 0)[right]{$B$};
\node at (0, 6)[left]{$B$};
\node at (6, 6)[right]{$A$};

\end{tikzpicture}

\caption{Uma representa\c c\~ao planar $K$ do plano projetivo.}  \label{Projplanar2} 
\end{figure}

Agora, a proje\c c\~ao ortogonal 
da semiesfera norte sobre o plano $0xy$ fornece uma triangula\c c\~ao $K$ do disco $D$ 
de raio $1$, centrado na origem, cuja triangula\c c\~ao da borda \'e sim\'etrica
relativamente  \`a origem. Com a identifica\c c\~ao dos simplexos  na  borda, a 
quantidade $n_0 - n_1$ vale $0$ na borda. Ent\~ao  pelo teorema \ref{teo1}, temos
$$n_0^{\mathcal T} - n_1^{\mathcal T} + n_2^{\mathcal T} = +1$$ 
 para qualquer triangula\c c\~ao do plano projetivo. 

\subsection{Caso da garrafa de Klein.}

O caso da garrafa de Klein \'e an\'alogo ao caso do toro. Sendo uma triangula\c c\~ao 
${\mathcal T}$ da garrafa de 
Klein, escolhemos um meridiano $M$ e um paralelo $P$ e definimos uma sub-triangula\c c\~ao 
${\mathcal T}'$ de ${\mathcal T}$ compat\'ivel com $M$ e $P$. 
O corte ao longo de $M$ e $P$ permite obter uma representa\c c\~ao planar $K$ 
da garrafa de Klein como um ret\^angulo triangulado com identifica\c c\~oes 
na borda. Na borda, temos 
$n_1  = n_0  +1$. Ent\~ao, pelo teorema \ref{teo1}, temos
$$n_0^{\mathcal T} - n_1^{\mathcal T} + n_2^{\mathcal T} = 0$$ 
 para qualquer triangula\c c\~ao da garrafa de Klein. 

\begin{figure}[H]
 \includegraphics[height=2.5cm]{Klein.jpg} \qquad  \qquad 
  \includegraphics[height=2.5cm]{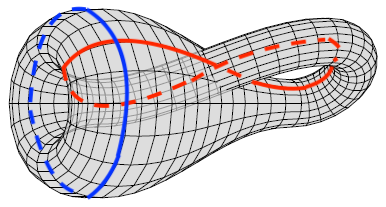} 
\caption{A garrafa de Klein e os cortes.}\label{garrafaKlein}
\end{figure}
 
\begin{figure}[H]
\begin{tikzpicture}[scale=0.6]

%Rectangulo
 \draw [red, thick]  (0, 0) -- (6,0) ;
 \draw [blue, thick]  (6,0) -- (6,6) ; 
 \draw [blue, thick] (0,6) -- (0,0); 
\draw [red, thick]  (6,6) --  (0,6) ;
 
\coordinate  (0) at (0,0);
\coordinate  (25) at (0,4.5);
\coordinate  (26) at (0,3.2);
\coordinate  (27) at (0,2.5);
\coordinate  (13) at (0,6);

\coordinate  (9) at (6,0);
\coordinate  (8) at (6,1.5);
\coordinate  (19) at (6,2.8);
\coordinate  (20) at (6,3.5);
\coordinate  (21) at (6,6);

\coordinate  (1) at (2,0);
\coordinate  (4) at (3,0);
\coordinate  (6) at (4.8,0);

\coordinate  (22) at (2,6);
\coordinate  (23) at (3,6);
\coordinate  (24) at (4.8,6);

\coordinate  (2) at (1.2,2);
\coordinate  (3) at (2.4,2);
\coordinate  (5) at (3.7,2.5);
\coordinate  (7) at (5,2.2);

\coordinate  (14) at (1.5,5);
\coordinate  (15) at (2.5,4);
\coordinate  (16) at (4,3.8);
\coordinate  (17) at (4.6,3);
\coordinate  (18) at (4.9,5);

\draw [green, thick]  (0) -- (2) -- (1) -- (3) -- (2);
\draw [green, thick]  (5) -- (4) -- (3) -- (5) -- (6) -- (7) -- (5);
\draw [green, thick] (9) -- (7) -- (8);
\draw [green, thick] (27) -- (2) -- (26);
\draw [green, thick] (25) -- (14) -- (13);
\draw [green, thick] (14) -- (22) -- (15) -- (14) -- (2) -- (15) -- (3);
\draw [green, thick] (21) -- (18) -- (24) -- (16) -- (18) -- (17) -- (16) -- (23) -- (15) -- (16) -- (5) -- (15);
\draw [green, thick] (18) -- (20) -- (17) -- (19) -- (7) -- (17) -- (5);

%setas
\node at (1,0) [red] {>};
\node at (1,0) [below]{$a$};
\node at (2.5,0) [red] {>};
\node at (2.5,0)[below]{$b$};
\node at (4,0)[red] {>};
\node at (4,0)[below]{$c$};
\node at (5.5,0)[red] {>};
\node at (5.5,0)[below]{$d$};

\node at (1,6)[red] {>};
\node at (1,6)[above]{$a$};
\node at (2.5,6)[red] {>};
\node at (2.5,6)[above]{$b$};
\node at (4,6)[red] {>};
\node at (4,6)[above]{$c$};
\node at (5.5,6)[red] {>};
\node at (5.5,6)[above]{$d$};

\node at (0, 5.2)[blue]{$\vee$};
\node at (0,5.2)[left]{$e$};
\node at (0, 3.8)[blue]{$\vee$};
\node at (0,3.8)[left]{$f$};
\node at (0, 2.8)[blue]{$\vee$};
\node at (0,2.8)[left]{$g$};
\node at (0, 1.25)[blue]{$\vee$};
\node at (0,1.25)[left]{$h$};

\node at (6, 0.9)[blue]{$\wedge$};
\node at (6,0.9)[right]{$e$};
\node at (6, 2.2)[blue]{$\wedge$};
\node at (6,2.2)[right]{$f$};
\node at (6, 3.1)[blue]{$\wedge$};
\node at (6,3.1)[right]{$g$};
\node at (6, 4.7)[blue]{$\wedge$};
\node at (6,4.7)[right]{$h$};

\node at (0, 0)[left]{$A$};
\node at (6, 0)[right]{$A$};
\node at (0, 6)[left]{$A$};
\node at (6, 6)[right]{$A$};

\end{tikzpicture}

\caption{Uma  representa\c c\~ao planar $K$ da garrafa de  Klein.} 
\end{figure}
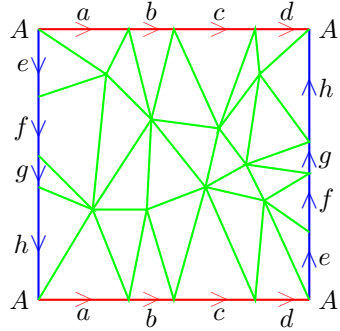

\subsection{Caso do toro pin\c cado.} \label{pinc\'e}

Nem {toda superf\'\i cie} 
com singularidades admitem uma representa\c c\~ao planar. 
O toro pin\c cado \'e um exemplo de superf\'\i cie com singularidades que admite 
tal representa\c c\~ao planar.  
\begin{figure}[H]
\includegraphics[height=2.5cm]{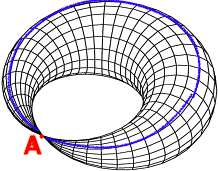} 
\caption{O toro pin\c cado}\label{toropincado}
\end{figure}

O toro pin\c cado \'e uma superf\'\i cie de parametriza\c c\~ao cartesiana em $\R^3$:
\begin{equation*}
\begin{cases}
x = \left( a + b \;cos(v) \, cos\left(\frac{1}{2} u\right) \right) cos(u) \\
y = \left( a + b\; cos(v) \, cos\left(\frac{1}{2} u\right) \right) sin(u) \\
z= b\; sin(v)\,  cos\left(\frac{1}{2} u\right)  
 \end{cases}
\end{equation*}
onde $a$ e $b$ s\~ao os grande e pequeno raios respectivamente. 

Seja ${\mathcal T}$ uma triangula\c c\~ao qualquer do toro pin\c cado. 
Escolhemos um ``paralelo''  $P$ passando pelo ponto singular $A$
do toro pin\c cado. 
Definimos uma sub-triangula\c c\~ao ${\mathcal T}'$ de ${\mathcal T}$ compat\'ivel com $P$. 
O corte ao longo de $P$ permite obter uma representa\c c\~ao planar $K$ 
do toro pin\c cado (figura \ref{pincado})  com identifica\c c\~oes 
na borda $K_0$. Observamos que o ponto $A$ esta duplicado. Na borda, temos 
$n_0^{K_0} - n_1^{K_0} = 0$. Ent\~ao, pelo teorema \ref{teo1}, temos
 $$n_0^{\mathcal T} - n_1^{\mathcal T} + n_2^{\mathcal T} = +1$$
 para qualquer triangula\c c\~ao do toro pin\c cado. 

\begin{figure}[H]
\begin{tikzpicture}[scale=0.8]

%Rectangulo
 \draw [blue, thick]  (1,0) -- (-2,3) -- (1,6) ;
 \draw [blue, thick]  (1,6) -- (4,6) ;
\draw [blue, thick]  (1,0) --  (4,0) ;
 \draw [blue, thick]  (4,0) -- (7,3) -- (4,6) ;
 
 \node at (-1.5,3.5) [blue,rotate=40]  {$>$};
  \node at (-0.5,4.5) [blue,rotate=40]  {$>$};
   \node at (0.5,5.5) [blue,rotate=40]  {$>$};
    \node at (1.5,6) [blue]  {$>$};
     \node at (2.7,6) [blue]  {$>$};
      \node at (3.75,6) [blue]  {$>$};
       \node at (4.8,5.2) [blue,rotate=-40]  {$>$};
        \node at (5.8,4.2) [blue,rotate=-40]  {$>$};
         \node at (6.5,3.5) [blue,rotate=-40]  {$>$};
 \node at (-1.5,2.5) [blue,rotate=-40]  {$>$};
  \node at (-0.5,1.5) [blue,rotate=-40]  {$>$};
   \node at (0.5,0.5) [blue,rotate=-40]  {$>$};
    \node at (1.5,0) [blue]  {$>$};
     \node at (2.7,0) [blue]  {$>$};
      \node at (3.75,0) [blue]  {$>$};
       \node at (4.8,0.8) [blue,rotate=40]  {$>$};
        \node at (5.8,1.8) [blue,rotate=40]  {$>$};
         \node at (6.5,2.5) [blue,rotate=40]  {$>$};
                  
 \node at (-2.1,3.1) [red,above]  {$\bf A$};
 \node at (7.1,3.1) [red,above]  {$\bf A$};
  \node at (-2,3) [red]  {$\bullet$};
 \node at (7,3) [red]  {$\bullet$};

 \coordinate  [label=above:$a$](9) at (-1,4);
 \coordinate [label=above:$b$] (13) at (0,5);
\coordinate  [label=above:$c$](26) at (1,6); 
\coordinate  [label=above:$d$](22) at (2,6);
\coordinate  [label=above:$e$](23) at (3.5,6);
\coordinate [label=above:$f$] (24) at (4,6);
\coordinate [label=above:$g$] (18) at (5.5,4.5);
\coordinate  [label=above:$h$](20) at (6,4);

\coordinate  [label=below:$a$](27) at (-1,2);
\coordinate  [label=below:$b$] (0) at (0,1);
\coordinate [label=below:$c$] (30) at (1,0);
\coordinate  [label=below:$d$](1) at (2,0);
\coordinate [label=below:$e$] (4) at (3.5,0);
\coordinate  [label=below:$f$](6) at (4,0);
\coordinate [label=below:$g$] (8) at (5.5,1.5);
\coordinate  [label=below:$h$](19) at (6,2);

\coordinate  (25) at (0,3);
\coordinate   (21) at (5.5,3);

\coordinate  (2) at (1.2,2);
\coordinate  (3) at (2.4,2);
\coordinate  (5) at (3.7,2.5);
\coordinate  (7) at (5,2.2);

\coordinate (14) at (1.5,5);
\coordinate (15) at (2.5,4);
\coordinate (16) at (4,3.8);
\coordinate  (17) at (4.6,3);

\draw [green, thick]  (-2,3) -- (25) ;
\draw [green, thick]  (27) -- (25) -- (9);
\draw [green, thick]  (27) -- (25) -- (9);
\draw [green, thick]  (13) -- (25);
\draw [green, thick]  (21) -- (7,3);
\draw [green, thick] (20) -- (21) -- (17);
\draw [green, thick] (19) -- (21) -- (7);
\draw [green, thick] (6) -- (7) -- (5);
\draw [green, thick]  (2) -- (30);
\draw [green, thick]  (14) -- (26);
\draw [green, thick]  (0) -- (2) -- (1) -- (3) -- (2);
\draw [green, thick]  (5) -- (4) -- (3) -- (5) -- (6); 
\draw [green, thick] (7) -- (8);
\draw [green, thick] (27) -- (2) -- (25);
\draw [green, thick] (25) -- (14) -- (13);
\draw [green, thick] (14) -- (22) -- (15) -- (14) -- (2) -- (15) -- (3);
\draw [green, thick] (24) -- (16) -- (18) -- (17) -- (16) -- (23) -- (15) -- (16) -- (5) -- (15);
\draw [green, thick] (20) -- (17);
\draw [green, thick] (19) -- (7) -- (17) -- (5);

\end{tikzpicture}
\caption{Uma  representa\c c\~ao planar $K$ do toro pin\c cado. } \label{pincado}
\end{figure}
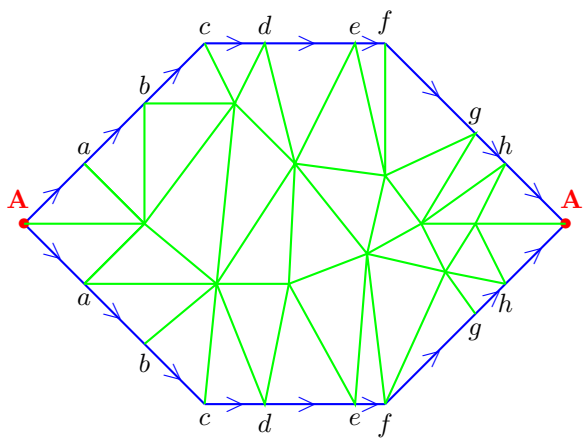

%\pagebreak

\end{document}